\documentclass[10pt]{article}
\pdfoutput=1

\usepackage{amssymb}
\usepackage{amsmath}
\usepackage{mathtools}
\usepackage{algpseudocode}
\usepackage{graphics,epsfig}
\DeclareGraphicsExtensions{.pdf}

\usepackage{url,here}

\usepackage{algorithm}
\allowdisplaybreaks

\usepackage{multirow}

\begin{document}

\newenvironment {proof}{{\noindent\bf Proof.}}{\hfill $\Box$ \medskip}

\newtheorem{theorem}{Theorem}[section]
\newtheorem{lemma}[theorem]{Lemma}
\newtheorem{condition}[theorem]{Condition}
\newtheorem{proposition}[theorem]{Proposition}
\newtheorem{remark}[theorem]{Remark}
\newtheorem{definition}[theorem]{Definition}
\newtheorem{hypothesis}[theorem]{Hypothesis}
\newtheorem{corollary}[theorem]{Corollary}
\newtheorem{example}[theorem]{Example}
\newtheorem{descript}[theorem]{Description}
\newtheorem{assumption}[theorem]{Assumption}

\def\P{\mathbb{P}}
\def\R{\mathbb{R}}
\def\E{\mathbb{E}}
\def\N{\mathbb{N}}
\def\Z{\mathbb{Z}}

\renewcommand {\theequation}{\arabic{section}.\arabic{equation}}
\def \non{{\nonumber}}
\def \hat{\widehat}
\def \tilde{\widetilde}
\def \bar{\overline}

\def\ind{{\mathchoice {\rm 1\mskip-4mu l} {\rm 1\mskip-4mu l}
{\rm 1\mskip-4.5mu l} {\rm 1\mskip-5mu l}}}
\algnewcommand{\algorithmicgoto}{\textbf{go to step}}%
\algnewcommand{\Goto}[1]{\algorithmicgoto~\ref{#1}}%

\title{\Large\ { \bf Computational identification of irreducible state-spaces for stochastic reaction networks}}

\author{Ankit Gupta and Mustafa Khammash \\
Department of Biosystems Science and Engineering \\ ETH Zurich \\  Mattenstrasse 26 \\ 4058 Basel, Switzerland. 
}
\date{}

\maketitle
\begin{abstract}
Stochastic models of reaction networks are becoming increasingly important in Systems Biology. In these models, the dynamics is generally represented by a continuous-time Markov chain whose states denote the copy-numbers of the constituent species. The state-space on which this process resides is a subset of non-negative integer lattice and for many examples of interest, this state-space is countably infinite. This causes numerous problems in analyzing the Markov chain and understanding its long-term behavior. These problems are further confounded by the presence of conservation relations among species which constrain the dynamics in complicated ways. In this paper we provide a linear-algebraic procedure to disentangle these conservation relations and represent the state-space in a special decomposed form, based on the copy-number ranges of various species and dependencies among them. This decomposed form is advantageous for analyzing the stochastic model and for a large class of networks we demonstrate how this form can be used for finding all the closed communication classes for the Markov chain within the infinite state-space. Such communication classes are irreducible state-spaces for the dynamics and they support all the extremal stationary distributions for the Markov chain. Hence our results provide important insights into the long-term behavior and stability properties of stochastic models of reaction networks. We discuss how the knowledge of these irreducible state-spaces can be used in many ways such as speeding-up stochastic simulations of multiscale networks or in identifying the stationary distributions of complex-balanced networks. We illustrate our results with several examples of gene-expression networks from Systems Biology.
\end{abstract}

\medskip

\setcounter{equation}{0}

\section{Introduction} \label{sec:intro}

Many biological processes are described as reaction networks, where finitely many species interact with each other through some fixed reaction channels \cite{Hethcote,Berger,Alon,CellSignaling,Vilar,Gardner}. Traditionally, reaction networks have been mathematically studied by expressing the dynamics as a set of ordinary differential equations (ODEs). However it is now well-known that these deterministic formulations become highly inaccurate when the copy-numbers of the reacting species are \emph{small}. This is because the timing of reactions becomes random, introducing \emph{noise} into the dynamics, which can significantly change the behavior of the system being modeled \cite{Arkin}. Such situations arise commonly in Systems Biology, since intracellular networks often involve species with low copy-numbers like gene-transcripts, signaling proteins, messenger RNAs, transcription factors etc. \cite{Elowitz,Arkin}. The biochemical noise generated by the intermittency of reactions can be taken into account using stochastic models of reaction networks. A common approach is to represent the dynamics as a continuous-time Markov chain (CTMC) whose states denote the copy-numbers of the constituent species. For recent surveys on these stochastic models and the methods available for analyzing them see \cite{Goutsias,DASurvey,schnoerr2017approximation}. Over the past few years, such models have been extensively used for understanding the role of noise in various biological mechanisms \cite{Rao,Elowitz}.

We now formally describe such a stochastic model for a reaction network. Throughout this paper $\R$, $\R_+$, $\Z$, $\N$ and $\N_{0}$ denote the sets of all reals, nonnegative reals, integers, positive integers and nonnegative integers respectively. Consider a network with $d$ species $\mathbf{S}_1,\dots,\mathbf{S}_d$ which interact through $K$ reaction channels of the form
\begin{align}
\label{formofthereaction}
\sum_{i=1}^d \nu_{ik} \mathbf{S}_i \longrightarrow \sum_{i=1}^d \rho_{ik} \mathbf{S}_i, \qquad k=1,\dots,K,
\end{align} 
where $\nu_{ik}$ and $ \rho_{ik}$ belong to $\N_0$ and they denote the number of molecules of $\mathbf{S}_i$ that are consumed and produced by reaction $k$. Under the classical well-mixed assumption \cite{GP}, the network's state at any time is described by a vector $x = (x_1,\dots,x_d) \in \N^d_0$ of copy-number counts of all the species, i.e. $x_i$ denotes the copy-number of species $\mathbf{S}_i$. The net change in the state due to reaction $k$ is simply given by the \emph{stoichiometry vector} $\zeta_k \in \Z^d$ whose $i$-th component is $\zeta_{ik} = ( \rho_{ik}  - \nu_{ik}) \in \Z$. The usual CTMC model for the reaction dynamics \cite{DASurvey} stipulates that the transition rate from state $x$ to state $(x +\zeta_k)$ is $\lambda_k(x)$ for each $k=1,\dots,K$, where $\lambda_1,\dots, \lambda_K : \N^d_0  \to \R_+$ are called the \emph{propensity} functions for the network and they are assumed to satisfy:  
\begin{align}
\label{statespacecond}
\textnormal{ for any} \quad x \in \N^d_0 \quad \textnormal{and} \quad k =1,\dots,K  \quad \textnormal{if} \quad \lambda_k(x) > 0 \quad \textnormal{then} \quad (x+ \zeta_k) \in \N^d_0.  
\end{align}
This condition ensures that the set $\N^d_0$ is a valid \emph{state-space} for the CTMC, because it is \emph{closed} under the reaction dynamics i.e. if the starting state is in $\N^d_0$ then the dynamics remains in $\N^d_0$ throughout its trajectory.

Under mild conditions on the propensity functions (see Chapter 6 in \cite{EK}), one can ensure that the CTMC $$( X(t) )_{t  \geq 0 } = ( X_1(t),\dots,X_d(t) )_{t \geq 0 }$$ with the above transition structure is well-defined for any initial state $x_0 \in \N^d_0$. Define the probability that the reaction dynamics is at state $y\in   \N^d_0$ at time $t$ by
\begin{align}
\label{transition_probabilities}
p_{x_0}(t,y) = \P\left( X(t) = y \right).
\end{align}
Then the dynamics of $p_{x_0}(t,\cdot)$ is given by the Chemical Master Equation (CME) \cite{GillespieCME} which has the following form:
\begin{align}
\label{defn_cme}
\frac{ d p_{x_0}(t,y) }{dt} = & \sum_{k=1}^K  \left( p_{x_0}(t, y - \zeta_k) \lambda_k(y -  \zeta_k) -p_{x_0}(t, y ) \lambda_k(y) \right), 
\end{align}  
for each $y \in  \N^d_0$ which can be accessed by the CTMC. It can be shown that CME has a unique solution provided that the CTMC is non-explosive (see \cite{DASurvey}). For most examples in Systems Biology, the number of accessible states is either infinite or very large, and hence solving the CME to obtain the probabilities $p_{x_0}(t,y)$ is nearly impossible. One generally estimates these values by simulating the CTMC $(X(t))_{t\geq0}$ using Monte Carlo methods such as Gillespie's \emph{Stochastic Simulation Algorithm} (SSA) \cite{GP}. Another popular approach for obtaining approximate solutions is the \emph{Finite State Projection} (FSP) method which efficiently truncates the state-space to a small finite set and then solves the CME over this finite set \cite{FSP}. Both SSA and FSP based approaches work well for smaller networks and over finite time-intervals, and hence they do not help in satisfactorily assessing the long-term behavior and stability properties of the stochastic model. For Markov chains over a finite state-space, one way to assess this long-term behavior is by computing the disjoint \emph{closed communication classes} in the state-space (see Section \ref{subsec:reachability}), using matrix or graphical methods \cite{Kemeny}. However there do not exist methods for systematically finding all the closed communication classes for stochastic reaction network models with infinite state-spaces. The main goal of this paper is to develop such a method that can \emph{provably} find all such classes under biologically reasonable assumptions on the network. This has important implications regarding the long-term behavior of the stochastic model because these closed communication classes support the stationary distributions (see Section \ref{subsec:appl}), which are like attracting fixed-points for the CME \eqref{defn_cme} in the space of probability distributions. Moreover each closed communication class serves as an \emph{irreducible state-space} for the underlying CTMC (see Section \ref{subsec:reachability}). For this reason, we refer to a closed communication class as an irreducible state-space in this paper.

For many networks, the state-space $\N^d_0$ is simply \emph{too large}, in the sense that it contains several extraneous states that are never visited by the dynamics. This is mainly due to the \emph{conservation relations} present in the network which impose constraints on the copy-number ranges of the involved species. Hence our first task is to develop a systematic procedure to \emph{weed-out} these extraneous states and obtain a \emph{smaller} non-empty set $\mathcal{E}_0$ which is also a valid state-space (i.e.\ \eqref{statespacecond} holds with $\N^d_0$ replaced with $\mathcal{E}_0$). Furthermore we would like the representation of $\mathcal{E}_0$ to be explicit enough to ensure that the copy-number ranges of various species can be easily identified. Apart from enabling the search for irreducible state-spaces, which is the primary goal of this paper, this explicit state-space representation allows us to gain a better understanding of the network dynamics. We now explore these issues in a greater detail.

Let us return to the CTMC $( X(t) )_{t \geq 0 } $ starting at some initial state $x_0 \in \N^d_0$. The ideal or the smallest state-space $\mathcal{E}$ for this process would simply be the set of all states in $\N^d_0$ that the reaction dynamics has a positive probability of reaching in a finite time\footnote{The choice of this finite time is not important because for a CTMC if $p_{x_0}(t,y) >0$ for some $t > 0$ then the same holds for all $t > 0$ (see Theorem 3.2.1 in \cite{Norris})}, i.e.
\begin{align}
\label{ideal_state_space}
\mathcal{E} = \{ y \in \N^d_0 :  p_{x_0}(t,y) > 0 \},
\end{align}
for some $t > 0$. The set $\mathcal{E}$ is non-empty (as $x_0 \in \mathcal{E}$) and using Chapman-Kolmorogov inequalities (see \cite{Norris}) one can easily check that $\mathcal{E}$ is a valid state-space and so the CTMC $(X(t))_{t \geq 0 }$ resides in this set throughout its trajectory. However the set $\mathcal{E} $ is difficult to characterize because the probabilities $p_{x_0}(t,\cdot)$ are usually unknown. Hence we look for a bigger set $\mathcal{E}_0$ which contains $\mathcal{E} $ and is also a valid state-space. The standard choice is the \emph{stoichiometry compatibility class} (see \cite{DASurvey}) defined by
\begin{align}
\label{stoichcompatibility}
\mathcal{E}_0 = ( x_0 + \textnormal{Range}(S)) \cap \N^d_0,
\end{align}
where $S = \textnormal{Col}(\zeta_1,\dots,\zeta_K)$ is the $d \times K$ \emph{stoichiometry matrix} whose columns are the reaction stoichiometry vectors and $\textnormal{Range}(S)$ is the range or column space of $S$. To see the containment $\mathcal{E} \subset \mathcal{E}_0$ note that if $y \in \mathcal{E}$ then $y$ must be reachable from $x_0$ in finitely many transition steps in each of the directions $\zeta_1,\dots,\zeta_K$. Let $r_k$ be the number of steps needed in direction $k$, then setting $r = (r_1,\dots, r_K) \in \N^K_0$ we must have $y = x_0 + S r$ which ensures that $(y - x_0) \in \textnormal{Range}(S)$ and hence $y \in \mathcal{E}_0$. We next write $\mathcal{E}_0$ in terms of conservation relations for the network, which are the nonzero vectors in the \emph{left nullspace} of the stoichiometry matrix $S$
\begin{align}
\label{defn_consrels}
\mathcal{L}(S) = \{ \gamma \in \R^d : \gamma^T S  = \mathbf{0}^T\},
\end{align}
where $\mathbf{0}$ is the vector of all zeros in $\R^K$. For any nonzero $\gamma = (\gamma_1,\dots,\gamma_d)  \in \mathcal{L}(S)$, the relation $\gamma^T S  = \mathbf{0}^T$ implies that for each reaction $k=1,\dots,K$, the stoichiometry vector $\zeta_k$ is orthogonal to $\gamma$, i.e. $\langle \gamma, \zeta_k \rangle =0$, where $\langle \cdot, \cdot \rangle$ denotes the standard inner product in $\R^d$. Therefore the stochastic reaction dynamics $(X(t))_{t \geq 0}$ will satisfy
\begin{align}
\label{conservationimpl1}
\langle \gamma, X(t) \rangle = \langle \gamma, X(0) \rangle =\langle \gamma, x_0 \rangle  \qquad \textnormal{for all } t \geq 0,
\end{align}
which means that the copy-numbers of species included in the support set 
\begin{align}
\label{suppsetgamma}
\textnormal{supp}(\gamma) :=\{ i=1,\dots,d:  \gamma_i  \neq 0 \}
\end{align}
conserve the linear constraint specified by vector $\gamma$. Suppose that for some integer $n>0$ we have $\textnormal{Rank}(S) = (d - n)$ and so the dimension of the subspace $\mathcal{L}(S) $ is $n$. Let $\{  \gamma_1,\dots,\gamma_n \}$ be a \emph{basis} for $\mathcal{L}(S)$ and define a $d \times n$ matrix $\Gamma$ and a $n \times 1$ vector $c$ by
\begin{align}
\label{defn_consdata}
\Gamma = \textnormal{Col}( \gamma_1,\dots,\gamma_n) \qquad \textnormal{and} \qquad c = \Gamma^T x_0.
\end{align}
From \eqref{conservationimpl1} and the orthogonality of vector spaces $\mathcal{L}(S)$ and $\textnormal{Range}(S)$, we can equivalently express the set $\mathcal{E}_0$ as   
\begin{align}
\label{defn_statespacewithconser}
\mathcal{E}_0 = \{ x \in \N^d_0 :   \Gamma^T x = c\}.
\end{align}
Taking a cue from the terminology used in \cite{Schuster}, we call $\Gamma$ as the \emph{conservation matrix}, $c$ as the \emph{conservation vector} and the pair $(\Gamma,c)$ as the \emph{conservation data} for the network. If the subspace $\mathcal{L}(S)$ is trivial (i.e. $n=0$) then $\mathcal{E}_0 = \N^d_0$, but in the other situation when this subspace is nontrivial (i.e. $n>0$), the states in $\N^d_0$ outside $\mathcal{E}_0$ are not reachable by the dynamics because $\mathcal{E}_0$ is a valid state-space for the CTMC $ ( X(t) )_{t \geq 0 }$ and it contains the initial state $X(0) =x_0$.

The representations \eqref{stoichcompatibility} or \eqref{defn_statespacewithconser} for $\mathcal{E}_0$ are quite abstract, making it difficult to get a sense of the dynamics and particularly the copy-number ranges of all the species. We remedy this problem by developing a method that systematically screens the space of conservation relations $\mathcal{L}(S)$ and expresses $\mathcal{E}_0 $ in a more explicit form. We briefly describe the steps needed for this purpose. We call a nonzero vector $\gamma \in \mathcal{L}(S)$ a \emph{semi-positive} conservation relation if all its nonzero entries are positive. Up to independence, the only other type of conservation relations are the \emph{mixed-sign} ones which have at least two nonzero entries with opposite signs. Semi-positive conservation relations ensure that the species in their support-sets (see \eqref{suppsetgamma}) are \emph{bounded} in the sense that their copy-numbers have a bounded range. This is evident from \eqref{conservationimpl1} because if $\gamma$ is a semi-positive conservation relation then for any $i \in \textnormal{supp}(\gamma) $, we have $\gamma_i > 0$ and since $\langle \gamma, X(t) \rangle = \sum_{j \in  \textnormal{supp}(\gamma)} \gamma_j X_j(t)$, we must also have $X_i(t) \in [0, \langle \gamma, x_0 \rangle/\gamma_i]$. We scan the space of semi-positive conservation relations and identify all the \emph{bounded} species along with their suitable copy-number ranges. It is possible that these semi-positive relations do not span the whole space of conservation relations $\mathcal{L}(S)$, and so we then look for any mixed-sign conservation relations among the remaining \emph{unbounded} species. We show that such conservation relations force a certain subset of \emph{unbounded} species (called \emph{restricted} species) to mimic the dynamics of the remaining \emph{unbounded} species (called \emph{free} species) according to an appropriately constructed affine map. For example, the following ATP-hydrolysis reaction occurs in many living cells:
\begin{align*}
\mathbf{ATP}(\textnormal{\small adenosine triphosphate}) + \mathbf{H_2 O}(\textnormal{\small water}) \xrightleftharpoons[]{}  \mathbf{ADP}(\textnormal{\small adenosine diphosphate}) + \mathbf{P}(\textnormal{\small phosphate}) .
\end{align*}
Generally $\mathbf{ATP}$ and $\mathbf{H_2 O}$ molecules are present in high concentrations in the cytosol, and so the dynamics of low copy-number species $\mathbf{S}_1 = \mathbf{ADP}$ and $\mathbf{S}_2 = \mathbf{P}$ can be well approximated by the simple network
\begin{align}
\label{simpl:atphydrolysis}
\emptyset \longrightarrow \mathbf{S}_1 + \mathbf{S}_2 \longrightarrow \emptyset.
\end{align}
The stoichiometry matrix for this network is
\begin{align*}
S = \left[ 
\begin{array}{cc}
1 & -1\\
1 & -1
\end{array}
\right]
\end{align*}
and the space $\mathcal{L}(S)$ of conservation relations has dimension $1$. One can see that there are no semi-positive conservation relations and the only independent conservation relation is $\gamma = (1,-1)$ which is mixed-sign. Moreover \eqref{conservationimpl1} forces $X_1(t) = \phi( X_2(t) )$ for all $t \geq 0$, where $\phi$ is the affine map given by $\phi(x) = x + \langle \gamma,x_0 \rangle$.

In large networks there are several conservation relations and each species can participate in many such relations. To account for all the possibilities, we will employ standard linear-algebraic methods, such as basic matrix manipulations, solving linear-algebraic systems and Linear Programs (LPs) \cite{Boyd}, to classify each species as one of three types: \emph{free}, \emph{bounded} or \emph{restricted}. Under fairly general conditions satisfied by most biological networks, we prove that by relabeling the species, we can express the state-space $\mathcal{E}_0$ in a special decomposed form
\begin{align}
\label{defn_dss}
\mathcal{E}_0 = \mathcal{E}_b \times \Phi,
\end{align}
where $\mathcal{E}_b$ is a finite set in $\N^{d_b}_0$ and $\Phi$ is the \emph{graph} of an \emph{affine} function $\phi : \R^{d_f}  \to \R^{d_r}_0$ restricted to the nonnegative integer orthant, i.e.
\begin{align}
\label{defngraphphi}
\Phi = \left\{ (x, \phi(x)) :  x \in \N^{d_f}_0 \ \  \textnormal{and} \  \  \phi(x) \in \N^{d_r}_0  \right\}.
\end{align}
Here $d_f,d_b$ and $d_r$ are nonnegative integers denoting the number of  \emph{free}, \emph{bounded} and \emph{restricted} species respectively. The finite set $ \mathcal{E}_b \subset  \N^{d_b}_0$ contains the dynamics of \emph{bounded} species, while the infinite set $\Phi \subset \N^{d_f +d_r}_0 $ serves as a state-space for the dynamics of both \emph{free} and \emph{restricted} species, with the latter being ``locked" in a fixed affine relationship (given by function $\phi$) with the former. Notice that in comparison to both \eqref{stoichcompatibility} or \eqref{defn_statespacewithconser}, the form \eqref{defn_dss} for state-space $\mathcal{E}_0$ is \emph{more explicit} as it clearly expresses the copy-number ranges for each species as well as the relationships among them. This enables a better understanding of the dynamics which can be leveraged to improve the efficiency of existing analytical methods. For instance, one can use the form \eqref{defn_dss} to design optimal state-space truncations for the \emph{Finite State Projection} (FSP) method for solving CMEs \cite{FSP}. One can also use this form as a guide for automatically separating the high-copy-number and the low-copy-number species for the \emph{method of conditional moments} (MCM) approach \cite{Theis} or for deriving suitable hybrid approximations of the dynamics \cite{Benni}. We do not explore these ideas in this paper but instead focus on using this form for the analysis of the long-term behavior of the underlying Markov process.

Suppose there exists a conservation relation in $\mathcal{L}(S)$ such that all its $d$ components are strictly positive. Such a situation generally occurs when the network satisfies some form of global \emph{mass conservation relation} and in this case all the species are \emph{bounded} (i.e. $d_b =d$, $d_f = d_r = 0$), and so the state-space $\mathcal{E}_0 = \mathcal{E}_b$ is finite. Using elementary matrix or graph-theoretic methods \cite{Kemeny} we can easily find all the closed communication classes or the irreducible state-spaces (see Section \ref{subsec:reachability}) $\mathcal{E}_1,\dots,\mathcal{E}_Q$ within $\mathcal{E}_0$, where the dynamics eventually lies starting from any initial state $x_0$. These classes are mutually disjoint and each $\mathcal{E}_q$ supports a unique stationary distribution $\pi_q$ for the stochastic reaction dynamics. Moreover these distributions $\pi_1,\dots,\pi_Q$ are the \emph{extremal points} of the simplex formed by all the stationary distributions of the network
\begin{align}
\label{stat_simplex}
\Sigma = \left\{ \pi = \sum_{q=1}^Q \alpha_q \pi_q: \quad \textnormal{each} \quad \alpha_1,\dots,\alpha_Q \geq 0 \quad \textnormal{and} \quad \sum_{q=1}^Q \alpha_q =1 \right\}.
\end{align} 
Typically for reaction networks in Systems Biology, a conservation relation with all strictly positive entries does not exist. This is because such networks are abstract representations of the actual processes where many details, such as the dynamics of abundant species (like ATP molecules, enzymes etc.), are intentionally omitted to make the analysis more tractable and pertinent to a given problem. Consequently, unlike classical chemical kinetics, global mass conservation fails because mass is allowed to be created and destroyed. Hence it is important to consider this general case, where \emph{free} and \emph{restricted} species are present, and the state-space $\mathcal{E}_0$ must be necessarily infinite provided certain natural assumptions are satisfied (see Assumption \ref{massacttype}). In such a scenario the stationary distributions do not always exist, but if the existence can be guaranteed on each irreducible state-space (see Section \ref{subsec:reachability}), then again the extremal points $\pi_1, \pi_2,\dots$ of the simplex of stationary distributions $\Sigma$ are the unique stationary distributions supported on all the irreducible state-spaces $\mathcal{E}_1,\mathcal{E}_2,\dots,$ within the infinite state-space $\mathcal{E}_0$. The primary goal of this paper is to develop a method that explicitly identifies all these irreducible state-spaces using the decomposed state-space form \eqref{defn_dss}. For this purpose we adapt and generalize the ideas contained in \cite{GuptaIFAC}, with the main observation being that for most biological networks, the \emph{free} species can be organized in the form of \emph{birth and death cascades}, depending on the minimum number of reactions that the species requires to be created from nothing (denoted by $\emptyset$) or get reduced to it. As our examples suggest, for many Systems Biology networks these cascades have a natural correspondence to gene expression or signaling stages in the dynamics (see Section \ref{sec:examples}). Combining this cascade construction along with matrix methods that are used in the finite state-space case, our method provably determines all the irreducible state-spaces for a large class of networks satisfying some biologically reasonable criteria. The knowledge of irreducible state-spaces has several important applications which we discuss in Section \ref{subsec:appl}.

In comparison to the method in \cite{GuptaIFAC}, the method we develop in this paper is far more versatile and can analyze a much larger class of networks arising in Systems and Synthetic Biology. The method in \cite{GuptaIFAC} has several shortcomings that severely restrict its range of applicability. Most prominently, it only applies to networks where there are no mixed-sign conservation relations (i.e. no \emph{restricted} species) and all the \emph{free} species can be arranged in birth-cascades. These requirements are often violated by important networks (see Section \ref{sec:examples}) and this motivated us to devise the method presented in this paper, which can handle all types of conservation relations and intricate interactions among the dynamics of \emph{bounded}, \emph{free} and \emph{restricted} species. The main novelty of the approach we develop in this paper is in the elaborate construction of \emph{tree-like} structures to define reachability relationships between birth and death cascades (see Section \ref{sec:freebddspecies}). It is shown that the \emph{leaves} of these cascade trees contain all the information about irreducible state-spaces.

This paper is organized as follows. In Section \ref{sec:prelim} we introduce some preliminary concepts that will be used throughout the paper. The method for computing the decomposed form of state-space is explained in Section \ref{sec:computingdss} and the process for finding all the irreducible state-spaces is described in Section \ref{sec:irredclasses}. The algorithms for implementing these procedures are provided in Section \ref{sec:algorithms}. In Section \ref{sec:examples} we illustrate the applicability and usefulness of our methods for state-space analysis by considering several examples from Systems Biology. In particular we discuss how this state-space analysis can contribute towards our understanding of the underlying biological process. In Section \ref{sec:conclusion} we conclude and the detailed proof of our main result, Theorem \ref{maintheorem} is given in the Appendix.

\section*{Notation}

Most of the notation used in this paper has already been introduced in Section \ref{sec:intro}. However some additional notation and clarifications are needed which we mention now. 
For any set $A$, we denote its cardinality by $|A|$. The vector of all zeros in any dimension is denoted by ${\bf 0}$. Similarly in any dimension $d$, the $i$-th standard basis is denoted by $e_i$ and it is the vector whose  $i$-th entry is $1$ while the rest are zeros. The identity matrix is denoted by ${ \bf I}$. For any $m \times n$ matrix $M$ and any $k\leq l \leq m$, the $(l-k+1) \times n$ matrix formed by rows $(k+1),(k+2),\dots,l$ of matrix $M$ is denoted by $\textnormal{Proj}(M,k,l)$.  If $v_1,\dots,v_n$ are the columns of $M$ then for any $A \subset \R$, the set $\textnormal{Colspan}_A(M)$ stands for
\begin{align*}
\left\{x \in \R^m :  x = \sum_{i=1}^n a_i v_i \textnormal{ for some } a_1,\dots,a_n \in A\right\}.
\end{align*}
The dimension of any vector space $V$ is denoted by $\textnormal{dim}(V)$ and this vector space is called trivial if $\textnormal{dim}(V)=0$. While multiplying a matrix with a vector we always regard the vector as a column vector. All inequalities involving vectors or matrices must be interpreted componentwise.

\section{Preliminaries} \label{sec:prelim}

Consider the reaction network $\mathcal{N}$ described in Section \ref{sec:intro} with $d$ species and $K$ reactions of the form \eqref{formofthereaction}. Let $\Lambda : \N^d_0 \to \R^K_+$ be the \emph{propensity} map given by
\begin{align}
\label{defn_Lambda}
\Lambda(x) = \left(  \lambda_1(x),\dots,\lambda_K(x) \right).
\end{align}
We represent this network by the triplet $\mathcal{N} = ( \mathcal{V}  , \mathcal{O}  ,\Lambda)$, where $\mathcal{V} $ and $\mathcal{O}$ are two $d \times K$ matrices whose entries at row $i$ and column $k$ are given by $ \nu_{ik}$ and $\rho_{ik}$ respectively. We call these matrices as the \emph{reactant} matrix and the \emph{product} matrix respectively, because they tabulate the number of molecules of each species that are created and removed by each of the reactions. Note that the stoichiometry matrix $S$ for this network is simply $S =(\mathcal{O} -  \mathcal{V})$. The left nullspace $\mathcal{L}(S)$ of this matrix is the space of all conservation relations for the network and as explained in Section \ref{sec:intro}, the constraints imposed by these relations can be described by some conservation data $(\Gamma,c)$, and with this data at hand the designated state-space $\mathcal{E}_0$ is given by \eqref{defn_statespacewithconser}. In the rest of this section we present some concepts and assumptions associated with the reactions networks we consider in this paper.

\subsection{The reachability relation and irreducible state-spaces}\label{subsec:reachability}

 For any $x,y \in \mathcal{E}_0$ let $p_x(t,y)$ (see \eqref{transition_probabilities}) be the probability that the stochastic dynamics starts at $x$ and reaches $y$ at time $t$. If $p_x(t,y) >0$ for some $t\geq 0$, then we say that state $y$ is \emph{reachable} from state $x$, and we denote this relation as
\begin{align}
\label{reln_reachability}
x \stackrel{  \mathcal{N}  }{\longrightarrow} y.
\end{align}
This relation does not depend on the time-value $t$ (see Theorem 3.2.1 in \cite{Norris}) and it is {\bf transitive}: i.e. for any $x,y,z \in \mathcal{E}$, if $x \stackrel{  \mathcal{N}  }{\longrightarrow} y$ and $y \stackrel{  \mathcal{N}  }{\longrightarrow} z$ then $x \stackrel{  \mathcal{N}  }{\longrightarrow} z$ (see Chapter 6 in \cite{Grimmett}). We say that two states $x,y \in  \mathcal{E}_0$ \emph{communicate} if we have both $x \stackrel{  \mathcal{N}  }{\longrightarrow} y$ and $y \stackrel{  \mathcal{N}  }{\longrightarrow} x$, and this relation is denoted by 
\begin{align*}
 x \stackrel{  \mathcal{N}  }{\longleftrightarrow} y.
\end{align*}
It is known that $ \stackrel{  \mathcal{N}  }{\longleftrightarrow}$ is an \emph{equivalence} relation on $\mathcal{E}_0$ and it partitions the set $\mathcal{E}_0$ into disjoint equivalence classes which are referred to as \emph{communication classes} in the Markov chain literature (see \cite{Norris}). The communication classes can be further classified as \emph{closed} or \emph{open}. A communication class $\mathcal{E} \subset \mathcal{E}_0$ is called \emph{closed} if and only if for any $x \in  \mathcal{E}$ and $y \in \mathcal{E}_0$, if $x \stackrel{  \mathcal{N}  }{\longrightarrow} y$ then $y \in \mathcal{E}$. By definition a closed communication class $\mathcal{E}$ needs to be closed under the reaction dynamics (i.e.  \eqref{statespacecond} holds with $\N^d_0$ replaced with $\mathcal{E}$), and so it is a valid state-space for the underlying CTMC. Moreover as all the states in $\mathcal{E}$ are reachable from each other, $\mathcal{E}$ is an \emph{irreducible} state-space for the CTMC. Therefore we refer to such a set $\mathcal{E}$ as an irreducible state-space in the paper.

When the state-space $\mathcal{E}_0$ is finite, then all the irreducible state-spaces correspond exactly to all the positive-recurrent classes for the underlying Markov chain (see \cite{Norris}), and starting from any initial state in $\mathcal{E}_0$, the CTMC will eventually get \emph{trapped} in one of these disjoint irreducible state-spaces. The same holds true for infinite state-spaces if the existence of stationary distributions on all the irreducible state-spaces can be guaranteed. This can be done using methods developed by Meyn and Tweedie \cite{MeynandTweedieBook,Meyn}. In particular, Theorem 4.5 in \cite{Meyn} shows that stationary distributions will exist if one can find a \emph{Foster-Lyapunov} function $V : \mathcal{E}_0 \to \R_+$ satisfying $V(x) \to \infty$ as $\|x\| \to \infty$, such that the $\R_+$-valued process $ ( V( X(t) ) )_{t  \geq 0 }$ experiences a negative drift outside some compact (finite) set $C \subset \mathcal{E}_0$. In a recent paper \cite{GuptaPLOS} we develop a computational framework for constructing such Foster-Lyapunov functions for a large class of biochemical reaction networks, which includes several well-known examples from Systems and Synthetic Biology.

The problem of finding closed communication classes for stochastic reaction networks is quite challenging, as establishing the reachability \eqref{reln_reachability} between any two states $x,y \in \mathcal{E}_0$ is tantamount to  showing that there exists a sequence of $n$ reactions $k_1,\dots,k_n \in \{1,\dots,K\}$ such that $y = x + \sum_{i=1}^n \zeta_{k_i} $ and $ \lambda_{k_j} (z_j) >0$ for each $j=1,\dots,n$, where $ z_j = x + \sum_{i=1}^{j-1} \zeta_{k_i} $. These conditions ensure that starting from state $x$, firing of reactions $k_1,\dots,k_n$ in this order, takes the state to $y$, and this firing of reactions is a positive-probability event because at all the \emph{intermediate} states ($z_j$-s), the propensity $ \lambda_{k_j} (z_j) $ for the next reaction in this sequence is positive. When the state-space $\mathcal{E}_0$ is finite, matrix methods (see \cite{Kemeny}) can be used for computing the closed communication classes, without having to explicitly find any positive-probability reaction sequences between states. In this paper we provide an extension of this approach that can handle a large class of biological networks with infinite state-spaces (see Section \ref{sec:irredclasses}). One of the main difficulty that we have to deal with is that some reaction channels may \emph{switch-off} at certain states, due to their propensities being zero at those states, and hence the set of possible transition directions is not the same for all the states in $\mathcal{E}_0$. To account for this \emph{switching-off} of reactions we need to impose some conditions on the propensity functions, as we mention in Section \ref{sec:cond_prop_func}. However, first we discuss the applications of our computational method for identifying all the closed communication classes or irreducible state-spaces for stochastic reaction networks with infinite state-spaces.

\subsection{Applications} \label{subsec:appl}

It is interesting to note that for many networks, simply knowing the irreducible state-spaces $\mathcal{E}_1, \mathcal{E}_2,\dots$ allows one to compute the corresponding stationary distributions $\pi_1,\pi_2,\dots$, either analytically (see \cite{ACKProd} and Example \ref{ex:productdistr}) or numerically \cite{sFSP}. Knowledge of these irreducible state-spaces has many other important consequences. For example, if the initial state $X(0)$ of the stochastic reaction dynamics $ ( X(t) )_{t \geq 0}$ lies in the irreducible state-space $\mathcal{E}_q$, then the ideal state-space $\mathcal{E}$ (see \eqref{ideal_state_space}) consisting of \emph{only} the reachable states, coincides exactly with $\mathcal{E}_q$. Furthermore due to Theorem 1.10.2 in \cite{Norris}, for any bounded real-valued function $f$ on $\mathcal{E}_q$ we have
\begin{align}
\label{ergodicconvergenceoff}
\lim_{t \to \infty} \E( f(X(t)) ) &= \sum_{y \in  \mathcal{E}_q } f(y) \pi_q(y)
 \end{align}
 and the following limit holds with probability $1$
 \begin{align}
 \label{resultLLN}
\lim_{t \to \infty} \frac{1}{t} \int_{0}^{t} f(X(s))ds &= \sum_{y \in \mathcal{E}_q } f(y) \pi_q(y).
\end{align} 
Using \eqref{ergodicconvergenceoff} one can show that various statistical quantities (means, variances, covariances etc.) associated with the stochastic reaction dynamics converge to their \emph{steady state} values as $t \to \infty$ (see \cite{GuptaPLOS}), which is helpful in the design of controllers that can robustly steer the mean copy-number of some species to specific reference values \cite{Briat}. Relation \eqref{resultLLN} shows that the stationary distribution of the population can be computed by evaluating the proportion of time spent in various states by a single stochastic trajectory $(X(t))_{ t \geq 0}$ over a long period of time. Such an insight can help in leveraging experimental techniques such as \emph{Flow-Cytometry} and \emph{Time-Lapse Microscopy} in the study of isogenic cell populations (see \cite{GuptaPLOS}). One can also use \eqref{resultLLN} to \emph{speed-up} the estimation of the stationary distribution $\pi_q$ using Monte Carlo simulations. Moreover the knowledge about the exact support of $\pi_q$ (viz. $\mathcal{E}_q$) can be used to sample from this stationary distribution more efficiently. Commonly we find that for networks arising in Systems Biology there is only one irreducible state-space $\mathcal{E}_1$ within $\mathcal{E}_0$ and hence the simplex $\Sigma$ collapses to a unique stationary distribution $\pi_1$ (see Section \ref{sec:examples}). In such a scenario the underlying CTMC is \emph{ergodic}, and relations \eqref{ergodicconvergenceoff} and \eqref{resultLLN} hold for any bounded real-valued function $f$ on the full state-space $\mathcal{E}_0$ and for any initial condition $X(0) \in \mathcal{E}_0$. Checking ergodicity of networks is important for many applications. For example, under ergodicity one can apply tools from \emph{Transition Path Theory} \cite{Metzner} to study the topology of networks by analyzing the statistical properties of trajectories that \emph{flow} between two subsets of the irreducible state-space. Another important application area where checking ergodicity is crucial, is for \emph{speeding-up} the stochastic simulations of multiscale networks, which have reactions firing at multiple timescales \cite{Kang}. Such networks are common in Systems Biology and it is known that their exact stochastic simulation, using Gillespie's SSA for example, is highly cumbersome, because most of the simulation time is spent in generating the \emph{fast} reaction events. To circumvent these problems approximate simulation approaches have been developed \cite{ssSSA,weinan1,weinan2} that apply the \emph{quasi-stationary assumption} (QSA) on \emph{fast} ergodic subnetworks, by supposing that their stochastic dynamics relaxes to stationarity between subsequent reactions at the \emph{slower} timescales. These \emph{fast} subnetworks can change their structure, along with their ergodic properties, depending on the states of the \emph{slow} variables which may determine the set of available \emph{fast} reactions. In this context, our results provide a way for automated discovery of \emph{fast} ergodic subnetworks, during the simulation run, and aid the correct application of QSA. We illustrate this point through an example in Section \ref{ex:toxin}. This example also shows how \emph{restricted} species can arise naturally when a \emph{fast} subnetwork within the bigger network, is considered in isolation.

\subsection{Conditions on the propensity functions}  \label{sec:cond_prop_func}

To facilitate the search for irreducible state-spaces, we shall assume that the network $\mathcal{N} = ( \mathcal{V}  , \mathcal{O}  ,\Lambda)$ satisfies the following:
\begin{assumption}
\label{massacttype}
For each reaction $k=1,\dots,K$ and each $x = (x_1,\dots,x_d) \in \mathcal{E}_0$, we have $\lambda_k(x) > 0$ if and only if $x_i \geq \nu_{ik}$ for every $i=1,\dots,d$.
\end{assumption}
In other words, at state $x$, reaction $k$ has a positive probability of firing if any only if for each species $\mathbf{S}_i$, the number of available molecules ($x_i$) exceeds the number of molecules consumed by the reaction ($\nu_{ik}$). The ``only if" part of this condition is almost always satisfied, because a reaction cannot fire unless for each species, the required number of molecules are present for consumption, but the ``if" part of this condition may get violated if the propensity function for a reaction is zero even though all the required molecules for having the reaction are present. However such situations do not typically arise for networks in  Systems Biology as we now explain.

Observe that Assumption \ref{massacttype} is certainly satisfied if we have \emph{mass-action kinetics} \cite{Goutsias} where each propensity function $\lambda_k : \N^d_0 \to \R_+$ has the form
\begin{align}
\label{massactionkinetics}
\lambda_k(x) = \theta_k \prod_{i =1}^d \frac{ x_i(x_i-1)\dots (x_i - \nu_{ik} +1 ) }{  \nu_{ik} ! },
\end{align}
for some rate constant $\theta_k > 0$. Apart from mass-action kinetics, networks in Systems Biology generally have propensity functions describing either Michaelis--Menten kinetics for enzyme-substrate interactions or Hill kinetics for ligand-protein binding dynamics \cite{Keener}. In both these cases, the propensity functions have a rational form $\lambda_k(x) =  p_k (x) / q_k (x) $, where the denominator $q_k(x)$ is always positive and the numerator $p_k (x) $ satisfies the criterion in Assumption \ref{massacttype}. Consequently the network satisfies  Assumption \ref{massacttype} even though the propensity functions are not of mass-action form \eqref{massactionkinetics}. Furthermore even if a network does not satisfy Assumption \ref{massacttype}, it can often be modified in such a way that this assumption is satisfied and its dynamics remains the same (see Section \ref{ex:ge}). We end this section with a simple proposition.
\begin{proposition}
\label{prop_reach_reln_add}
Consider a network $\mathcal{N}$ satisfying Assumption \ref{massacttype} with conservation data $(\Gamma,c)$ and state-space $\mathcal{E}_0$ given by \eqref{defn_statespacewithconser}. Then the reachability relation $\stackrel{  \mathcal{N}  }{\longrightarrow} $ is {\bf positively additive} on $\mathcal{E}_0$, which is to say that for any $x,y\in \mathcal{E}_0$ and $z \in \N^d_0$, if $x \stackrel{  \mathcal{N}  }{\longrightarrow} y$ and $(x+z) \in \mathcal{E}_0$ then we also have $(y+z) \in \mathcal{E}_0$ and $(x+z) \stackrel{  \mathcal{N}  }{\longrightarrow} (y+z)$. 
\end{proposition}
\begin{proof}
This is a simple consequence of Assumption \ref{massacttype} because for any reaction $k$ and state $u \in \mathcal{E}_0$, if $\lambda_k( u ) >0$ then $\lambda_k(u+z) >0$ for any $z \in \N^d_0$. Therefore the positive-probability reaction sequence that leads the state from $x$ to $y$, also serves as a positive-probability reaction sequence that takes the state from $(x+z)$ to $(y+z)$ (recall the discussion in Section \ref{subsec:reachability}).
\end{proof}

\subsection{Inverse of a reaction network} \label{sec:defninverse}

We now define the inverse $\mathcal{N}_{ \textnormal{inv} } = (\mathcal{V}_{ \textnormal{inv} }, \mathcal{O}_{ \textnormal{inv} },\Lambda_{\textnormal{inv}} )$ of the reaction network $\mathcal{N} = ( \mathcal{V}  , \mathcal{O}  ,\Lambda)$, which is obtained by \emph{flipping} the arrows in \eqref{formofthereaction}. In other words, the $K$ reactions in $\mathcal{N}_{ \textnormal{inv} } $ are given by
\begin{align}
\label{invformofthereaction}
\sum_{i=1}^d \rho_{ik} \mathbf{S}_i \longrightarrow \sum_{i=1}^d \nu_{ik} \mathbf{S}_i, \qquad k=1,\dots,K.
\end{align} 
Hence $\mathcal{V}_{ \textnormal{inv} } =\mathcal{O}$, $\mathcal{O}_{ \textnormal{inv} }= \mathcal{V}$ and we define the propensity map $$\Lambda_{\textnormal{inv}} (x) = \left( \lambda_{1,\textnormal{inv}} (x),\dots,\lambda_{K,\textnormal{inv}}(x) \right)$$ by letting each $\lambda_{k,\textnormal{inv}}$ to have the mass-action form (see \eqref{massactionkinetics}) with the rate constant $\theta_k=1$. Observe that the stoichiometry matrix $S_{ \textnormal{inv} }$ for $\mathcal{N}_{ \textnormal{inv} } $ is simply the negative of the stoichiometry matrix for $\mathcal{N} $, and so the space of conservation relations \eqref{defn_consrels} as well as the state-space $\mathcal{E}_0$ \eqref{defn_statespacewithconser} remain the same for both networks.

By construction, $\mathcal{N}_{ \textnormal{inv} }$ always satisfies Assumption \ref{massacttype} and if the original network $\mathcal{N}$ also satisfies this assumption then we have the following correspondence between the reachability relations induced by the two networks 
\begin{align}
\label{reln:netinverse}
x \stackrel{  \mathcal{N}  }{\longrightarrow} y \qquad \textnormal{if and only if} \qquad  y \stackrel{ \mathcal{N}_{ \textnormal{inv} } }{\longrightarrow} x, \qquad \textnormal{for any} \qquad x,y \in \mathcal{E}_0.
\end{align}  
To see this correspondence suppose that under network $\mathcal{N}$, the dynamics can reach $y$ from $x$ in one reaction step. In such a scenario for some reaction $k$ we have $\lambda_k(x) >0$ and $y = x +\zeta_k$. Since $x_i \geq \nu_{ik}$ and $\zeta_{ik} = (\rho_{ik} - \nu_{ik})$, we have $y_i \geq \rho_{ik}$ for each $i=1,\dots,d$, which ensures that $\lambda_{k,\textnormal{inv}}(y) >0$ and so under the inverse network $\mathcal{N}_{ \textnormal{inv} }$, the dynamics can reach state $x$ from state $y$ by a single firing of reaction $k$. Extending this idea to incorporate a sequence of intermediate states and reactions (see Section \ref{subsec:reachability}) one can conclude that \eqref{reln:netinverse} holds.

\subsection{Reaction network under a permutation} \label{sec:srnundersigma}

Our description of the stochastic model for network $\mathcal{N} = ( \mathcal{V}  , \mathcal{O}  ,\Lambda)$, stipulates that if the state is $x = (x_1,\dots,x_d)$ then $x_i$ denotes the copy-number of species $i$ (viz. $\mathbf{S}_{i}$). However in order to simplify the representation for state-space we would often need to redefine the correspondence between the species and the location of their copy-numbers in the state vector. This can be conveniently done by \emph{permuting} the original network to obtain a dynamically equivalent network, as we describe below.

We denote the set of all species labels by $\mathcal{D} =\{1,\dots,d\}$. Let $\sigma : \mathcal{D} \to \mathcal{D}$ be any permutation (one-to-one and onto) map. Let $P_\sigma$ be the $d \times d$ permutation matrix given by
\begin{align}
\label{defnpermutmatrix}
P_\sigma = \textnormal{Col}\left( e_{ \sigma^{-1}(1) },\dots,  e_{ \sigma^{-1}(d) } \right),
\end{align}
where $e_1,\dots,e_d$ are the standard basis vectors in $\R^d$ and let $\sigma^{-1}$ denote the inverse of map $\sigma$. Let $\mathcal{V}^{\sigma} = P_\sigma \mathcal{V}$, $\mathcal{O}^{\sigma} = P_\sigma \mathcal{O}$ and the \emph{propensity} map $\Lambda^{\sigma} : \N^d_0 \to \R^K_+$ be as in \eqref{defn_Lambda} with each $\lambda_k$ replaced by $\lambda^{\sigma }_k$ given by
\begin{align}
\label{defn_lambdak_sigma}
 \lambda^{\sigma }_k(x) = \lambda_k(  P^{T}_{\sigma} x ) \qquad \textnormal{for each} \qquad k=1,\dots,K,
\end{align}
where $P^{T}_{\sigma}= P_{ \sigma^{-1}}$ denotes the transpose of matrix $P_\sigma $. We define the permuted network as $\mathcal{N}^{\sigma} = ( \mathcal{V}^{\sigma}   , \mathcal{O}^{\sigma}   ,\Lambda^{\sigma} )$. One can see that $\mathcal{N}^\sigma$ is dynamically equivalent to $\mathcal{N}$ in the following sense. If $(X(t))_{t \geq 0}$ represents the stochastic reaction dynamics for network $\mathcal{N} $ then $(X^\sigma(t))_{t \geq 0}$ defined by
\begin{align}
\label{reln:dynequiv}
X^{\sigma}(t) = P_{\sigma} X(t) \qquad \textnormal{ for all } t \geq 0,
\end{align} 
represents the dynamics for the permuted network $\mathcal{N}^\sigma$. In other words, if $X^{\sigma}(t) = (x_1,\dots,x_d)$ then for each $i=1,\dots,d$, $x_i$ denotes the copy-number of species $\sigma(i)$ at time $t$. Therefore the appropriate conservation data and state-space for network $\mathcal{N}^\sigma$ are $( \Gamma_\sigma,c)$ and $\mathcal{E}_0^\sigma$ respectively, where 
\begin{align}
\label{defn_e0_sigma}
\Gamma_\sigma =P_{\sigma} \Gamma \qquad  \textnormal{and}  \qquad  \mathcal{E}_0^\sigma = P_{\sigma}\mathcal{E}_0:= \{   P_{\sigma} x : x \in \mathcal{E}_0 \}.
\end{align}
One can also see that if network $\mathcal{N}$ satisfies Assumption \ref{massacttype}, then the same holds for the permuted network $\mathcal{N}^\sigma$. Due to the dynamical equivalence \eqref{reln:dynequiv} we have the following proposition.
\begin{proposition}\label{prop:equivstatespacepermutation}
A set $\mathcal{E}^{\sigma} \subset \mathcal{E}_0^\sigma $ is an irreducible state-space for network $\mathcal{N}^\sigma$ if and only if the set $\mathcal{E} = P^T_\sigma \mathcal{E}^\sigma \subset  \mathcal{E}_0  $ is an irreducible state-space for network $\mathcal{N}$.
\end{proposition}

 \section{Computing the decomposed form of state-space}    \label{sec:computingdss}

The aim of this section is to provide a procedure to obtain the decomposed form \eqref{defn_dss} for state-space $\mathcal{E}_0$ for a network $\mathcal{N} = ( \mathcal{V}  , \mathcal{O}  ,\Lambda)$ with conservation data $(\Gamma,c)$. For this purpose, we may need to permute this network according to some permutation $\sigma$ and work with the equivalent network $\mathcal{N}^{\sigma}$ with state-space $\mathcal{E}_0^\sigma$ (see Section \ref{sec:srnundersigma}). 

The decomposed form can be constructed in two simple steps. Firstly by scanning the space of all semi-positive conservation relations, the \emph{bounded} species are identified and their appropriate finite state-space $\mathcal{E}_b$ is found (Section \ref{sec:findingboundedstatespace}). Secondly the rest of the species are classified as  \emph{free} or \emph{restricted} depending on any mixed-sign conservation relations between them. Moreover, the affine function $\phi$ (see \eqref{defngraphphi}), which gives the static relationship between these two sets of species, is determined (Section \ref{sec:idenfreeandrestricted}). The detailed algorithm for performing state-space decomposition is presented in Section \ref{sec:algorithms}.

\subsection{Identifying the \emph{bounded} species and their state-space}\label{sec:findingboundedstatespace}

Note that the space of all conservation relations can be expressed as $\mathcal{L} (S) =\{ \Gamma \alpha : \alpha \in \R^n \}$. Suppose that $\gamma = \Gamma \alpha$ is a semi-positive conservation relation for some $\alpha \in \R^n$. From \eqref{conservationimpl1} and \eqref{defn_consdata}, one can see that for any $i \in \textnormal{supp}(\gamma) $ we have
\begin{align}
\label{conservedrelnbddspecies}
0 \leq X_i(t) \leq  \frac{\langle c, \alpha    \rangle  }{\gamma_i}\qquad \textnormal{ for all } t \geq 0,
\end{align}
where $X_i(t)$ denotes the copy-number of species $i$ at time $t$. This shows that species $i$ is \emph{bounded}, but since it may be involved in several semi-positive conservation relations, the upper-bound $\langle c, \alpha \rangle/ \gamma_i$ for its copy-numbers may not be \emph{sharp}. To systematically account for all these relations and obtain a sharper upper-bound $b_i$ we solve the following Linear Program (LP):   
\begin{align}
\label{LPP2}
b_i = &\min_{ \alpha \in \R^n } \langle c, \alpha \rangle \\
& \textnormal{ subject to } \qquad \Gamma \alpha  \geq \mathbf{0}  \quad \textnormal{    and    } \quad \langle e_i , \Gamma \alpha \rangle = 1, \notag
\end{align}
where $ e_i$ is the $i$-th standard basis vector in $\R^d$. If the feasible region of this LP is empty, then $b_i = \infty$ and species $i$ is not involved in any semi-positive conservation relation. Henceforth we partition the set of \emph{all} species $\mathcal{D} = \{1,\dots,d\}$ into the set of \emph{bounded} species $\mathcal{D}_b = \{i \in \mathcal{D}  :  b_i < \infty \}$ and the set of \emph{unbounded} species  $\mathcal{D}_u =\{i \in \mathcal{D}  :  b_i = \infty \}$. Let $d_b =| \mathcal{D}_b |$ and $d_u =| \mathcal{D}_u | =d -d_b$ be the cardinalities of these two sets. 

Choose a permutation map $\sigma_1 : \mathcal{D} \to \mathcal{D}$ satisfying
\begin{align}
\label{defn_permsigma1}
\sigma_1(l)  \in \left\{
\begin{array}{cl}
\mathcal{D}_{b} & \textnormal{ for } l = 1,\dots,d_b  \\
\mathcal{D}_u & \textnormal{ for } l =(d_b+1),\dots,d
\end{array} \right.
\end{align}
and  consider the reaction dynamics of the network $\mathcal{N}^{\sigma}$ under permutation $\sigma = \sigma_1$ (see Section \ref{sec:srnundersigma}). Now the entries in rows $1,\dots,d_b$ of the state vectors will contain the copy-numbers of \emph{bounded} species in $\mathcal{D}_{b}$. These copy-numbers, arranged as vectors in $\N_0^{d_{b} }$, will always lie in the finite rectangular set 
\begin{align}
\label{defn_bdd_rectangle}
\mathcal{R}^{\sigma}_{b} = \{ (x_1,\dots,x_{d_{b}  }) \in \N^{d_{b} }_0 : 0 \leq x_l \leq b_{\sigma(l) } \quad \textnormal{for each } l = 1,\dots,d_{b}  \},
\end{align}
but all the elements in this set may not be reachable from each other due to conservation relations among \emph{bounded} species (see Section \ref{sec:examples}). We deal with these conservations relations now.

Let $S_{\sigma}$ be the stoichiometry matrix for network $\mathcal{N}^{\sigma}$ and let  $(\Gamma_{\sigma} , c)$ be its conservation data (see \eqref{defn_e0_sigma}). Setting $S_{\sigma}^{b} := \textnormal{Proj}(S_{\sigma} ,1,d_b )$, the conservation relations among the \emph{bounded} species are given by nonzero vectors in its left nullspace $\mathcal{L} ( S_{\sigma}^{b} )$. Suppose $n_{b} = \textnormal{dim} \left(  \mathcal{L} ( S_{\sigma}^{b}  ) \right) \geq 1$ and let $\{ \hat{\gamma}_1,\dots,\hat{\gamma}_{n_{b}} \}$ denote a \emph{basis} for $\mathcal{L} ( S_{\sigma}^{b}  )$. For each $j=1,\dots, n_{b}$, let $\bar{\gamma}_j = ( \hat{\gamma}_j ,  \mathbf{0} )  \in \R^d$ and set $\hat{c}_j =  \langle \alpha_j, c\rangle$ where $\alpha_j \in \R^n$ is the unique solution of the linear-system $\Gamma_{\sigma} \alpha_j =\bar{\gamma}_j.$ This solution $\alpha_j$ is unique because the columns of $\Gamma_{\sigma}$ are independent and they span $\mathcal{L} ( S_{\sigma} )$ which contains $\bar{\gamma}_j$. For the permuted network $\mathcal{N}^{\sigma}$, the state vectors for all the \emph{bounded} species in $\mathcal{D}_{b}$  will always lie in the finite set
\begin{align}
\label{defnEsigmab}
\mathcal{E}^{\sigma}_{b} = \{ x \in  \mathcal{R}^{\sigma}_{b} :  \langle \hat{\gamma}_j , x\rangle = \hat{c}_j  \ \textnormal{ for each } \ j=1,\dots,n_{b} \}.
\end{align}

\subsection{Identifying the \emph{free} and the \emph{restricted }species} \label{sec:idenfreeandrestricted}

We now partition the set $\mathcal{D}_u$ of \emph{unbounded} species, into a set $\mathcal{D}_f$ of \emph{free} species and a set $\mathcal{D}_r$ of \emph{restricted} species. Letting $S_{\sigma}^{u} = \textnormal{Proj}(S_{\sigma},d_b+1,d)$, we define the number of \emph{free} species $d_f$ and the number of \emph{restricted} species $d_r$ by
\begin{align}
\label{defnnumfreerestricted}
d_f :=  \textnormal{Rank}(S_{\sigma}^{u}) \qquad  \textnormal{and} \qquad d_r := d_u - d_f = d -d_b -d_f.
\end{align}     
Note that the total number of species ($d$) is equal to the sum of the number of \emph{free} ($d_f$), \emph{bounded} ($d_b$) and \emph{restricted} species ($d_r$). Observe that $d_r$ is the dimension of the left nullspace $\mathcal{L} ( S_{\sigma}^{u} )$ of matrix $S_{\sigma}^{u}$, which corresponds to the space of all conservation relations among the \emph{unbounded} species. Any nonzero vector $\gamma \in \mathcal{L} ( S_{\sigma}^{u} )$ must necessarily be a mixed-sign conservation relation, because otherwise the species in its support set \eqref{suppsetgamma} would be \emph{bounded} which is not the case. Recall that $ n  = \textnormal{dim} (  \mathcal{L} (S_{\sigma} )  ) $ is the total number of conservation relations while $n_b$ are the number of conservation relations among the \emph{bounded} species. Hence we must have $n \geq n_b + d_r$.  As it turns out, usually for biological networks (see Section \ref{sec:examples}) this inequality is strict (i.e.\ $n = n_b + d_r$) which is same as saying that the following assumption holds.
\begin{assumption}
\label{assmp:noofconsrels2}
There are no conservation relations involving both the bounded and the free/restricted species i.e.\
\begin{align*}
\textnormal{Rank}(S_{ \sigma}) = \textnormal{Rank}(S_{\sigma}^{b}) + \textnormal{Rank}(S_{\sigma}^{u}),
\end{align*}
where $S_{\sigma}^{b} = \textnormal{Proj}(S_{\sigma},1,d_b)$ and $S_{\sigma}^{u} = \textnormal{Proj}(S_{\sigma},d_b+1,d)$.
\end{assumption}
This assumption will be required to hold from now on. Suppose that \emph{restricted} species exist (i.e. $d_r \geq 1$). Hence, the space $\mathcal{L} ( S_{\sigma}^{u} )$ is nontrivial. Let $\{ \delta'_1,\dots, \delta'_{d_r} \}$ be a basis for $\mathcal{L} ( S_{\sigma}^{u} )$. For any subset $I =\{i_1,\dots,i_{d_f}\} \subset \{1,\dots,d_u\}$ with $|I| = d_f$ elements, let $A_I$ be the $d_u \times d_u$ matrix given by 
\begin{align}
\label{defn_matAi}
A_I = \textnormal{Col}\left( e_{i_1},\dots, e_{i_{d_f}}, \delta'_1,\dots, \delta'_{d_r}   \right),
\end{align}
where $e_i$-s are the standard basis vectors in $\R^{d_u}$. Define another set
\begin{align}
\label{defn_indexsetfreespecies}
\mathcal{I}_f = \{ I \subset \{1,\dots,d_u\}   : |I| = d_f \quad \textnormal{ and } \quad \textnormal{Rank}( A_I) = d_u \}.
\end{align}
Note that this set is nonempty and its cardinality is bounded above by ${  d_u \choose d_f}$.

Fix a $I \in \mathcal{I}_f $ and let $I^c$ denote its complement in the set $\{1,\dots, d_u\}$. We define the set $\mathcal{D}_f $ of \emph{free} species and the set $\mathcal{D}_r$ of \emph{restricted} species as $\mathcal{D}_f = \{ \sigma(d_b+ i) : i \in I \}$ and $\mathcal{D}_r = \{ \sigma(d_b+ i) : i \in I^c \}$. These two sets partition the set $\mathcal{D}_u$. We now choose another permutation map $\sigma_2 : \mathcal{D} \to \mathcal{D}$ satisfying
\begin{align}
\label{defn_permsigma2}
\sigma_2(l) = \sigma_1(l) \quad \textnormal{    for   } l =1,\dots,d_b \qquad  \textnormal{and} \qquad
\sigma_2(l)  \in \left\{
\begin{array}{cl}
\mathcal{D}_f & \textnormal{    for   } l =(d_b+1),\dots,(d_b+d_f) \\
\mathcal{D}_r & \textnormal{    for   } l =(d_b+d_f+1),\dots,d. 
\end{array} \right.
\end{align}

Let $\mathcal{N}^{\sigma}$ be the network under permutation $\sigma = \sigma_2$ (see Section \ref{sec:srnundersigma}), and let $S_{\sigma}$ and $(\Gamma_{\sigma} , c)$ be its stoichiometry matrix and conservation data respectively. For each $i =1,\dots, d_r$, the vector $\bar{\delta}_i = ( \mathbf{0} , \delta'_i ) \in \R^d$ belongs to $\mathcal{L}(S_{\sigma})$ and hence the vector $\hat{\delta}_i = P_{\sigma_2} P^T_{\sigma_1} \bar{\delta}_i $ belongs to $\mathcal{L}(S_{\sigma})$. Since the permutations $\sigma_1$ and $\sigma_2$ are identical on $\{1,\dots,d_b\}$, each $\hat{\delta}_i $ must have the form $\hat{\delta}_i  = ( \mathbf{0}  , \delta^{(1)}_i, \delta^{(2)}_i )$ for some vectors $\delta^{(1)}_i\in \R^{d_f}$ and $\delta^{(2)}_i \in \R^{d_r}$. Define a $d_f \times d_r$ matrix $\Delta_1 = \textnormal{Col}( \delta^{(1)}_1,\dots,\delta^{(1)}_{d_r} )$ and a $d_r \times d_r$ matrix $\Delta_2 = \textnormal{Col}( \delta^{(2)}_1,\dots,\delta^{(2)}_{d_r} )$. Observe that if $A_I$ is the matrix given by \eqref{defn_matAi}, then there exists a $d_u \times d_u$ permutation matrix $Q$ such matrix $Q A_I$ has the form
\begin{align}
\label{matrixpAi}
Q A_I = \left[ 
\begin{array}{cc}
{\bf I} &  \Delta_1 \\
\mathbf{0} & \Delta_2 
\end{array}\right],
\end{align}
where ${\bf I}$ is the $d_f \times d_f$ identity matrix and $0$ is the $d_r \times d_f$ matrix of all zeroes. Matrix $A_I$ is invertible because $I \in \mathcal{I}_f$, and hence matrix $\Delta_2$ is also invertible.

From now on let $S^{b}_{\sigma} =  \textnormal{Proj}(S_{\sigma},1,d_b)$, $S^{u}_{\sigma} =  \textnormal{Proj}(S_{\sigma},d_b+1,d)$, $S^{f}_{\sigma}=  \textnormal{Proj}(S^{u}_{\sigma},1,d_f)$ and $S^{r}_{\sigma} = \textnormal{Proj}(S^u_{\sigma},d_f+1,d_f +d_r)$. For each $i=1,\dots,d_r$, the vector $\hat{\delta}_{i} = ( \mathbf{0} , \delta^{(1)}_i, \delta^{(2)}_i  )$ belongs to  $\mathcal{L} (S_{\sigma} )$, where $\mathbf{0}$ denotes the vector of zeros in $\R^{d_b}.$ Hence we must have $\Delta^T_1 S^{f}_{\sigma} = - \Delta^T_2  S^{r}_{\sigma}$ which allows us to write
\begin{align}
\label{srdependsonsf}
S^{r}_{\sigma}  = -(\Delta^T_2)^{-1}\Delta^T_1 S^{f}_{\sigma} =  -(\Delta^{-1}_2)^{T}\Delta^T_1 S^{f}_{\sigma}.
\end{align}
This also shows that $\textnormal{Rank}( S^{f}_{ \sigma }  ) = d_f$ because $d_f =  \textnormal{Rank}(S_{\sigma}^{u})$ (see \eqref{defnnumfreerestricted}) and $S_{\sigma}^{u}$ has the form
\begin{align*}
S_{\sigma}^{u} = 
\left[ 
\begin{array}{c}
S_{\sigma}^{f}  \\
S_{\sigma}^{r} 
\end{array}
\right] = 
\left[ 
\begin{array}{c}
{\bf I} \\
-(\Delta^{-1}_2)^{T}\Delta^T_1
\end{array}
\right] S_{\sigma}^{f} . 
\end{align*}

The state-space for network $ \mathcal{N}^{\sigma}$ is $\mathcal{E}^{\sigma}_0 = P_{ \sigma } \mathcal{E}_0 = \{ x \in \N^d_0 : \Gamma_{ \sigma } x = c \}$ (see \eqref{defn_e0_sigma}). Any element in this set can be expressed as $x = ( x_b, x_f,x_r )$ where $x_b \in \mathcal{E}^{\sigma}_b$, $x_f \in \N^{d_f}_0$ and $x_r \in \N^{d_r}_0$. Since the columns of $\Gamma_{\sigma} $ span $\mathcal{L}(S_{\sigma})$, for each $i=1,\dots,d_r$, there exists a unique $n \times d_r$ matrix $M$ such that
\begin{align*}
\Gamma_{\sigma}  M = \textnormal{Col}( \hat{\delta}_1,\dots, \hat{\delta}_{d_r}).
\end{align*}
Since Assumption \ref{assmp:noofconsrels2} holds, the last $(d_f+d_r)$ rows of the condition $\Gamma_{ \sigma } x = c$ yield $ \Delta^T_1 x_f +\Delta^T_2 x_r  =\hat{c}$ for $\hat{c} = M^T c$. This shows that $x_r = \phi(x_f)$, where the \emph{affine} map $\phi : \R^{d_f} \to \R^{d_r}$ is defined by
\begin{align}
\label{defn_map_phi}
\phi(x) = (\Delta^T_2)^{-1} \hat{c}  -(\Delta^T_2)^{-1}\Delta^T_1 x  .
\end{align}
This analysis proves the following proposition.
\begin{proposition}
\label{prop:ssdecomp}
Suppose Assumption \ref{assmp:noofconsrels2} holds. Then state-space $\mathcal{E}^{\sigma}_0$ for network $\mathcal{N}^{\sigma}$ can be expressed as
\begin{align}
\label{defn_ebsigma2b}
\hat{\mathcal{E}}^{\sigma}_0 = \mathcal{E}^{\sigma}_b \times \Phi,
\end{align}
where $\Phi$ is the graph \eqref{defngraphphi} of function $\phi$ defined by \eqref{defn_map_phi}. 
\end{proposition}
Note that this state-space decomposition result only depends on the reaction stoichiometries but not on their propensities. Indeed the propensity functions can be completely general as long as they satisfy the basic assumption \eqref{statespacecond} which ensures that the dynamics is contained in the positive orthant. We end this section with an important remark.
\begin{remark}
\label{rem:flixibilityofI}
Note that the classification of unbounded species into free and restricted \\ species depends on the set $I$ which can chosen to be any element in the set $\mathcal{I}_f$ given by \eqref{defn_indexsetfreespecies}. This flexibility will be useful in the next section. 
\end{remark}

\section{Identifying the irreducible state-spaces}  \label{sec:irredclasses}

In this section we shall assume that network $\mathcal{N} = ( \mathcal{V}  , \mathcal{O}  ,\Lambda)$ satisfies Assumption \ref{massacttype}. As a consequence, the permuted network $\mathcal{N}^{ \sigma } = ( \mathcal{V}^{\sigma}   , \mathcal{O}^{\sigma}   ,\Lambda^{\sigma} ) $ also satisfies this assumption, which is a property that will play a crucial role in our search for irreducible state-spaces within the infinite state-space $\mathcal{E}^{\sigma}_0$. Recall that these state-spaces are the closed communication classes for relation $ \stackrel{ \mathcal{N}^{ \sigma }  }{\longleftrightarrow}$ on $\mathcal{E}^{\sigma}_0$ (see Section \ref{subsec:reachability}). From the discussion in Section \ref{sec:idenfreeandrestricted} it is immediate that \emph{restricted} species have no independent dynamics of their own and they essentially \emph{mimic} the \emph{free} species according to the affine map $\phi$. This suggests that for finding irreducible state-spaces we can simply remove the \emph{restricted} species and concentrate on the dynamics of the \emph{bounded} and the \emph{free} species. We now describe this step formally.

\subsection{Network reduction by elimination of \emph{restricted} species} \label{sec:netreductionbyelimination}

We construct a ``reduced" network $\tilde{\mathcal{N} }^{\sigma}$ with $(d_b+d_f)$ species in the set $\mathcal{D}_b \cup \mathcal{D}_f $ where $\mathcal{D}_b = \{  \sigma(1),\dots, \sigma(d_b) \} $ and $ \mathcal{D}_f =\{ \sigma(d_b+1),\dots, \sigma(d_b+d_f) \}$. For each reaction $k=1,\dots,K$, let its propensity function $\tilde{ \lambda }^{\sigma}_k : \N^{d_b}_0 \times \N^{d_f}_0 \to \R_+$ be given by
 \begin{align}
 \label{defnreducedpropensity}
\tilde{ \lambda }^{\sigma}_k(x_b,x_f ) = \lambda_k^{\sigma }( x_b ,x_f, \phi(x_f) )
\end{align}
where $x_b \in \N^{d_b}_0, x_f \in \N^{d_f}_0$ and the function $\lambda_k^{\sigma }$ is defined by \eqref{defn_lambdak_sigma}. Let the \emph{propensity} map $\tilde{ \Lambda }^{\sigma} : \N^{ d_b+d_f}_0 \to \R^K_+$ be as in \eqref{defn_Lambda} with each $\lambda_k$ replaced by $\tilde{ \lambda }^{\sigma}_k$. Setting $\tilde{ \mathcal{V} }^{\sigma} = \textnormal{Proj}(\mathcal{V}^{ \sigma },1,d_b+d_f) $ and $\tilde{ \mathcal{O} }^{\sigma} = \textnormal{Proj}(\mathcal{O}^{ \sigma } ,1,d_b+d_f)$, we define the reduced network as $\tilde{\mathcal{N} }^{\sigma}= ( \tilde{ \mathcal{V} }^{\sigma},\tilde{ \mathcal{O} }^{\sigma},  \tilde{ \Lambda }^{\sigma}  )$.

Unfortunately the reduced network $\tilde{\mathcal{N} }^{\sigma}$ may not satisfy Assumption \ref{massacttype} even though the original network does (see Example \ref{ex:toxin}). We need to avoid this problem because our approach requires this property. Fortunately for most biological examples this can be done by exploiting the flexibility in the choice of set $I$ (see Remark \ref{rem:flixibilityofI}), which classifies each \emph{unbounded} species as \emph{free} or \emph{restricted}. Note that different choices of $I$ will yield different reduced networks but they correspond to the same dynamics for the original network. Hence the irreducible state-spaces  for the original network can be found with any $I$ chosen as per our convenience. We sequentially examine each element in the finite set $\mathcal{I}_f$ (see \eqref{defn_indexsetfreespecies}) until we either find a $I$ for which the affine function $\phi$ satisfies the following assumption or we exhaust all the possibilities (see Algorithm \ref{finddecompss}).  
\begin{assumption}
\label{affinefunction}
Consider an affine map $f : \R^{d_f} \to \R^{d_r}$
\begin{align*}
f(x) = F_0 + F_1 x,
\end{align*}
where $F_0$ is a vector in $\R^{d_r}$ and $F_1$ is a $d_r \times d_f$ matrix. We say that this map is compatible with network $\mathcal{N}^{ \sigma } = ( \mathcal{V}^{\sigma}   , \mathcal{O}^{\sigma}   ,\Lambda^{\sigma} ) $ if $F_0$ and $F_1$ have all the entries in $\N_0$ and the matrix inequality
\begin{align}
\label{affmap:requirements}
 F_0  \mathbf{1}^T + F_1 \mathcal{V}^{\sigma}_f \geq  \mathcal{V}^{\sigma}_r
\end{align}
holds, where $\mathcal{V}^{\sigma}_f = \textnormal{Proj}(  \mathcal{V}^{\sigma}, d_b+1,d_b+d_f)$, $\mathcal{V}^{\sigma}_r = \textnormal{Proj}(  \mathcal{V}^{\sigma}, d_b+d_f +1,d)$ and $\mathbf{1}$ is the vector of all ones in $\R^K$.
\end{assumption}

If the affine map $\phi$ satisfies this compatibility condition, then the network reduction will automatically satisfy Assumption \ref{massacttype}. To see this note that \eqref{affmap:requirements} implies that for any reaction $k$, if $x_f \geq \nu_{k,f}$ then $\phi( x_f ) \geq \nu_{r,f}$, where $\nu_{k,f}$ ($\nu_{r,f}$) denotes the $k$-th column of matrix $\mathcal{V}^{\sigma}_f$ ($\mathcal{V}^{\sigma}_r$). Therefore the reduced network $\tilde{\mathcal{N} }^{\sigma}$ satisfies Assumption \ref{massacttype} because
\begin{align*}
\tilde{ \lambda }^{\sigma}_k(x_b,x_f ) > 0 \quad &\Longleftrightarrow \quad  \lambda^{\sigma}_k(x_b,x_f, \phi(x_f) ) > 0 \\
&\Longleftrightarrow \quad  (x_b , x_f, \phi(x_f)) \geq ( \nu_{k,b} , \nu_{k,f} , \nu_{k,r} ) \\
&\Longleftrightarrow \quad  (x_b , x_f) \geq ( \nu_{k,b} , \nu_{k,f}  ) \quad \textnormal{and} \quad \phi( x_f ) \geq \nu_{r,f}  \\
&\Longleftrightarrow \quad  (x_b , x_f) \geq ( \nu_{k,b} , \nu_{k,f}  ),
\end{align*}
where $\nu_{k,b}$ is the $k$-th column of matrix $\textnormal{Proj}(  \mathcal{V}^{\sigma},1,d_b)$, $\Longleftrightarrow $ denotes ``if and only if" and the last relation holds because $\phi( x_f ) \geq \nu_{r,f}$ follows from \eqref{affmap:requirements}.

It is easy to check that the reduced network $\tilde{\mathcal{N} }^{\sigma}$ preserves all the semi-positive conservation relations among the \emph{bounded} species. Therefore if $\phi$ satisfies Assumption \ref{affinefunction} the state-space for this reduced network is
\begin{align*}
\tilde{\mathcal{E} }^{\sigma}_0 = \mathcal{E}_b^{\sigma} \times \N^{d_f}_0.
\end{align*}
In what follows, we shall identify all the irreducible state-spaces for the communication relation $ \stackrel{\tilde{\mathcal{N} }^{\sigma} }{\longleftrightarrow}$ induced by the network $\tilde{\mathcal{N} }^{\sigma}$ on $\tilde{\mathcal{E} }^{\sigma}_0$. The next proposition then allows us to recover all the irreducible state-spaces for the original network $\mathcal{N}^{\sigma}$. 
\begin{proposition}
\label{mainprop_reducednetwork}
Assume that the affine map $\phi$, given by \eqref{defn_map_phi}. Then for any $F \subset \mathcal{E}^{\sigma}_b$ and $G \subset \N^{d_f}_0$, the set $F  \times G$ is an irreducible state-space for relation $ \stackrel{\tilde{\mathcal{N} }^{\sigma} }{\longleftrightarrow}$ if and only if the set $F \times \Phi_{G}$  is an irreducible state-space for relation $ \stackrel{\mathcal{N} ^{\sigma} }{\longleftrightarrow}$. Here $\Phi_G$ is the graph of function $\phi$ (see \eqref{defngraphphi}), restricted to the domain $G$.
\end{proposition}
\begin{proof}
The proof follows simply from the construction of the reduced network $\tilde{\mathcal{N} }^{\sigma}$ and the fact that the dynamics of the \emph{restricted} species is ``tied" to the dynamics of the \emph{free} species according to map $\phi$.
\end{proof}

\subsection{Networks with only \emph{bounded} species} \label{sec:netbounded}

We first consider the case when there are no \emph{free} species (i.e. $d_f=0$), and hence the state-space $\tilde{\mathcal{E} }^{\sigma}_0 = \mathcal{E}^{\sigma}_b$ for network $\tilde{\mathcal{N} }^{\sigma}$ is finite\footnote{Observe that as we have no \emph{free/restricted} species, we do not need to permute the network (i.e. $\sigma$ can be chosen as the identity permutation) or reduce it to eliminate the \emph{restricted} species (i.e. $\tilde{\mathcal{N} }^{\sigma} = \mathcal{N}$). However we employ this complicated notation to be consistent with the general case where both \emph{bounded} and \emph{free/restricted} species are present.}. In such a situation all the irreducible state-spaces can be found using simple matrix manipulations. We briefly describe this approach in the context of reaction networks and introduce the relevant concepts that will be useful later in the paper. 

Let $\mathcal{E}^{\sigma}_b = \{ y_1,\dots, y_{ N_b } \}$ be the finite state-space with $N_b = |  \mathcal{E}^{\sigma}_b |$ states. For each reaction $k$, let $\bar{\nu}^{\sigma}_k \in \N^{d_b}_0 $ and $\bar{\zeta}^{\sigma}_k \in \Z^{d_b} $ be the vectors denoting the $k$-th column of the reactant and the stoichiometry matrices of the network. As network $\tilde{\mathcal{N} }^{\sigma}$ satisfies Assumption \ref{massacttype}, only the following reactions have a positive propensity of firing
\begin{align}
\label{defn_availablerxns}
\mathcal{K}_r(y) = \{ k=1,\dots,K:  y \geq  \bar{\nu}^{\sigma}_k\}.
\end{align}
Define a $N_b \times N_b$ matrix $Z$ by
\begin{align}
\label{zeropatternmat1}
Z_{ij} = \left\{
\begin{tabular}{cl}
$1$ & if $y_{j} = y_{i} +\bar{\zeta}^{\sigma}_k$ \ for some $k \in \mathcal{K}_r(y_{i})$ \\  
$0$ & otherwise.
\end{tabular} \right.
\end{align}
We can view $Z$ as the \emph{zero-pattern} matrix\footnote{The \emph{zero-pattern} matrix corresponding to a finite Markov chain is obtained by setting all the positive entries in its probability transition matrix to $1$ and all the rest to $0$} of a finite Markov chain \cite{Kemeny} and use it to study reachability and communication relations corresponding to network $\tilde{\mathcal{N} }^{\sigma}$ (see Section \ref{subsec:reachability}). Define the $N_b \times N_b$ \emph{reachability} matrix $\Omega$ by  
\begin{align}
\Omega  = \left( {\bf I} + Z  \right)^{N_{b}-1},
\end{align}   
where ${\bf I}$ is the $N_b \times N_b$ identity matrix. For any two states $y_i, y_j \in \mathcal{E}^{\sigma}_b$, state $y_{i}$ is reachable from state $y_{j}$ if and only if the $ij$-th entry of this reachability matrix is positive (i.e. $\Omega_{ij} > 0$). Therefore based on this matrix we can define the communication relation $\Theta$ on $\mathcal{E}^{\sigma}_b$ by
\begin{align}
\label{defnthetaxf}
\Theta  = \{ (y_{i} ,y_{j} ) \in  \mathcal{E}^{\sigma}_{b} \times \mathcal{E}^{\sigma}_{b} :  \Omega_{ij} > 0 \quad \textnormal{and} \quad \Omega_{ji}  >0 \}.
\end{align}
As we mentioned in Section \ref{subsec:reachability}, $\Theta$ is an equivalence relation on $\mathcal{E}^{\sigma}_b$ which partitions this set into distinct equivalence or communication classes. Let $\bar{\Theta} = [\bar{\Theta}_{ij}]$ be the $N_b \times N_b$ matrix representing this relation, i.e. $\bar{\Theta}_{ij} = 1$ if $(y_{i} ,y_{j} ) \in \Theta$ and $\bar{\Theta}_{ij} = 0$ otherwise. One can verify that states $y_i$ and $y_j$ will belong to the same communication class if and only if rows $i$ and $j$ are identical in matrix $\bar{\Theta}$. Let $n_c$ be the number of unique rows of matrix $\bar{\Theta}$ and let $U$ be the $n_c \times N_b$ matrix formed by these rows. Each row $i$ of $U$ corresponds to a distinct communication class made up of those states $y_j$ for which $U_{ij} =1$. To study the interaction among communication classes we define another $n_c \times n_c$ matrix $R$ by
\begin{align}
\label{reachabilitycommclassrelan}
R = U Z U^T - {\bf I}.
\end{align}
This matrix captures the reachability relations among the communication classes (see Chapter 8 in \cite{kemeny1957introduction}). In other words, if $R_{ij} >0$, then there exists a reaction that takes a state in the $i$-th communication class to a state in the $j$-th communication class. If the $i$-th row of matrix $R$ only consists of zeros then the $i$-th communication class in \emph{closed} and otherwise it is \emph{open}. Starting from any initial state in $\mathcal{E}^{\sigma}_b$ the dynamics of the \emph{bounded} species will eventually get trapped in one of the closed communication classes and these are the {\bf only} irreducible state-spaces for network $\tilde{\mathcal{N} }^{\sigma}$ without any \emph{free} species.

\subsection{Networks with only \emph{unbounded} species} \label{sec:freespecies}

In this section we assume that there are no \emph{bounded} species (i.e. $d_b= 0 $) and all the irreducible state-spaces must lie inside the nonnegative integer orthant $\tilde{\mathcal{E} }^{\sigma}_0 = \N^{d_f}_0$, after the \emph{restricted} species have been removed (see Section \ref{sec:netreductionbyelimination}). We will find the irreducible state-spaces by adopting a simple scheme that attempts to arrange the \emph{free} species into \emph{birth and death cascades}. We begin by formalizing the notion of birth-cascades for the network $\tilde{\mathcal{N} }^{\sigma} =  ( \tilde{ \mathcal{V} }^{\sigma},\tilde{ \mathcal{O} }^{\sigma},  \tilde{ \Lambda }^{\sigma} )$.

For each reaction $k$, let $\tilde{\nu}^{\sigma}_k$ and $\tilde{\rho}^{\sigma}_k$ be the vectors in $\N^{d_f}_0$, denoting the $k$-th column of matrices $ \tilde{ \mathcal{V} }^{\sigma}$ and $ \tilde{ \mathcal{O} }^{\sigma}$ respectively. As $d_b =0$ the set of all \emph{free} species is $ \mathcal{D}_f =\{ \sigma(1),\dots, \sigma(d_f) \}$. From now on let $\mathcal{F} = \{1,\dots,d_f\}$ denote the set of addresses of all the \emph{free} species under the map $\sigma( \cdot)$. For any $A \subset \mathcal{F}$, let $\mathcal{D}^{ \sigma }_f (A)$ be the subset of \emph{free} species given by
\begin{align}
\label{defndfsigmaa}
\mathcal{D}^{ \sigma }_f (A):= \{ \sigma(i)  \in \mathcal{D}_f : i \in A \}
\end{align}
and let $\mathbb{B}(A) \subset \mathcal{F}$ be defined by
\begin{align*}
\mathbb{B}(A) = \left\{ i \in \mathcal{F}: i \notin A , \ \  \textnormal{supp}( \tilde{\nu}^{\sigma}_k) \subset A  \ \   \textnormal{and} \ \  i \in \textnormal{supp}(\tilde{\rho}^{\sigma}_k)  \ \  \textnormal{for some reaction} \ \  k =1 ,\dots,K \right \}.
\end{align*}
The set $\mathcal{D}^{ \sigma }_f (\mathbb{B}(A) )$ represents the \emph{free} species that do not belong to the set $\mathcal{D}^{ \sigma }_f (A)$ and are produced by a reaction that only consumes the \emph{free} species in $\mathcal{D}^{ \sigma }_f (A)$.  

Using this mapping $\mathbb{B}$, we define a sequence of subsets of $\mathcal{F}$, indexed by nonnegative integer \emph{levels} $l=0,1,\dots$, as follows: let $G_0 = \emptyset$ and for each $l \geq 1$ let
\begin{align*}
G_l = G_{l-1} \cup \mathbb{B}( G_{l-1} ).
\end{align*}
This sequence of sets $\{ G_l : l =0,1,2,\dots \}$ represent the \emph{birth-cascades} for network $\tilde{\mathcal{N} }^{\sigma}$. At any level $l$, the set $\mathcal{D}^{ \sigma }_f (G_l)$ consists of all those \emph{free} species that either belong to the previous cascade $\mathcal{D}^{ \sigma }_f (G_{l-1})$ or it is produced by a reaction that only consumes the \emph{free} species in this previous cascade $\mathcal{D}^{ \sigma }_f (G_{l-1})$. In particular for level $l=1$ the set $\mathcal{D}^{ \sigma }_f (G_1)$ consists of all those \emph{free} species that can be produced \emph{from nothing}, due to reactions of the form $\emptyset \longrightarrow {\bf S}$. From this cascade construction one can expect that all the \emph{free} species in each of these birth-cascades, can have arbitrarily high copy-numbers due to the reaction dynamics. This is precisely what we show later in the paper (see Lemma \ref{lem:level3}).

As the number of \emph{free} species is finite, the following number is well-defined
\begin{align*}
l_b = \max\{ l \in \N  :  G_l \neq G_{l-1} \} =  \min\{ l \in \N  :  \mathbb{B}( G_{l-1} )  = \emptyset \} .
\end{align*}
The monotonically increasing sequence of sets $G_0,G_1,\dots$ stops growing beyond level $l_b$ and so we have $G_l =  G_{l_b}$ for all $l \geq l_b$. We define the \emph{Birth-Cascade Path} (BCP) as the directed graph
\begin{align}
\label{defn_bcp}
\emptyset \Rightarrow G_1 \Rightarrow G_2 \Rightarrow \dots \Rightarrow G_{l_b}
\end{align}
and refer to $G_{l_b}$ as its terminal node.  The set of all the \emph{free} species that can be arranged into birth-cascades is $$\mathcal{B} = \mathcal{D}_f^{\sigma}( G_{l_b} ).$$ We define two reaction-sets as
\begin{align}
\label{defn_availablereactionsfree}
\mathcal{K}_r = \{ k=1,\dots,K:  \textnormal{supp}( \tilde{\nu}^{\sigma}_k)   \subset G_{l_b}  \} \quad \textnormal{and} \quad \mathcal{K}_p = \{ k=1,\dots,K:  \textnormal{supp}( \tilde{\rho}^{\sigma}_k)   \subset G_{l_b}  \}.
\end{align}
The set $\mathcal{K}_r$ ($\mathcal{K}_p $) consists of all those reactions that do {\bf not} consume (produce) the \emph{free} species that lie outside the set $\mathcal{B}$. Since $G_{l_b}$ is the terminal node of the BCP we must have that
\begin{align}
\label{reactioninclusionbirthsaturation}
\mathcal{K}_r   \subset  \mathcal{K}_p, 
\end{align}
which is to say that any reaction that only consumes the \emph{free} species in $\mathcal{B}$ must not produce any \emph{free} species outside this set.

Analogous to birth-cascades we now construct the death-cascades for network $\tilde{\mathcal{N} }^{\sigma}$, but we restrict our attention to only the species in $\mathcal{B}$ and only the reactions in $\mathcal{K}_r$. Let $\tilde{\mathcal{N} }^{\sigma}( \mathcal{B} )$ be the network restricted to only these species and reactions and let $\tilde{\mathcal{N} }^{\sigma}_{ \textnormal{inv} } ( \mathcal{B} ) $ be its inverse (see Section \ref{sec:defninverse}). We construct the birth-cascades for this inverse network $\tilde{\mathcal{N} }^{\sigma}_{ \textnormal{inv} } ( \mathcal{B} ) $ as before, and define the $l$-th death-cascade $\hat{G}_l$ for network $\tilde{\mathcal{N} }^{\sigma}$ as the $l$-th birth-cascade for network $\tilde{\mathcal{N} }^{\sigma}_{ \textnormal{inv} }( \mathcal{B} )$. Correspondingly the \emph{Death-cascade Path} (DCP) for network $\tilde{\mathcal{N} }^{\sigma}$ is the BCP for network $\tilde{\mathcal{N} }^{\sigma}_{ \textnormal{inv} }( \mathcal{B} )$ and we can represent it as
\begin{align}
\label{defn_dcp}
\emptyset \Rightarrow \hat{G}_1 \Rightarrow \hat{G}_2 \Rightarrow \dots \Rightarrow \hat{G}_{l_d},
\end{align}
where $\hat{G}_{l_d}$ is the terminal node of DCP and the set $$\mathcal{X} = \mathcal{D}^{ \sigma }_f( \hat{G}_{l_d} ) $$
consists of all the \emph{free} species that can be arranged into death-cascades. Note that since the set of species is restricted to $\mathcal{B}$ we must have $\mathcal{X} \subset \mathcal{B}$. The dynamical relation \eqref{reln:netinverse} satisfied by a network and its inverse, provides us with the following interpretation of death-cascades: the \emph{free} species in $\mathcal{D}^{ \sigma }_f (\hat{G}_1)$ for level $l=1$ degrade spontaneously due to reactions in $\mathcal{K}_r$ of the form $ {\bf S} \longrightarrow \emptyset $, while for higher levels $l>1$ each \emph{free} species ${\bf S}_i$ in $\mathcal{D}^{ \sigma }_f (\hat{G}_l)$ is converted by a reaction in $\mathcal{K}_r$ to some \emph{free} species ${\bf S}_j$ at a lower level, i.e. species ${\bf S}_j$ belongs to the set $\mathcal{D}^{ \sigma }_f (\hat{G}_{l-1})$. For most biological networks, spontaneous degradation is very common and so typically the first death-cascade $\mathcal{D}^{ \sigma }_f (\hat{G}_1)$ is heavily populated (see Section \ref{sec:examples}). The species in $\mathcal{B}$ that do not belong to $\mathcal{D}^{ \sigma }_f (\hat{G}_1)$ usually belong to one of the higher death-cascades because generally they undergo a sequence of conversions to eventually produce a spontaneously degradable species (see Section \ref{sec:examples}). These remarks suggest that for most biological networks we can expect to have
\begin{align}
\label{bdsatinclusion}
\mathcal{B}  =  \mathcal{X},
\end{align}
which is to say that all the \emph{free} species in $\mathcal{B}$ can be arranged into death-cascades by the procedure we just described.

For our main result we need to ensure that each and every molecule of the \emph{free} species in $\mathcal{B}$ can be \emph{flushed-out} from the system. Due to this reason we need to impose certain stoichiometric restrictions on the network as we describe now. Let $\tilde{S}_{\sigma}$ be the $d_f \times K$ stoichiometry matrix for network $\tilde{\mathcal{N} }^{\sigma}$ and let $\tilde{S}_{\sigma}\left[  \mathcal{K}_p \right]$ be the matrix formed by restricting $\tilde{S}_{\sigma}$ to only those columns that correspond to reactions in $\mathcal{K}_p$ (see \eqref{defn_availablereactionsfree}). For any $i \in \mathcal{F}$ the \emph{free} species $\sigma(i) \in \mathcal{D}_f$ is said to be \emph{singularly-degradable} if   
\begin{align}
\label{newcondition0}
- e_i \subset \textnormal{Colspan}_{\N_0} \left(\tilde{S}_{\sigma}\left[  \mathcal{K}_p \right]\right), 
\end{align}
where $e_i$ denotes the $i$-th standard basis vector in $\R^{d_f}$. Note that the column span in this condition is over nonnegative integers $\N_0$. Observe that if a \emph{free} species can degrade spontaneously because of a reaction like $ {\bf S} \longrightarrow \emptyset $ then it is certainly singularly-degradable. Hence condition \eqref{newcondition0} primarily pertains to those \emph{free} species that require a sequence of conversion reactions to produce a spontaneously degradable species. We can easily identify all the singularly-degradable \emph{free} species by computing a modified Hermite normal form\footnote{The Hermite normal form is an analogue of the row-reduced echelon form for integer matrices. For more details see \cite{Cohen}.} of the transpose of the matrix $- \tilde{S}_{\sigma}\left[  \mathcal{K}_p \right]$. For obtaining this normal form, the only admissible operation is the addition of one row to another, possibly after multiplication by a positive integer. Once this Hermite normal form has been computed, the \emph{free} species $\sigma(i)$ is singularly-degradable if and only if there exists a row with the leading entry of $1$ in column $i$ of the Hermite normal form.

We now state the main result of this section Proposition \ref{prop:no_bdd_species}, which determines an irreducible state-space for network $\tilde{\mathcal{N} }^{\sigma}$ and also provides a way to easily check if it is the \emph{only irreducible state-space} for the network in the infinite state-space $\tilde{\mathcal{E} }^{\sigma}_0 = \N^{d_f}_0$. To state our main result, we shall consider the network dynamics under the permutation $\sigma = \sigma_3 :  \mathcal{D} \to \mathcal{D}$ defined by
\begin{align}
\label{permsigma3}
\sigma_3 ( i ) =  \left\{
\begin{array}{cl}
j_i  & \textnormal{    for   } i =1,\dots,d_{fb} \\
m_{i - d_{fb}}& \textnormal{    for   } i =(d_{fb}+1),\dots,d_f \\
\sigma_2(i) & \textnormal{    for   } i =(d_f+1,\dots,(d_f+ d_r),
\end{array} \right. 
\end{align}
where $d_{fb} = | \mathcal{B}  |$ be the total number of \emph{free} species that can be arranged into birth-cascades, the set $\mathcal{B}$ is $\mathcal{B} = \{ j_1,\dots,j_{d_{fb}} \}$ and the set of remaining \emph{free} species is \\$\mathcal{D}_f \cap \mathcal{B}^c = \{m_1,\dots,m_{d_f - d_{fb} }\}$. Observe that under this permutation the copy-numbers of all the \emph{free} species in $\mathcal{B}$ occupy the first $d_{fb}$ components of the state vector.
\begin{proposition}
\label{prop:no_bdd_species}
Assume that network $\tilde{\mathcal{N} }^{\sigma}$ satisfies Assumption \ref{massacttype} and it has no bounded species (i.e. $d_b = 0$). Suppose that permutation $\sigma$ is defined by \eqref{permsigma3} and \eqref{bdsatinclusion} holds. Then we have the following:
\begin{itemize}
\item[(A)] If all the \emph{free} species in $\mathcal{B}$ are singularly-degradable then $\N_0^{ d_{fb} } \times \{ \mathbf{0} \}$ is an irreducible state-space for network $\tilde{\mathcal{N} }^{\sigma}$, where $\mathbf{0}$ denotes the $(d_f - d_{fb})$-dimensional vector of zeros.
\item[(B)] Additionally if all the \emph{free} species in $\mathcal{D}_f\cap \mathcal{B}^c$ are also singularly-degradable then $\N_0^{ d_{fb}} \times \{ \mathbf{0} \}$ is {\bf the only} irreducible state-space for network $\tilde{\mathcal{N} }^{\sigma}$.
\end{itemize}
\end{proposition}
\begin{proof}
This proposition is a special case of Theorem \ref{maintheorem} that is proved in the Appendix.
\end{proof}

We end this section with a simple example that emphasizes the importance of imposing the condition of being singularly-degradable on the \emph{free} species. Consider a network in which a single species ${\bf S}$ participates in the following two reactions
\begin{align*}
\emptyset \longrightarrow 2 {\bf S} \longrightarrow \emptyset. 
\end{align*}
We assume that the propensity functions follow mass-action kinetics. In this network the molecules of ${\bf S}$ are produced and degraded in \emph{pairs}. One can easily check that for this network $\mathcal{B} = \mathcal{X}= \{ {\bf S} \}$, but $\N_0$ is {\bf not} an irreducible state-space because species ${\bf S}$ is not singularly-degradable. One can check that this network has two irreducible state-spaces: the set of all odd integers $\mathbb{O}= \{1,3,5,\dots\}$ and the set of all even integers $\mathbb{E}= \{0,2,4,\dots\}$.

\subsection{Networks with both \emph{unbounded} and \emph{bounded} species} \label{sec:freebddspecies}

In this section we consider the general case where network $\tilde{\mathcal{N} }^{\sigma} =  ( \tilde{ \mathcal{V} }^{\sigma},\tilde{ \mathcal{O} }^{\sigma},  \tilde{ \Lambda }^{\sigma} )$ has both \emph{unbounded} and \emph{bounded} species and so all the irreducible state-spaces will lie in the infinite set $\tilde{\mathcal{E} }^{\sigma}_0 = \mathcal{E}_b^{\sigma} \times \N^{d_f}_0$ after the \emph{restricted} species have been removed (see Section \ref{sec:netreductionbyelimination}). We now develop a procedure to identify all these state-spaces by \emph{intertwining} the matrix-based approach in Section \ref{sec:netbounded} with the birth-death cascade construction described in Section \ref{sec:freespecies}. However instead of linear \emph{path-like} cascades we need to construct \emph{tree-like} structures as we explain below.

We first fix some notation. For each reaction $k=1,\dots,K$ let $\bar{\nu}^{\sigma}_k$, $ \tilde{\nu}^{\sigma}_k$, $\bar{\rho}^{\sigma}_k$ and $ \tilde{\rho}^{\sigma}_k$ denote the $k$-th column of matrices $\textnormal{Proj}( \tilde{ \mathcal{V} }^{\sigma},1,d_b)$, $\textnormal{Proj}( \tilde{ \mathcal{V} }^{\sigma},d_b+1,d_b+d_f)$, $\textnormal{Proj}( \tilde{ \mathcal{O} }^{\sigma},1,d_b)$ and $\textnormal{Proj}( \tilde{ \mathcal{O} }^{\sigma},d_b+1,d_b+ d_f)$ respectively. Recall that the set of all \emph{free} species is $\mathcal{D}_f = \{ \sigma(d_b+1),\dots, \sigma(d_b+d_f)\}$ and let $\mathcal{F} = \{1,\dots,d_f\}$ denote the set of addresses of all the \emph{free} species under the map $\sigma(d_b+ \cdot)$. For any $A \subset \mathcal{F}$, let $\mathcal{D}^{ \sigma }_f (A) \subset \mathcal{D}_f$ be given by \eqref{defndfsigmaa} with $\sigma(\cdot)$ replaced by $\sigma(d_b+ \cdot)$. Corresponding to such a set $A$ we define a set of reactions as
\begin{align*}
\mathcal{K}_r(  A ) = \{ k=1,\dots,K:  \textnormal{supp}( \tilde{\nu}^{\sigma}_k)   \subset A \},
\end{align*}
which includes only those reactions that do {\bf not} consume any \emph{free} species that lie outside the set $\mathcal{D}^{\sigma}_f(A)$. For any $y \in \mathcal{E}_b^{\sigma}$ define
\begin{align*}
\mathcal{K}_r( y, A  ) = \mathcal{K}_r( y) \cap \mathcal{K}_r( A),
\end{align*}
where $\mathcal{K}_r( y)$ is given by \eqref{defn_availablerxns}. Let $Z(A)$ be the zero-pattern matrix defined by \eqref{zeropatternmat1} with $\mathcal{K}_r( y)$ replaced by $\mathcal{K}_r( y,A)$. Using the procedure described in Section \ref{sec:netbounded} we can compute the associated communication relation $\Theta(A)$ (see \eqref{defnthetaxf}) along with the corresponding set of equivalence or communication classes. We can also compute the reachability relations $R(A)$ (see \eqref{reachabilitycommclassrelan}) among these classes and determine which classes are closed and open. From now on let $\mathcal{C}(A)$ denote the set of closed communication classes under relation $\Theta(A)$.

Suppose $A_1,A_2$ are two subsets of $\mathcal{F}$ such that $ A_1 \subset A_2$. Then certainly $\mathcal{K}_r(  A_1 ) \subset \mathcal{K}_r(  A_2 ) $ and hence $Z_{ij}(A_1) \leq Z_{ij}(A_2)$ for each $i,j$. Consequently $\Theta(A_2)$ has fewer equivalence classes than $\Theta(A_1)$, and each class of $\Theta(A_1)$ is contained in a unique class of $\Theta(A_2)$. Let $C_1$ and $C_2$ be two closed communication classes in the sets $\mathcal{C}(A_1)$ and $\mathcal{C}(A_2)$ respectively. We say that $C_1$ \emph{reaches} $C_2$ if either $C_1 \subset C_2$, or $C_1$ is a subset of an open class $O$ for $\Theta(A_2)$ and the closed class $C_2$ is reachable from $O$ under relation $R(A_2)$. We define a map from $\mathcal{C}(A_1)$ to a subset of $\mathcal{C}(A_2)$ by
\begin{align*}
\Psi_{(A_1,A_2 )}(C_1)=\{ C_2 \in \mathcal{C}(A_2) :  C_1 \textnormal{ reaches } C_2 \}.
\end{align*}
Observe that this map is transitive in the sense that for any $A_1 \subset A_2 \subset A_3 \subset \mathcal{F}$ if $C_2 \in \Psi_{(A_1,A_2 )}(C_1)$ and $C_3 \in \Psi_{(A_2,A_3 )}(C_2)$ then $C_3 \in \Psi_{ (A_1,A_3) }(C_1)$. This transitivity holds because $\mathcal{K}_r(  A_1 ) \subset \mathcal{K}_r(  A_2 ) \subset \mathcal{K}_r(  A_3 )$.

We now develop the notion of a  \emph{Birth-Cascade Tree} (BCT) by generalizing the ideas in Section \ref{sec:freespecies}. This tree is developed in levels or generations indexed by nonnegative integers $l=0,1,2,\dots$, and it is a \emph{directed graph} with nodes in the set
\begin{align}
\label{setofbctnodes}
\mathbb{T} = \{  (C,A) :  A \subset \mathcal{F}  \textnormal{ and } C \in \mathcal{C}(A) \}
\end{align}
and edges in the set $\mathbb{T}  \times \mathbb{T}$. For generation $l=0$ the BCT is initialized by simply adding the nodes $(C, \emptyset)$ for each $C \in \mathcal{C}(\emptyset)$. For any generation $l >1$, the BCT is extended as follows: for each node $(C,A)$ that was added to the BCT in the previous generation $(l-1)$ we compute the set  
\begin{align}
\label{operatorbb}
\mathbb{B}(C,A) = \{ i \in \mathcal{F}: i \notin A \textnormal{ and there exists a } k \in \mathcal{K}_r( C, A  )\textnormal{ such that } i \in \textnormal{supp}( \tilde{\rho}^{\sigma}_k )  \},
\end{align}
where
\begin{align*}
\mathcal{K}_r( C, A  ) = \bigcup_{y \in C}\mathcal{K}_r( y, A  ),
\end{align*}
is the set of those reactions that only consume \emph{free} species in the set $\mathcal{D}^{\sigma}_f(A)$ as reactants and have a positive probability of firing when the dynamics of \emph{bounded} species is in the closed communication class $C$. If $\mathbb{B}(C,A) \neq \emptyset$ then we set $A' = A \cup \mathbb{B}(C,A)$ and for each $C' \in \Psi_{(A,A' )}(C)$ we add the node $(C',A')$ to the BCT along with the directed edge $(C,A) \Rightarrow (C',A')$.

The interpretation of the set $\mathbb{B}(C,A)$ is similar to the interpretation of $\mathbb{B}(A)$ in Section \ref{sec:freespecies} with the only difference being that now the states of \emph{bounded} species can move around in the closed communication class $C$. As the number of \emph{free} species is finite, the BCT will stop growing beyond some generation $l_b$ and this point the construction of BCT is complete and any node $(C,A)$ from which there are no outgoing edges is called a \emph{leaf} of the BCT. For such a leaf node we have
\begin{align}
\label{leafnodecond}
\mathbb{B}(C,A) = \emptyset,
\end{align}
and there exists a set of $(l+1)$ BCT nodes, $\{(C_i,G_i): i = 0,\dots, l \}$ such that $G_0 = \emptyset$, $C_l = C$, $G_l = A$ and the BCT has the following directed path
\begin{align}
\label{defn_bct}
(C_0,G_0) \Rightarrow  (C_1,G_1) \Rightarrow \dots \Rightarrow  (C_l,G_l).  
\end{align}
The transitivity of the map $\Psi$ mentioned above implies that for any $i \in \{ 0,\dots,(l-1) \}$ we have $C_l \in \Psi_{ (G_i,G_l) }(C_i)$. Moreover since the set of reactions $\mathcal{K}_r(G_1),\mathcal{K}_r(G_2),\dots$ is monotonically increasing, if $C_0 \subset C_l$ then $C_i \subset C_l$ for each $i=1,\dots,(l-1)$.

Let $\mathbb{L}$ be the set of all leaf nodes in the BCT. A leaf node $(C,A) \in \mathbb{L}$ is called \emph{minimal} if there is no other leaf node that is strictly contained in $(C,A)$ i.e. if $(C',A') \in \mathbb{L}$, $C' \subset C$ and $A' \subset A$ then $C = C'$ and $A = A'$. Henceforth we denote the set of all minimal leaf nodes by $\mathbb{L}_{ \textnormal{min} }$. Our main result of the paper, Theorem \ref{maintheorem}, will show that under certain conditions these minimal leaf nodes exactly characterize all the irreducible state-spaces for the network $\tilde{\mathcal{N} }^{\sigma}$ in the infinite state-space $\tilde{\mathcal{E} }^{\sigma}_0 = \mathcal{E}_b^{\sigma} \times \N^{d_f}_0$. However before we state this result we need an appropriate notion of death-cascades and we need to impose some stoichiometric restrictions on the network as in Section \ref{sec:freespecies}.

Pick a minimal leaf node $(C,A) \in \mathbb{L}_{ \textnormal{min} }$ and let $\tilde{\mathcal{N} }^{\sigma}(C,A)$ be the network formed by removing all the \emph{free} species outside $\mathcal{D}^{\sigma}_f(A)$ and discarding all the reactions outside $\mathcal{K}_r( C, A  )$. Moreover we also replace the finite state-space $\mathcal{E}_b^{\sigma}$ of the \emph{bounded} species with its subset $C$. This replacement is feasible because \eqref{leafnodecond} holds ensuring that reactions in $\mathcal{K}_r( C, A  )$ cannot produce any \emph{free} species outside $\mathcal{D}^{\sigma}_f(A)$, and $C$ is a closed communication class for the states of \emph{bounded} species under these reactions. Let $\tilde{\mathcal{N} }^{\sigma}_{ \textnormal{inv} }(C,A)$ be the inverse of network $\tilde{\mathcal{N} }^{\sigma}(C,A)$ (see Section \ref{sec:defninverse}). We compute the BCT for network $\tilde{\mathcal{N} }^{\sigma}_{ \textnormal{inv} }(C,A)$ and refer to it as the \emph{Death-cascade Tree} (DCT) for network $\tilde{\mathcal{N} }^{\sigma}(C,A)$. Let $\mathbb{L}_d(C,A)$ denote the set of all leaf nodes of this DCT. Note that for any leaf node $(C',A') \in \mathbb{L}_d(C,A)$ we certainly have $C' \subset C$ and $A' \subset A$. We say that $(C,A)$ is \emph{death-exhaustive} if there exists a leaf node $(C',A') \in \mathbb{L}_d(C,A)$ such that $A' = A$. If such a leaf node exists then it is the only leaf node in $ \mathbb{L}_d(C,A)$ since $C'= C$. This is shown in the proof of Lemma \ref{lem:deathexhaustive}.

Consider a minimal leaf node $(C,A) \in \mathbb{L}_{ \textnormal{min} }$ and let $\tilde{\mathcal{K}}_p(A)$ be the set of all those reactions that do not involve any \emph{bounded} species and do not produce any \emph{free} species outside the set $\mathcal{D}^{\sigma}_f(A)$, i.e. 
\begin{align}
\label{defn_noboundedfreecons}
\tilde{\mathcal{K}}_p(A) = \{ k=1,\dots,K: \textnormal{supp}(\bar{\nu}^{\sigma}_k) =  \textnormal{supp}(\bar{\rho}^{\sigma}_k) =\emptyset \ \  \textnormal{and} \ \  \textnormal{supp}( \tilde{\rho}^{\sigma}_k)   \subset A  \}.
\end{align}
Let $\hat{S}_{\sigma}= \textnormal{Proj}( \tilde{S}_{\sigma}, d_b+1,d_b+ d_f )$ be the $d_f \times K$ matrix formed by the last $d_f$ rows of the $(d_b +d_f) \times K$ stoichiometry matrix $\tilde{S}_{\sigma}$ for network $\tilde{\mathcal{N} }^{\sigma}$ and let $\hat{S}_{\sigma}[  \tilde{\mathcal{K}}_p(A) ]$ be the matrix formed by restricting $\hat{S}_{\sigma}$ to only those columns that correspond to reactions in $\tilde{\mathcal{K}}_p(A)$. Recall that the set of all \emph{free} species is $\mathcal{D}_f = \{ \sigma(d_b+i) : i  \in \mathcal{F} \}$. For any $i \in \mathcal{F}$ we say that the \emph{free} species $\sigma(d_b+ i) \in \mathcal{D}_f$ is \emph{singularly-degradable} w.r.t. $A$ if \eqref{newcondition0} is satisfied with matrix $\hat{S}_{\sigma}[  \tilde{\mathcal{K}}_p(A) ]$. To obtain the irreducible state-space corresponding to the leaf node $(C,A) \in \mathbb{L}_{ \textnormal{min} }$, we shall permute the network according to the permutation $\sigma_3(A) : \mathcal{D} \to \mathcal{D}$ which ensures that the copy-numbers of all the \emph{free} species in $\mathcal{D}^{\sigma}_f(A)$ occupy the components $(d_b+1),\dots, (d_b+ |A|)$ of the state vector. Such a permutation can be defined by setting it to be the same as $\sigma_2$ on the set $\{1,\dots,d_b\}$ and letting its image on the set $\{ (d_b+1),\dots, (d_b+d_f) \}$ be determined by \eqref{permsigma3} with $\mathcal{B} = \mathcal{D}^{\sigma}_f(A)$. We now state the main result of our paper.
\begin{theorem}
\label{maintheorem}
Suppose Assumption \ref{assmp:noofconsrels2} holds, network $\tilde{\mathcal{N} }^{\sigma}$ satisfies Assumption \ref{massacttype} and its state-space is $\tilde{\mathcal{E} }^{\sigma}_0 = \mathcal{E}_b^{\sigma} \times \N^{d_f}_0$. Let $\mathbb{L}$ and $\mathbb{L}_{ \textnormal{min} }$ be the sets of all leaf nodes and minimal leaf nodes for the BCT of the network. Then we have the following: 
\begin{itemize}
\item[(A)] Pick any minimal leaf node $(C,A) \in \mathbb{L}_{ \textnormal{min} }$. Suppose that this leaf node is death-exhaustive and all the free species in $\mathcal{D}^{ \sigma }_f(A)$ are singularly-degradable w.r.t. $A$. Then the set $C \times \N_0^{ |A| } \times \{ \mathbf{0} \}$ is an irreducible state-space for network $\tilde{\mathcal{N} }^{\sigma}$ permuted according to permutation $\sigma = \sigma_3(A)$, where $\mathbf{0}$ denotes the $(d_f - |A|)$-dimensional vector of zeros. Additionally if all the free species in the set $\mathcal{D}_f \cap ( \mathcal{D}^{ \sigma }_f(A))^c $ are also singularly-degradable w.r.t. $A$ then $C \times \N_0^{ |A| } \times \{ \mathbf{0} \}$ is {\bf the only} irreducible state-space for network $\tilde{\mathcal{N} }^{\sigma}$ that can have a nonempty intersection with the set $C \times \N_0^{d_f}$. 

\item[(B)] Suppose that all leaf nodes are minimal (i.e. $ \mathbb{L}  =  \mathbb{L}_{ \textnormal{min} }$), each leaf node $(C,A) \in \mathbb{L}$ is death exhaustive and all the free species in $\mathcal{D}_f$ are singularly-degradable w.r.t. $A$. Then up to the relabeling of free species in $\mathcal{D}_f$, all the irreducible state-spaces for network $\tilde{\mathcal{N} }^{\sigma}$ are given by
\begin{align*}
C \times \N_0^{ |A| } \times \{ \mathbf{0} \} \quad \textnormal{ for all }\quad  (C,A) \in  \mathbb{L}.
\end{align*}
\end{itemize}
\end{theorem}
\begin{proof}
This theorem is proved in the Appendix and it generalizes Proposition \ref{prop:no_bdd_species} by taking the dynamics of \emph{bounded} species into account.
\end{proof}

We end this section with a couple of useful remarks.
\begin{remark}
\label{rem:bctissinglepath}
Commonly in Systems Biology networks (see Section \ref{sec:examples}) there is only one leaf node $(C,A)$ and there is only one class $C_0$ in $\mathcal{C}( \emptyset )$ such that $C_0 \subset C$. Hence all leaf nodes are minimal (because $ \mathbb{L}  =  \mathbb{L}_{ \textnormal{min} } = \{ (C,A)\}$) and one can see that the BCT reduces to a single path $$(C_0,\emptyset) \Rightarrow  (C_1,G_1) \Rightarrow \dots \Rightarrow  (C_{l_b},G_{l_b} ) = (C,A).$$ 
In such a scenario we can talk about the birth-cascades for the network in the same way as in Section \ref{sec:freespecies}. In particular $\mathcal{D}^\sigma_f(G_l)$ is the set of \emph{free} species that belong to the $l$-th birth cascade. If the leaf node $(C,A)$ is death-exhaustive as well, then the Death-Cascade Tree (DCT) for network $\tilde{\mathcal{N} }^{\sigma}(C,A)$ is also a single path and so we can also talk about the death-cascades for the network in the same way as in Section \ref{lem:deathexhaustive}.
\end{remark}

\begin{remark}
\label{rem:originalclasses}
Suppose from part (A) of Theorem \ref{maintheorem} we obtain an irreducible state-space for network $\tilde{\mathcal{N} }^{\sigma}$. Then the corresponding irreducible state-space for the original network $\mathcal{N}$ can be easily identified using Propositions \ref{mainprop_reducednetwork} and \ref{prop:equivstatespacepermutation}. Similarly if {\bf all} the irreducible state-spaces for network $\tilde{\mathcal{N} }^{\sigma}$ can be found using part (B) of Theorem \ref{maintheorem}, then we can identify all the irreducible state-spaces for the original network $\mathcal{N}$ using Propositions \ref{mainprop_reducednetwork} and \ref{prop:equivstatespacepermutation}.   
\end{remark}

\section{ Algorithms} \label{sec:algorithms}

The aim of this section is to provide detailed algorithmic descriptions of various procedures that can be used to apply the results in this paper. We start with network $\mathcal{N} = ( \mathcal{V}, \mathcal{O}, \Lambda)$ with $d$ species in the set $\mathcal{D}=\{1,\dots,d\}$ and $K$ reactions of the form (see Section \ref{sec:prelim}). We assume that this network satisfies Assumption \ref{massacttype} and its conservation data is $(\Gamma,c)$. Our first goal is to find a decomposed state-space $\mathcal{E}^{\sigma}_0$ of the form \eqref{defn_dss} under some suitably constructed permutation $\sigma : \mathcal{D} \to \mathcal{D}$. This is accomplished in the method $\Call{FindDecomposedStateSpace}{\cdot}$ (see Algorithm \ref{finddecompss}). This method starts by identifying the \emph{bounded} species and finding their optimal state-space (see Section \ref{sec:findingboundedstatespace}). It then computes the numbers of \emph{free} ($d_f$) and \emph{restricted} ($d_r$) species, and if $d_r>0$, then it tries to classify the \emph{unbounded} species into \emph{free} and \emph{restricted} species in such a way that the associated affine map $\phi$ (see \eqref{defn_map_phi}) is compatible with the reaction network (see Assumption \ref{affinefunction}). Such a compatible map is found by iterating through the set $\mathcal{I}_f$ (see \eqref{defn_indexsetfreespecies}), whose elements determine all plausible partitions of the set of \emph{unbounded} species ($\mathcal{D}_u$) into the sets of \emph{free} ($\mathcal{D}_f$). The outputs returned by $\Call{FindDecomposedStateSpace}{\cdot}$ are the permutation map $\sigma$ and the decomposed state-space $\mathcal{E}^{\sigma}_0$.

If the network has any \emph{restricted} species (i.e. $d_r > 0$), then we construct the reduced network $\tilde{ \mathcal{N} }^{ \sigma } =   ( \tilde{ \mathcal{V} }, \tilde{ \mathcal{O}}, \tilde{ \Lambda})$ along with the associated permutation $\sigma = \sigma_2$ (see \eqref{defn_permsigma2}), by systematically removing the \emph{restricted} species as described in Section \ref{sec:netreductionbyelimination}. Assuming that a compatible affine map $\phi$ was discovered by $\Call{FindDecomposedStateSpace}{\cdot}$, the reduced network $\tilde{ \mathcal{N} }^{ \sigma }$ will also satisfy Assumption \ref{massacttype}, which is necessary for the process of identifying all the irreducible state-spaces. This process is accomplished by our next method $\Call{FindirreducibleStateSpaces}{\cdot}$ (see Algorithm \ref{findrednet}), which works by first constructing the Birth-Cascade Tree (BCT) for network $\tilde{ \mathcal{N} }^{ \sigma }$ and then examining all the minimal leaf nodes of this BCT (see Section \ref{sec:freebddspecies}). The construction of BCT is performed by method \\ $\Call{ConstructBCT}{\cdot}$ (see Algorithm \ref{constructbct}) which returns the set of leaves $\mathbb{L}$ of the BCT.

For each minimal leaf node $(C,A) \in \mathbb{L}_{ \textnormal{min} }$ the method $\Call{FindirreducibleClasses}{\cdot}$ checks if the set $C \times \N^{|A|}_0 \times \{\bf 0\}$ is an irreducible state-space for network $\tilde{ \mathcal{N} }^{ \sigma }$ under the permutation $\sigma = \sigma_3(A)$ (see Section \ref{sec:freebddspecies}) along with whether all the \emph{free} species are singularly-degradable w.r.t. $A$. If both of these are true, then part (A) of Theorem \ref{maintheorem} tells us that $C \times \N^{|A|}_0 \times \{\bf 0\}$ is the only irreducible state-space which can have common elements with the set $C \times \N^{d_f}_0$. If this is the case for all minimal leaf nodes and $ \mathbb{L}  =  \mathbb{L}_{ \textnormal{min} }$, then using $\Call{FindirreducibleClasses}{\cdot}$ we recover all the irreducible state-spaces for our network due to part (B) of Theorem \ref{maintheorem}. Remark \ref{rem:originalclasses} tells us how these irreducible state-spaces for network $\tilde{ \mathcal{N} }^{ \sigma }$ can be mapped to the corresponding irreducible state-space for the original network $\mathcal{N}$.

\begin{algorithm}[H]  
\caption{Finds a decomposed state-space $\mathcal{E}^{\sigma}_0$ of the form \eqref{defn_dss} for a network $\mathcal{N} = ( \mathcal{V}, \mathcal{O}, \Lambda)$ with conservation data $(\Gamma,c)$}      
 \label{finddecompss}
 \begin{algorithmic}[1]
\Require Network $\mathcal{N}$ satisfies Assumption \ref{massacttype} 
\Function{FindDecomposedStateSpace}{$ \mathcal{N}, \Gamma,c $}  
\State For each species $i \in \mathcal{D}$ solve the LP \eqref{LPP2} to compute $b_i$. 
\State Set $\mathcal{D}_b = \{i \in \mathcal{D} : b_i <\infty \}$ and $\mathcal{D}_u = \{i \in \mathcal{D} : b_i  = \infty \}$ to be the sets of \emph{bounded} and \emph{unbounded} species respectively. Also set $d_b = | \mathcal{D}_b |$ and $d_u = |\mathcal{D}_u|$.
\State Select the permutation map $\sigma_1: \mathcal{D} \to \mathcal{D}$ according to \eqref{defn_permsigma1}. 
\State Set $\sigma = \sigma_1$ and construct the permuted network $\mathcal{N}^{\sigma}$ (see Section \ref{sec:srnundersigma}) along with its conservation data $( \Gamma_{\sigma} ,c )$.
\State Compute the finite set $\mathcal{E}^\sigma_b$ according to \eqref{defnEsigmab}. This is the state-space for \emph{bounded} species in $\mathcal{D}_b$.
\State Evaluate $d_f$ and $d_r$ according to \eqref{defnnumfreerestricted}.
\If{ $d_r = 0$ }   \State  {\bf Output:} The decomposed state-space is $\mathcal{E}^{\sigma}_0 = \mathcal{E}^\sigma_b  \times \N^{d_f}_0$ under permutation $\sigma$.
\Else{ } 
\State Verify Assumption \ref{assmp:noofconsrels2} and if it does not hold then \Return {\bf QUIT}
\State Compute the set $\mathcal{I}_f$ given by \eqref{defn_indexsetfreespecies}.
 \ForAll{ $I \in \mathcal{I}_f$ }
\State Set $\mathcal{D}_f = \{ \sigma(d_b+ i) : i \in I \}$ and $\mathcal{D}_r = \{ \sigma(d_b+ i) : i \in I^c \}$ to be the sets of \emph{free} and \emph{restricted} species respectively.
\State Select the permutation map $\sigma_2: \mathcal{D} \to \mathcal{D}$ according to \eqref{defn_permsigma2}.
\State Set $\sigma = \sigma_2$ and construct the permuted network $\mathcal{N}^{\sigma}$ along with its conservation data $( \Gamma_{\sigma} ,c )$
\State Define the affine map $\phi: \R^{d_f} \to \R^{d_r}$ according to \eqref{defn_map_phi}.
\If{ $\phi$ satisfies Assumption \ref{affinefunction} }
\State  {\bf Output:} Map $\phi$ is compatible with network $\mathcal{N}^{\sigma}$.
\State Exit the for-loop and \Goto{marker1}
\EndIf
 \EndFor
\State If $\phi$ does not satisfy Assumption \ref{affinefunction} then  \Return {\bf QUIT}. Otherwise let $\Phi$ be the graph of $\phi$ given by \eqref{defngraphphi} \label{marker1}.
\State {\bf Output:} The decomposed state-space is $\mathcal{E}^{\sigma}_0 = \mathcal{E}^\sigma_b  \times \Phi$ under permutation $\sigma$
\EndIf
\EndFunction
\end{algorithmic}
\end{algorithm}

\begin{algorithm}[H]  
\caption{Identifies the irreducible state-spaces for network $\tilde{ \mathcal{N} }^{ \sigma } =   ( \tilde{ \mathcal{V} }, \tilde{ \mathcal{O}}, \tilde{ \Lambda})$ in the state-space $\mathcal{E}^\sigma_b  \times \N^{d_f}_0$}           
 \label{findrednet}
 \begin{algorithmic}[1]
 \Require Network $\tilde{ \mathcal{N} }^{ \sigma }$ only has \emph{bounded} and \emph{free} species in the sets $\mathcal{D}_b =\{ \sigma(i): i =1,\dots, d_b \}$ and $\mathcal{D}_f =\{ \sigma(d_b+i): i =1,\dots, d_f \}$ respectively. 
 \Function{FindirreducibleStateSpaces}{$\tilde{ \mathcal{N} }^{ \sigma }, \mathcal{E}^\sigma_b  \times \N^{d_f}_0$}
\State Let $\mathbb{L} = \Call{ConstructBCT}{\tilde{ \mathcal{N} }^{ \sigma }, \mathcal{E}^\sigma_b  \times \N^{d_f}_0}$
\State Identify the set of minimal leaf nodes $\mathbb{L}_{ \textnormal{min} } \subset \mathbb{L}$
\State Initialize $\mathbb{L}' = \emptyset$.
 \ForAll{ $(C,A) \in \mathbb{L}_{ \textnormal{min} } $ }
 \State Let $\tilde{\mathcal{N} }^{\sigma}(C,A)$ be the network formed by removing all the \emph{free} species outside $\mathcal{D}^{\sigma}_f(A)$ and discarding all the reactions outside $\mathcal{K}_r( C, A  )$ (see Section \ref{sec:freebddspecies}).
 \State Let $\tilde{\mathcal{N} }^{\sigma}_{ \textnormal{inv} }(C,A)$ be the inverse of network $\tilde{\mathcal{N} }^{\sigma}(C,A)$ (see Section \ref{sec:defninverse}).
 \State Set $\mathbb{L}_d(C,A) = \Call{ConstructBCT}{\tilde{\mathcal{N} }^{\sigma}_{ \textnormal{inv} }(C,A), C \times \N^{d_f}_0}$
 \State Let $\tilde{\mathcal{K}}_p(A)$ be the set of reactions given by \eqref{defn_noboundedfreecons}.
\State Construct the matrix $\hat{S}_{\sigma}[  \tilde{\mathcal{K}}_p(A) ]$ by restricting the stoichiometry matrix for network $\tilde{ \mathcal{N} }^{ \sigma }$ to only its last $d_f$ rows and only the columns that correspond to reactions in $\tilde{\mathcal{K}}_p(A)$.
 \State Compute the set of \emph{free} species $\mathcal{D}_{sd}(A) \subset \mathcal{D}_f$ that are singularly-degradable w.r.t. $A$ i.e.
 \begin{align*}
\mathcal{D}_{sd}(A) =\{ \sigma(d_b+i ): i=1,\dots,d_f \textnormal{    and    } \eqref{newcondition0} \textnormal{ is satisfied with matrix } \hat{S}_{\sigma}[  \tilde{\mathcal{K}}_p(A) ] \}.
\end{align*}
\If{ $\mathbb{L}_d(C,A) =\{(C,A)\}$ \textbf{ AND } $\mathcal{D}_{sd}(A) \subset \mathcal{D}^\sigma_f :=\{ \sigma(d_b+i) : i \in A\}$}
\State Set $\sigma = \sigma_3(A)$ (see Section \ref{sec:freebddspecies}) and construct the permuted network $\tilde{ \mathcal{N} }^{ \sigma }$ (see Section \ref{sec:srnundersigma}).
\State {\bf Output:} The set $C \times \N^{|A|}_0 \times \{ \bf{0} \}$ is an irreducible state-space for network $\tilde{ \mathcal{N} }^{ \sigma }$ under permutation $\sigma$. \label{irreducibleclassesfound}
\If{$\mathcal{D}_{sd}(A) = \mathcal{D}_f$}
\State Update $\mathbb{L}' \gets \mathbb{L}'  \cup \{ (C,A) \}$.
\EndIf
 \EndIf
 \EndFor
 \If{ $\mathbb{L} =\mathbb{L}_{ \textnormal{min} }=  \mathbb{L}' $}
 \State {\bf Output:} The irreducible state-spaces found in step \ref{irreducibleclassesfound} are the {\bf only} irreducible state-spaces for network $\tilde{ \mathcal{N} }^{ \sigma }$ in the state-space $\mathcal{E}^\sigma_b  \times \N^{d_f}_0$ up to the relabeling of \emph{free} species in $\mathcal{D}_f$.
 \EndIf
 \EndFunction
\end{algorithmic}
\end{algorithm}

\begin{algorithm}[H]  
\caption{Constructs the Birth-Cascade Tree (BCT) for network $\tilde{ \mathcal{N} }^{ \sigma } =   ( \tilde{ \mathcal{V} }, \tilde{ \mathcal{O}}, \tilde{ \Lambda})$ with state-space $\mathcal{E}^\sigma_b  \times \N^{d_f}_0$. This BCT is a directed graph with nodes in the set $\mathbb{T}$ given by \eqref{setofbctnodes}.}           
 \label{constructbct}
 \begin{algorithmic}[1]
 \Require Network $\tilde{ \mathcal{N} }^{ \sigma }$ only has \emph{bounded} and \emph{free} species in the sets $\mathcal{D}_b =\{ \sigma(i): i =1,\dots, d_b \}$ and $\mathcal{D}_f =\{ \sigma(d_b+i): i =1,\dots, d_f \}$ respectively. 
\Function{ConstructBCT}{$\tilde{ \mathcal{N} }^{ \sigma }, \mathcal{E}^\sigma_b  \times \N^{d_f}_0$}
\State Initialize $l=0$, $\mathbb{L}  = \emptyset$ and $G_0= \{ (C,  \emptyset):  \textnormal{ for each } C \in \mathcal{C}( \emptyset )  \}$. 
\State Add each node in $G_0$ to the BCT.
\Repeat
\State Set $G_{l+1} = \emptyset$ 
\ForAll{$(C,A) \in G_l $}
\State Compute the set $B = \mathbb{B}(C,A)$ where the operator $\mathbb{B}$ is defined by \eqref{operatorbb}.
\If{ $B \neq \emptyset$}
	\State Set $A' = A \cup B$. 
	\ForAll{$C' \in \Psi_{(A,A' )}(C)$}
	\State Update $G_{l+1} \gets G_{l+1} \cup \{ (C',A') \}$.
	\State Add the node $(C',A')$ to the BCT along with the directed edge $(C,A) \Rightarrow (C',A')$.
	\EndFor
\Else{}
\State Update $ \mathbb{L} \gets  \mathbb{L} \cup \{ (C,A) \}$ 		
\EndIf
\EndFor
\State Update $l \gets l+1$
\Until{ $G_l = \emptyset$}
 \State {\bf Output:} The set of leaf nodes of the BCT is $\mathbb{L}$.
 \EndFunction
\end{algorithmic}
\end{algorithm}

\section{Examples} \label{sec:examples}

In this section we illustrate our results using several networks from Systems Biology. We start by considering a family of simple Gene-Expression networks which illustrate various theoretical ideas developed in this paper (see Section \ref{ex:ge}). Next we consider a couple of Circadian Clock models (Section \ref{ex:cc}) and a Bacterial Heat-Shock response model (Section \ref{ex:hsr}). The networks underlying these models have many species and reactions, and we discuss how the state-space analysis presented in this paper can help in understanding network design as well as the long-term behavior of the associated stochastic models. In Section \ref{ex:toxin} we provide a simple example of a Toxin-Antitoxin network to demonstrate how our results can aid the automatic, real-time application of the \emph{quasi-stationary} approximation to speed up stochastic simulations of the multiscale network. Finally in Section \ref{ex:productdistr}, we present a class of networks, where our analysis along with certain existing results, provide the exact stationary distribution for the stochastic model.

In all the examples, the reactions propensity functions ($\lambda_k$-s) are assumed to have the \emph{mass-action} form \eqref{massactionkinetics} unless otherwise stated. For correct interpretation of our results we provide a ``Species Chart" that encodes the names of network species into the notation used in our paper i.e. $\mathbf{S}_1,\mathbf{S}_2,\dots$. Throughout this section the copy-number of species $\mathbf{S}_i$ is denoted by $x_i$. Often we would need to permute the network (see Section \ref{sec:srnundersigma}) for our analysis. The permutation $\sigma$ for which the final results are shown is given as a vector $\sigma = ( \sigma(1),\dots, \sigma(d) )$, and under this permutation the $i$-th component of the state-vector corresponds to the copy-number of species $\mathbf{S}_{ \sigma(i) }$.

\subsection{Family of Gene-Expression networks} \label{ex:ge}

We now consider several variants of the simple Gene-Expression network given in \cite{MO}. In these networks there is a Gene ($G$), which is responsible for the transcription of \emph{massenger} RNA or mRNA ($M$) molecules, that later translate into the Protein ($P$) of interest. Both mRNA and Protein molecules degrade spontaneously and as in \cite{Kepler} we allow for the Gene to switch between an \emph{active} ($G_{ \textnormal{on} }$) and an \emph{inactive} ($G_{ \textnormal{off} }$) state. For now we assume that the transcription of mRNA is only possible in the active gene state $G_{ \textnormal{on} }$.

The first model (see Network 0 in Figure \ref{figure:ge}{\bf A}) we examine consists of $4$ species (see the Species Chart in Table \ref{secchart:ge}) and $6$ reactions displayed in Figure \ref{figure:ge}{\bf B}. For this network there is only one (independent) conservation relation $\gamma = (1,1,0,0)$ which is semi-positive, and so there are two \emph{bounded} species ${\bf S}_1$ and ${\bf S}_2$, whose copy-numbers $x_1$ and $x_2$ are ``locked" in the relation $x_1 + x_2 = c = 1$, which simply says that the total number of Gene copies $c$ is conserved by the dynamics. 
The other two species ${\bf S}_3 $ and ${\bf S}_4 $ are not involved in this conservation relation, and are hence \emph{free} species. 

\begin{table}[h!]
\begin{align*}
\begin{array}{|l|l|} \hline
\multicolumn{2}{|c|}{\textnormal{Species Chart} } \\ \hline
{\bf S}_1 = G_{ \textnormal{off} } &  {\bf S}_3 = M \\
{\bf S}_2 = G_{ \textnormal{on} }  & {\bf S}_4 = P \\ \hline
\end{array}
\end{align*}
\caption{Species chart for Gene-expression networks in Section \ref{ex:ge}. $G_{ \textnormal{on} }$ and $G_{ \textnormal{off} }$ denote the gene in active and inactive states. $M$ denotes the mRNA and $P$ denotes the corresponding Protein. }
\label{secchart:ge}
\end{table}

\begin{figure}[ht!]
\centering
\frame{\includegraphics[width=0.92\textwidth]{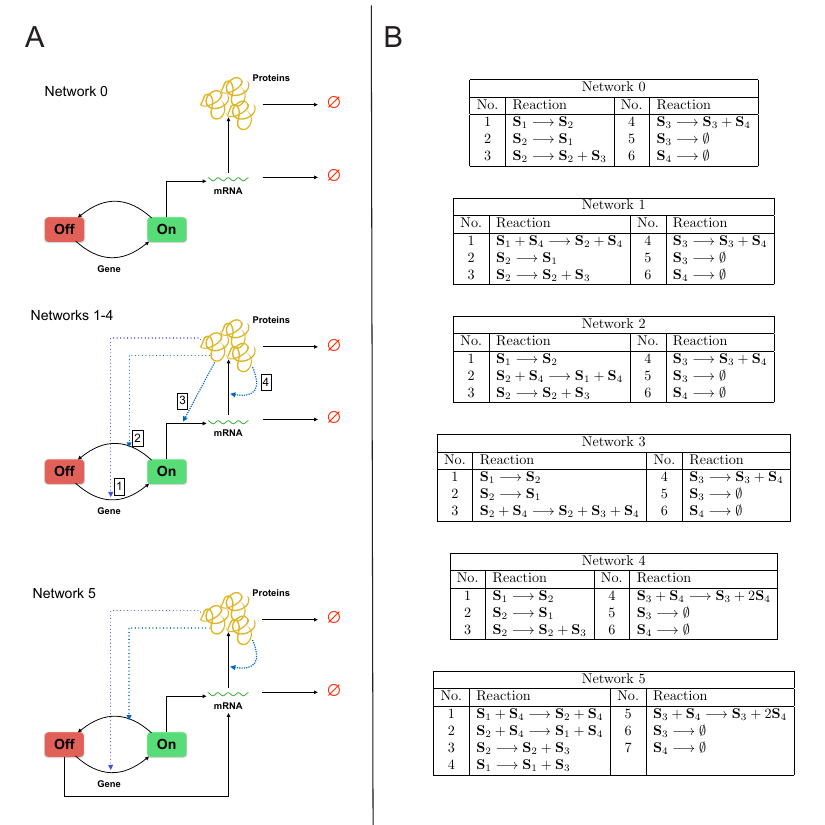}}
\caption{The Gene-expression networks in Section \ref{ex:ge}. The networks are depicted in panel {\bf A} and the corresponding reactions sets are tabulated in panel {\bf B}. The \emph{feedback} interactions are shown with dotted arrows and for each of the Networks 1-4, the corresponding feedback interaction is clearly marked in the middle figure of panel {\bf A}.}
\label{figure:ge}
\end{figure}

Using Algorithms \ref{finddecompss} and \ref{findrednet}, we can identify the decomposed state-space and the irreducible state-spaces for Network 0 (see Table \ref{stspaceanalysis:ge}). One can verify that the state-space $\mathcal{E}^\sigma_b = \{ (1,0), (0,1) \} $ for \emph{bounded} species is a closed communication class in $\mathcal{C}( \emptyset )$ and so the Birth-Cascade Tree (BCT) for this network will only consist of one path whose nodes give us the birth-cascades of \emph{free} species (see Remark \ref{rem:bctissinglepath}). It is easy to see that the length of this path is $3$, the first birth-cascade is $\{ \mathbf{S}_3 \}$ and the second one is $\{ \mathbf{S}_4 \}$. Observe that these two cascades correspond naturally to the two stages in the network, transcription and translation, thereby suggesting that our cascade construction approach is a useful tool for understanding the structure of large Systems Biology networks. This point is reinforced by examples in Sections \ref{ex:cc} and \ref{ex:hsr}. We can check that the only leaf node of BCT is death-exhaustive and recover the death-cascades for the network (see Remark \ref{rem:bctissinglepath}). In Network 0, both \emph{free} species degrade spontaneously and so they are both singularly-degradable w.r.t. any set. Other conditions of Theorem \ref{maintheorem} also hold and so we can use this result to conclude that the full state-space $\mathcal{E}^\sigma_b \times \N^2_0$ is the unique irreducible state-space for Network 0. This concludes our state-space analysis for Network 0.

We now create Networks 1-4 by adding feedback from protein molecules (see \cite{Kepler}) to various reactions in Network 0 (see Figure \ref{figure:ge}{\bf A}). This feedback is added to reaction $k$ by multiplying its original mass-action propensity function $\lambda_k(x)$ by a Hill-type factor of the form
\begin{align}
\label{hilltypefeedbackform}
 \theta_{\textnormal{fb}} \frac{ x^n_4 }{ c  +x^n_4 },
\end{align}
where $x_4$ is the number of protein molecules, and $\theta_{\textnormal{fb}}, c$ and $n$ are strictly positive parameters. Note that $\lambda_k(x) = 0$ if $x_4 = 0$ which means that if we just multiply the propensity function for reaction $k$ in Network 0 by this factor \eqref{hilltypefeedbackform}, then the modified network will not satisfy Assumption \ref{massacttype} which is required for our analysis. However we can circumvent this problem by simply adding a molecule of species $\mathbf{S}_4$ to both sides of reaction $k$, changing it from $A \longrightarrow B$ to $A + \mathbf{S}_4 \longrightarrow B + \mathbf{S}_4$. This simple trick ensures that the modified network satisfies Assumption \ref{massacttype} and its dynamics remains unaffected as the reaction stoichiometry is the same. Incorporating the feedback mechanism this way, we list the reactions for Networks 1-4 in Figure \ref{figure:ge}{\bf B} and provide the results from the state-space analysis of these networks in Table \ref{stspaceanalysis:ge}. Observe that in each of these cases a unique irreducible state-space is guaranteed by Theorem \ref{maintheorem}, but this irreducible state-space varies among networks. Also note that in all the Networks 0-4, the reactions stoichiometries are the same and consequently their decomposed state-space is identical.

\begin{table}[h!]
\begin{align*}
\begin{array}{|c|c|l|} \hline
\textnormal{Network} &  \textnormal{Decomposed State-Space} & \textnormal{Irreducible state-spaces} \\ \hline
0 & \mathcal{E}^\sigma_b \times \N^2_0 & \mathcal{E}^\sigma_b \times \N^2_0 \\
1 &\mathcal{E}^\sigma_b \times \N^2_0  & \{ (1,0,0,0) \}\\
2 & \mathcal{E}^\sigma_b \times \N^2_0 &  \mathcal{E}^\sigma_b \times \N^2_0  \\
3 &\mathcal{E}^\sigma_b \times \N^2_0  & \mathcal{E}^\sigma_b \times \{ (0,0) \} \\
4 &\mathcal{E}^\sigma_b \times \N^2_0  &\mathcal{E}^\sigma_b \times \N_0 \times  \{ 0 \} \\ \hline
5 &\mathcal{E}^\sigma_b \times \N^2_0  &\{(0,1)\} \times \N_0 \times  \{ 0 \} \  \textnormal{ and } \\
& &  \{(1,0)\} \times \N_0 \times  \{ 0 \} \\
\hline
\end{array}
\end{align*}
\caption{State-space analysis for Gene-expression networks in Section \ref{ex:ge}. All results reported w.r.t the identity permutation $\sigma = (1,2,3,4)$ and $\mathcal{E}^\sigma_b = \{ (1,0), (0,1) \}$. For all the five networks \emph{all} irreducible state-spaces were identified by Theorem \ref{maintheorem}.}
\label{stspaceanalysis:ge}
\end{table}

We now consider another gene-expression model (see Network 5 in Figure \ref{figure:ge}) by allowing the transcription of mRNA in the inactive gene state $G_{ \textnormal{off} }$, and having feedback from protein molecules to the translation reaction as well as the gene-switching reactions. In this network also the decomposed state-space is the same as before but there are exactly two irreducible state-spaces, $\{(0,1)\} \times \N_0 \times \{0\}$ and $\{(1,0)\} \times \N_0 \times \{0\}$ according to Theorem \ref{maintheorem} (see Table \ref{stspaceanalysis:ge}).   

We end this section by remarking that if all the Hill-type feedback factors have a positive \emph{basal} level, i.e. instead of \eqref{hilltypefeedbackform} we have
\begin{align*}
 \theta_1 + \theta_{\textnormal{fb}} \frac{ x^n_4 }{ c  +x^n_4 },
\end{align*}
for some $\theta_1 > 0$, then we do not need to add species $\mathbf{S}_4$ to both sides of feedback reactions to ensure that the modified networks satisfy Assumption \ref{massacttype}. Indeed now Networks 1-4 will satisfy this assumption with the original set of reactions (same as Network 0) and hence they all will have the full state-space $\mathcal{E}^\sigma_b \times \N^2_0$ as their unique irreducible state-space. The same holds true for Network 5 even though it has an extra reaction. This illustrates the significance of Assumption \ref{massacttype} in determining the irreducible state-spaces.

\subsection{Two Circadian Clock models} \label{ex:cc}

Circadian clocks are molecular time-keeping devices that coordinate many physiological processes in living organisms \cite{Reppert1}. These clocks generate oscillatory rhythms that are usually \emph{entrained} to the periodic cues provided by the day-light cycles \cite{Bagheri,Herzog1}. We consider two circadian clock networks in this section and prove using our analysis that there exists a unique irreducible state-space for both these models, thereby indicating that the stationary distributions for the associated stochastic models is unique. The existence of these stationary distribution can be checked using the techniques in \cite{GuptaPLOS} and hence the stochastic models for both these networks are ergodic (see Section \ref{sec:intro}). Therefore under constant inputs the individual stochastic trajectory of a single circadian clock may be oscillatory, but the mean trajectories, corresponding to the bulk or population-level behavior of several uncoupled and identical circadian clocks, {\bf cannot} be oscillatory due to \eqref{ergodicconvergenceoff}. This is consistent with both computational \cite{Bagheri,GuptaPLOS} and experimental \cite{Welsh2004} observations in the existing literature. In a recent paper it is argued that this loss of oscillatory activity at the population-level plays an important role in ensuring that the entrainment to periodic cues is robust at the population-level \cite{GuptaEntrainment}.

The first circadian clock model we consider is from Vilar et al. \cite{Vilar} and it is depicted in Figure \ref{figure:ccvilar}{\bf A}. It consists of gene-expression modules for an \emph{activator} protein $A$ and a \emph{repressor} protein $R$ which sequesters the activator protein $A$ by forming an inactive complex $AR$. The activator protein $A$ can enhance the transcription of both the mRNAs ($M_A$ and $M_R$) by binding to the promoter regions of the activator gene $D_A$ and the repressor gene $D_R$. When the promoter region is occupied the activator and the repressor genes are denoted by $D'_A$ and $D'_R$ respectively. The overall network consists of $9$ species (see the Species Chart in Table \ref{secchart:ccvilar}) and $16$ reactions (see Table \ref{rxnchart:ccvilar}). The results from our state-space analysis on this network are presented in Figure \ref{figure:ccvilar}{\bf B} and they show that the network has a unique irreducible state-space which coincides with its decomposed state-space. For this network the situation of Remark \ref{rem:bctissinglepath} applies and so we can arrange all the \emph{free} species into birth and death cascades (Figure \ref{figure:ccvilar}{\bf C}). These cascades correspond naturally to different stages in the network.

\begin{table}[h!]
\begin{align*}
\begin{array}{|l|l|} \hline
\multicolumn{2}{|c|}{\textnormal{Species Chart} } \\ \hline
{\bf S}_1 = M_{ A } &  {\bf S}_6 = G^b_{ A } \\
{\bf S}_2 = A  & {\bf S}_7 = G^u_{ A }  \\
{\bf S}_3 = M_R &  {\bf S}_8 =  G^b_{ R}  \\
{\bf S}_4 = R &  {\bf S}_9 =  G^u_{ R} \\
{\bf S}_5 = AR &    \\  \hline
\end{array}
\end{align*}
\caption{Species chart for the first Circadian clock model (Vilar et al. \cite{Vilar}). $A$ and $R$ denote the activator and repressor proteins. $G_A$ ($M_A$) and $G_R$ ($M_R$) denote the gene (mRNA) corresponding to these proteins. The superscript $b$ or $u$ indicates the \emph{bound} or \emph{unbound} from of the gene. $AR$ denotes the dimeric complex between $A$ and $R$.}
\label{secchart:ccvilar}
\end{table}

 \begin{table}[h!]
 \begin{center}
\begin{tabular}{|c|l|c|l|} \hline
No. & Reaction & No. & Reaction \\\hline
1 & $\mathbf{S_6} + \mathbf{S_2} \longrightarrow \mathbf{S}_7 $ & 9 & $\mathbf{S_2} \longrightarrow \emptyset $ \\
2 & $\mathbf{S}_7 \longrightarrow  \mathbf{S_6} + \mathbf{S_2} $ & 10 & $\mathbf{S_9} \longrightarrow \mathbf{S}_9 + \mathbf{S_3}  $ \\
3 & $\mathbf{S_8} + \mathbf{S_2} \longrightarrow \mathbf{S}_9 $ & 11 & $\mathbf{S_8} \longrightarrow \mathbf{S}_8 + \mathbf{S_3}  $ \\
4 & $\mathbf{S}_9 \longrightarrow   \mathbf{S_8} + \mathbf{S_2} $ & 12 & $ \mathbf{S_3} \longrightarrow \emptyset $ \\
5 & $\mathbf{S_7}  \longrightarrow \mathbf{S}_7+ \mathbf{S_1} $ & 13 & $\mathbf{S_3} \longrightarrow \mathbf{S}_3 + \mathbf{S_4}  $ \\
6 & $\mathbf{S_6}  \longrightarrow \mathbf{S}_6+ \mathbf{S_1}  $ & 14& $ \mathbf{S_4} \longrightarrow \emptyset $ \\
7 & $ \mathbf{S_1} \longrightarrow \emptyset$ & 15 & $\mathbf{S_2} + \mathbf{S_4} \longrightarrow \mathbf{S}_5$ \\
8 & $\mathbf{S_1} \longrightarrow \mathbf{S_1} + \mathbf{S}_2 $ & 16 & $\mathbf{S_5} \longrightarrow \mathbf{S}_4 $ \\
\hline
\end{tabular}
\end{center}
\caption{Reactions for the first Circadian clock model (Vilar et al. \cite{Vilar}). These reactions follow the Species Chart in Table \ref{secchart:ccvilar}. The network is depicted in Figure \ref{figure:ccvilar}.}
\label{rxnchart:ccvilar}
\end{table}

\begin{figure}[h!]
\begin{center}
\frame{\includegraphics[width=0.94\textwidth]{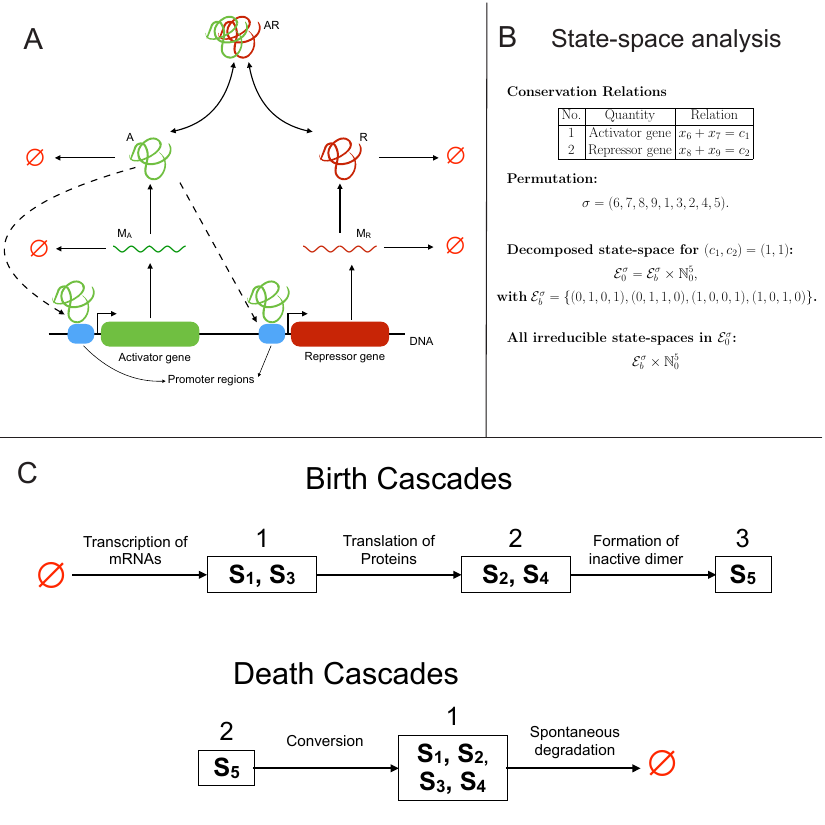}}
\end{center}
\caption{Results for the Circadian clock model (Vilar et al. \cite{Vilar}) depicted in panel {\bf A}. The results of state-space analysis are presented in panel {\bf B} and the birth-death cascades are shown in panel {\bf C}. The birth-cascades signify various stages in the network. Note that in this network there are four \emph{bounded} species ($\sigma(1),\dots,\sigma(4)$) and five \emph{free} species ($\sigma(5),\dots,\sigma(9)$). }
\label{figure:ccvilar}
\end{figure}

Next we examine the detailed mammalian circadian clock given in \cite{Leloup2003} (see Figure \ref{figure:ccleloup}{\bf A}). It consists of three gene-expression modules corresponding to \emph{Per}, \emph{Cry} and \emph{Bmal1} genes. The proteins created by these genes participate in complex regulatory steps through various mechanisms such as formation of dimeric complexes, translocation in an out of nucleus and undergoing phosphorylation-dephosphorylation cycles. Overall the network consists of $16$ species (see the Species Chart in Table \ref{secchart:ccleloup}) and $52$ reactions (see Table \ref{rxnchart:ccleloup}). Interestingly there are no conservation relations for this network and so all the species are \emph{free} and Theorem \ref{maintheorem} proves that $\N^{16}_0$ is the unique irreducible state-space (see Figure \ref{figure:ccleloup}). The birth and death cascades, displayed in Figure \ref{figure:ccleloup}{\bf C}, signify the various network stages as before.

\begin{table}[h!]
\begin{align*}
\begin{array}{|l|l|} \hline
\multicolumn{2}{|c|}{\textnormal{Species Chart} } \\ \hline
{\bf S}_1 = M_P &  {\bf S}_9 = PC_N \\
{\bf S}_2 = M_C & {\bf S}_{10} = PC_{CP} \\
{\bf S}_3 = M_B & {\bf S}_{11} = PC_{NP} \\
{\bf S}_4 = P_C & {\bf S}_{12} = B_C \\
{\bf S}_5 = C_C & {\bf S}_{13} = B_{CP} \\
{\bf S}_6 = P_{CP} & {\bf S}_{14} = B_N \\
{\bf S}_7 = C_{CP} & {\bf S}_{15} = B_{NP} \\
{\bf S}_8 = PC_C & {\bf S}_{16} = I_N \\ \hline
\end{array}
\end{align*}
\caption{Species chart for the second Circadian clock model (Leloup and Goldbeter \cite{Leloup2003}). $P$, $C$ and $B$ denote the \emph{Per}, \emph{Cry} and \emph{Bmal1} proteins. The dimeric complex between two proteins $X$ and $Y$ is denoted by $XY$. The subscripts $C$, $N$ and $P$ stand for ``Cytosol", ``Nucleus" and ``Phosphorylated" respectively. $M_X$ denotes the mRNA for protein $X$ and $I_N$ is an inactive trimer in the Nucleus.}
\label{secchart:ccleloup}
\end{table}

  \begin{table}[h!]
 \begin{center}
   \begin{align*}
\begin{array}{|c|l|c|l|}\hline 
\textnormal{No.} & \textnormal{Reaction} & \textnormal{No.}  & \textnormal{Reaction} \\ \hline
1 &  \emptyset \longrightarrow {\bf S}_1    & 27 &  {\bf S}_8 \longrightarrow {\bf S}_9    \\
2 & {\bf S}_1 \longrightarrow  \emptyset   & 28 & {\bf S}_8  \longrightarrow  \emptyset  \\
3 & {\bf S}_1 \longrightarrow  \emptyset   & 29 & {\bf S}_9 \longrightarrow {\bf S}_{11}  \\
4 &  \emptyset \longrightarrow {\bf S}_2    & 30 & {\bf S}_{11} \longrightarrow {\bf S}_9  \\
5 & {\bf S}_2 \longrightarrow  \emptyset   & 31 & {\bf S}_9 + {\bf S}_{14} \longrightarrow {\bf S}_{16}  \\
6 & {\bf S}_2 \longrightarrow  \emptyset   & 32 & {\bf S}_{16} \longrightarrow {\bf S}_9 + {\bf S}_{14}  \\
7 &  \emptyset \longrightarrow {\bf S}_3    & 33 & {\bf S}_9 \longrightarrow  \emptyset  \\
8 & {\bf S}_3 \longrightarrow  \emptyset    & 34 & {\bf S}_{10} \longrightarrow  \emptyset  \\
9 & {\bf S}_3 \longrightarrow  \emptyset    & 35 & {\bf S}_{10} \longrightarrow  \emptyset  \\
10 & {\bf S}_1 \longrightarrow {\bf S}_1 + {\bf S}_4    & 36 & {\bf S}_{11} \longrightarrow  \emptyset  \\
11 & {\bf S}_4 \longrightarrow {\bf S}_6    & 37 & {\bf S}_{11} \longrightarrow  \emptyset  \\
12 & {\bf S}_6 \longrightarrow {\bf S}_4    & 38 & {\bf S}_3 \longrightarrow {\bf S}_{12}  \\
13 & {\bf S}_8 \longrightarrow {\bf S}_4 +{\bf S}_5    & 39 & {\bf S}_{12} \longrightarrow {\bf S}_{13}  \\
14 & {\bf S}_4 +{\bf S}_5 \longrightarrow {\bf S}_8    & 40 & {\bf S}_{13} \longrightarrow {\bf S}_{12}  \\
15 & {\bf S}_4 \longrightarrow \emptyset   & 41 & {\bf S}_{12} \longrightarrow {\bf S}_{14}  \\
16 & {\bf S}_2  \longrightarrow {\bf S}_2 + {\bf S}_5    & 42 & {\bf S}_{14} \longrightarrow {\bf S}_{12}  \\
17 & {\bf S}_7 \longrightarrow {\bf S}_5    & 43 & {\bf S}_{12} \longrightarrow  \emptyset  \\
18 & {\bf S}_5 \longrightarrow {\bf S}_7    & 44 & {\bf S}_{13} \longrightarrow  \emptyset  \\
19 & {\bf S}_5 \longrightarrow  \emptyset    & 45 & {\bf S}_{13}\longrightarrow  \emptyset  \\
20 &{\bf S}_6 \longrightarrow  \emptyset   & 46 & {\bf S}_{14} \longrightarrow {\bf S}_{15}  \\
21 &   {\bf S}_6 \longrightarrow  \emptyset  & 47 & {\bf S}_{15} \longrightarrow {\bf S}_{14}  \\
22 & {\bf S}_7 \longrightarrow  \emptyset   & 48 & {\bf S}_{14} \longrightarrow  \emptyset  \\
23 & {\bf S}_7 \longrightarrow  \emptyset     & 49 & {\bf S}_{15} \longrightarrow  \emptyset  \\
24 & {\bf S}_8 \longrightarrow {\bf S}_{10}     &  50 & {\bf S}_{15} \longrightarrow  \emptyset  \\
25 & {\bf S}_{10} \longrightarrow {\bf S}_8    & 51 & {\bf S}_{16} \longrightarrow  \emptyset  \\
26 &  {\bf S}_9 \longrightarrow {\bf S}_8   & 52 & {\bf S}_{16} \longrightarrow  \emptyset \\  \hline
\end{array}
\end{align*}
\end{center}
\caption{Reactions for the second Circadian clock model (Leloup and Goldbeter \cite{Leloup2003}). These reactions follow the Species Chart in Table \ref{secchart:ccleloup}. The network is depicted in Figure \ref{figure:ccleloup}.}
\label{rxnchart:ccleloup}
\end{table}

\begin{figure}[h!]
\centering
\frame{\includegraphics[width=0.98\textwidth]{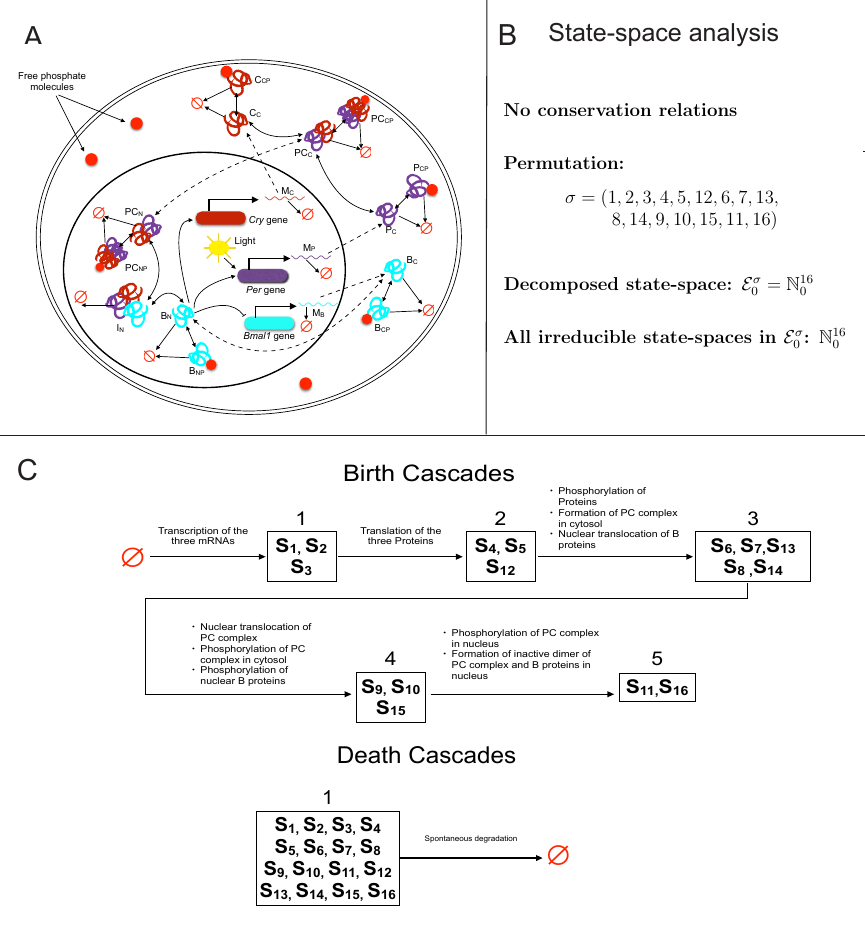}}
\caption{Results for the Circadian clock model (Leloup and Goldbeter \cite{Leloup2003}) depicted in panel {\bf A}. The results of state-space analysis are presented in panel {\bf B} and the birth-death cascades are shown in panel {\bf C}. The birth-cascades signify various stages in the network. Note that in this network all the species are \emph{free} as there are no conservation relations. }
\label{figure:ccleloup}
\end{figure}

\subsection{Bacterial Heat-Shock response model} \label{ex:hsr}

We now analyze the bacterial heat-shock response model developed in Kurata et al. \cite{KurataHS}. It is known that when bacterial cells are exposed to thermal shocks, certain cellular structures are damaged, causing some proteins to denature or unfold. Accumulation of these denatured proteins within cells can disrupt their normal functioning and hence a regulatory mechanism has evolved to detect and repair unfolded proteins. The model in \cite{KurataHS} describes this mechanism in bacterium \emph{E. Coli} and it consists of $5$ distinct gene-expression modules responsible for creating various heat-shock specific proteins ($hsps$), such as the heat-shock transcription factor $\sigma^{32}$ factor, proteases like \emph{HslVu, FtsH} etc. and a molecular chaperon \emph{DnaK}, which assists in the refolding of denatured proteins. The sigma factor $\sigma^{32}$ competes with the dominant sigma factor $\sigma^{70}$ to bind to the \emph{RNA Polymerase} (RNAP), and when it is able to bind, it can initiate the transcription of the heat-shock gene $ph$, which then produces the other heat-shock proteins. The sigma factor $\sigma^{32}$ is itself produced by the gene $pg$ when RNAP is bound to sigma factor $\sigma^{70}$. As explained in \cite{KurataHS} the heat-shock proteins interact in complex ways to realize both feedback and feedforward control that confers several performance advantages to the heat-shock response mechanism.
                            
The network underlying the model in \cite{KurataHS} is shown in Figure \ref{figure:heatshock}{\bf A} and it is fairly large, with $28$ species (see the Species Chart in Table \ref{secchart:heatshock}) and $61$ reactions (see Table \ref{rxnchart:heatshock}). The network contains $5$ semi-positive conservation relations corresponding to various conserved quantities (see Figure \ref{figure:heatshock}{\bf B}). Due to these relations $12$ species are \emph{bounded} while the rest are \emph{free}. Using our analysis we can the find the decomposed state-space for this network and verify the existence of a unique irreducible state-space (see Figure \ref{figure:heatshock}{\bf B}). Note that in this example, the irreducible state-space is a strict subset of the full state-space and in particular, the last two components of each state in this irreducible state-space are zeros. These two components being zero implies that the copy-number of unfolded protein molecules is zero. As these unfolded protein molecules can only \emph{leave} the network by converting to a properly folded protein (via reaction 25 in Table \ref{rxnchart:heatshock}), we can conclude that the heat-shock response network is such, that despite having noisy, stochastic dynamics, all the unfolded protein molecules will eventually get folded, irrespective of the initial count of these unfolded protein molecules. We base this assertion on the fact that starting from any initial state in the state-space, the Markovian reaction dynamics will eventually get trapped inside the unique irreducible state-space. What makes this assertion powerful is that it holds \emph{just} due to the structure of the network, regardless of the values of the rate constants of the reactions. Also for this network we are in the situation of Remark \ref{rem:bctissinglepath}. Hence we can arrange all the \emph{free} species into birth and death cascades which correspond to various network stages (see Figure \ref{figure:heatshock}{\bf {C}}).

\begin{table}[h!]
\begin{align*}
\begin{array}{|l|l|} \hline
\multicolumn{2}{|c|}{\textnormal{Species Chart} } \\ \hline
{\bf S}_1 = \textnormal{RNAP} &  {\bf S}_{15} = \sigma^{32}:\textnormal{DnaK}\\
{\bf S}_2 = \sigma^{70} & {\bf S}_{16} = \textnormal{P}_{ \textnormal{unfolded} } \\
{\bf S}_3 = \sigma^{70}:\textnormal{RNAP} & {\bf S}_{17} = \textnormal{P}_{ \textnormal{unfolded} }:\textnormal{DnaK}  \\
{\bf S}_4 = \sigma^{32}& {\bf S}_{18} = \textnormal{Protease}\\
{\bf S}_5 = \sigma^{32}:\textnormal{RNAP} & {\bf S}_{19} = \sigma^{32}:\textnormal{DnaK}:\textnormal{Protease} \\
{\bf S}_6 = D & {\bf S}_{20} = \textnormal{HslVu} \\
{\bf S}_7 = \textnormal{RNAP}:D& {\bf S}_{21} = \sigma^{32}:\textnormal{HslVu} \\
{\bf S}_8= \sigma^{70}:\textnormal{RNAP}:D& {\bf S}_{22} = \textnormal{mRNA}(\textnormal{DnaK})\\
{\bf S}_9 = \sigma^{32}:\textnormal{RNAP}& {\bf S}_{23} =  \textnormal{mRNA}(\textnormal{Protease})\ \\
{\bf S}_{10} = \textnormal{pg} & {\bf S}_{24} =  \textnormal{mRNA}(\textnormal{HslVu}) \\
{\bf S}_{11} = \textnormal{ph} & {\bf S}_{25} =  \textnormal{mRNA}( \sigma^{32})\\
{\bf S}_{12} = \sigma^{70}:\textnormal{RNAP}:\textnormal{pg} & {\bf S}_{26} = \textnormal{mRNA}(\textnormal{FtsH}) \\
{\bf S}_{13} = \sigma^{32}:\textnormal{RNAP}:\textnormal{ph}& {\bf S}_{27} = \textnormal{FtsH} \\ 
{\bf S}_{14} = \textnormal{DnaK} & {\bf S}_{28} =\sigma^{32}:\textnormal{DnaK}:\textnormal{FtsH}  \\  \hline
\end{array}
\end{align*}
\caption{Species chart for the Bacterial Heat-Shock response model (Kurata et al. \cite{KurataHS}). Here the complex formed by binding two biomolecules $A$ and $B$ is denoted by $A:B$. The main players in this network are the heat-shock proteins \emph{DnaK}, \emph{Protease}, \emph{HslVu} and \emph{FtsH}, and the transcription factor $\sigma^{32}$ which binds to the RNA Polymerase (RNAP) to initiate the production of the heat-shock proteins via the $ph$ gene. This transcription factor $\sigma^{32}$ is itself created by the gene $pg$ when the transcription factor $\sigma^{70}$ is bound to RNAP. Here $D$ denotes the part of DNA that does not include these genes, and $P$ is the protein of interest that needs to be converted from its denatured form $\textnormal{P}_{ \textnormal{unfolded} }$ to its proper form $\textnormal{P}_{ \textnormal{folded} }$. The mRNA for protein $X$ is denoted by $\textnormal{mRNA}(X)$.}
\label{secchart:heatshock}
\end{table}

  \begin{table}[h!]
 \begin{center}
 \begin{align*}
\begin{array}{|c|l|c|l|}\hline 
\textnormal{No.} & \textnormal{Reaction} & \textnormal{No.}  & \textnormal{Reaction} \\ \hline
1 & {\bf S}_1 + {\bf S}_2 \longrightarrow {\bf S}_3    & 32 & {\bf S}_{18}     \longrightarrow  \emptyset  \\
2 &  {\bf S}_3  \longrightarrow   {\bf S}_1 + {\bf S}_2 & 33 & {\bf S}_{19}   \longrightarrow  {\bf S}_{15} \\
3 &  {\bf S}_4 + {\bf S}_1 \longrightarrow  {\bf S}_5   & 34 & {\bf S}_{12}   \longrightarrow  {\bf S}_{25}+ {\bf S}_{12} \\
4 & {\bf S}_5   \longrightarrow {\bf S}_4 + {\bf S}_1 & 35 & {\bf S}_{25}    \longrightarrow   \emptyset   \\
5 & {\bf S}_1 + {\bf S}_6   \longrightarrow {\bf S}_7 & 36 & {\bf S}_{25}    \longrightarrow  {\bf S}_{4}+ {\bf S}_{25} \\
6 &  {\bf S}_7  \longrightarrow {\bf S}_1 + {\bf S}_6 & 37 & {\bf S}_{4}    \longrightarrow  \emptyset  \\
7 &  {\bf S}_4 + {\bf S}_{14} \longrightarrow {\bf S}_{15}   & 38 & {\bf S}_{19}    \longrightarrow  {\bf S}_{18}+ {\bf S}_{14} \\
8 &  {\bf S}_{15}  \longrightarrow {\bf S}_4 + {\bf S}_{14} & 39 & {\bf S}_{21}     \longrightarrow  {\bf S}_{20} \\
9 & {\bf S}_{16} + {\bf S}_{14}  \longrightarrow {\bf S}_{17} & 40 & {\bf S}_{28}   \longrightarrow  {\bf S}_{14}+ {\bf S}_{27} \\
10 & {\bf S}_{17} \longrightarrow {\bf S}_{16} + {\bf S}_{14} & 41 & {\bf S}_{13}    \longrightarrow  {\bf S}_{13}+ {\bf S}_{24} \\
11 & {\bf S}_{3} + {\bf S}_{6}  \longrightarrow {\bf S}_8 & 42 & {\bf S}_{24}   \longrightarrow  \emptyset  \\
12 & {\bf S}_8    \longrightarrow  {\bf S}_{3} + {\bf S}_{6} & 43 & {\bf S}_{24}    \longrightarrow  {\bf S}_{20}+ {\bf S}_{24} \\
13 & {\bf S}_5+ {\bf S}_{6}      \longrightarrow  {\bf S}_{9}  & 44 & {\bf S}_{20}    \longrightarrow \emptyset  \\
14 & {\bf S}_{9}    \longrightarrow   {\bf S}_5+ {\bf S}_{6}& 45 & {\bf S}_{21}    \longrightarrow  {\bf S}_{4} \\
15 & {\bf S}_{3}+ {\bf S}_{10}    \longrightarrow  {\bf S}_{12} & 46 & {\bf S}_{4}+ {\bf S}_{20}    \longrightarrow  {\bf S}_{21}\\
16 & {\bf S}_{12}     \longrightarrow  {\bf S}_{3}+ {\bf S}_{10} & 47 & {\bf S}_{21}    \longrightarrow  {\bf S}_{4}+ {\bf S}_{20} \\
17 &{\bf S}_{5}+ {\bf S}_{11}    \longrightarrow  {\bf S}_{13} & 48 & {\bf S}_{13}    \longrightarrow  {\bf S}_{13}+ {\bf S}_{26} \\
18 & {\bf S}_{13}   \longrightarrow  {\bf S}_{5}+ {\bf S}_{11} & 49 & {\bf S}_{26}  \longrightarrow  \emptyset \\
19 & {\bf S}_{15}+ {\bf S}_{18}    \longrightarrow  {\bf S}_{19}& 50 & {\bf S}_{26}   \longrightarrow  {\bf S}_{26}+ {\bf S}_{27} \\
20 & {\bf S}_{19}   \longrightarrow  {\bf S}_{15}+ {\bf S}_{18} & 51 & {\bf S}_{27}    \longrightarrow  \emptyset \\
21 & {\bf S}_{13}    \longrightarrow  {\bf S}_{13}+ {\bf S}_{22} & 52 & {\bf S}_{28}   \longrightarrow  {\bf S}_{15}  \\
22 & {\bf S}_{22}    \longrightarrow  \emptyset & 53 & {\bf S}_{15}+ {\bf S}_{27}    \longrightarrow  {\bf S}_{28} \\
23 &{\bf S}_{22}    \longrightarrow  {\bf S}_{14}+ {\bf S}_{22} & 54 & {\bf S}_{28}    \longrightarrow  {\bf S}_{15}+ {\bf S}_{27} \\
24 & {\bf S}_{14}    \longrightarrow \emptyset & 55 &  {\bf S}_{5}    \longrightarrow  {\bf S}_{1} \\
25 & {\bf S}_{17}    \longrightarrow  {\bf S}_{14} +   P_{ \textnormal{folded} } & 56 &  {\bf S}_{13}    \longrightarrow  {\bf S}_{1}+ {\bf S}_{11} \\
26 & {\bf S}_{15}    \longrightarrow  {\bf S}_{4} & 57 &  {\bf S}_{9}    \longrightarrow   {\bf S}_{7} \\
27 & {\bf S}_{19}    \longrightarrow  {\bf S}_{4}+ {\bf S}_{18} & 58 & {\bf S}_{15}    \longrightarrow  {\bf S}_{14} \\
28 &{\bf S}_{28}   \longrightarrow  {\bf S}_{4}+ {\bf S}_{27} & 59 &  {\bf S}_{19}    \longrightarrow  {\bf S}_{14}+ {\bf S}_{18} \\
29 & {\bf S}_{13}    \longrightarrow  {\bf S}_{13}+ {\bf S}_{23}& 60 &  {\bf S}_{28}    \longrightarrow  {\bf S}_{14}+ {\bf S}_{27} \\
30 & {\bf S}_{23}   \longrightarrow \emptyset& 61 &  {\bf S}_{21}    \longrightarrow   {\bf S}_{20} \\
31 & {\bf S}_{23}   \longrightarrow  {\bf S}_{18}+ {\bf S}_{23}&  &  \\ \hline
\end{array}
\end{align*} 
\end{center}
\caption{Reactions for the Bacterial Heat-Shock response model (Kurata et al. \cite{KurataHS}). These reactions follow the Species Chart in Table \ref{secchart:heatshock}. The network is depicted in Figure \ref{figure:heatshock}{\bf A}. Reaction 25 represents the refolding of denatured protein $P$ when it is in complex with the chaperon protein \emph{DnaK}. This reaction is the only way in which molecules of protein $P$ exit the network.}
\label{rxnchart:heatshock}
\end{table}

\begin{figure}[h!]
\centering
\frame{\includegraphics[width=0.98\textwidth]{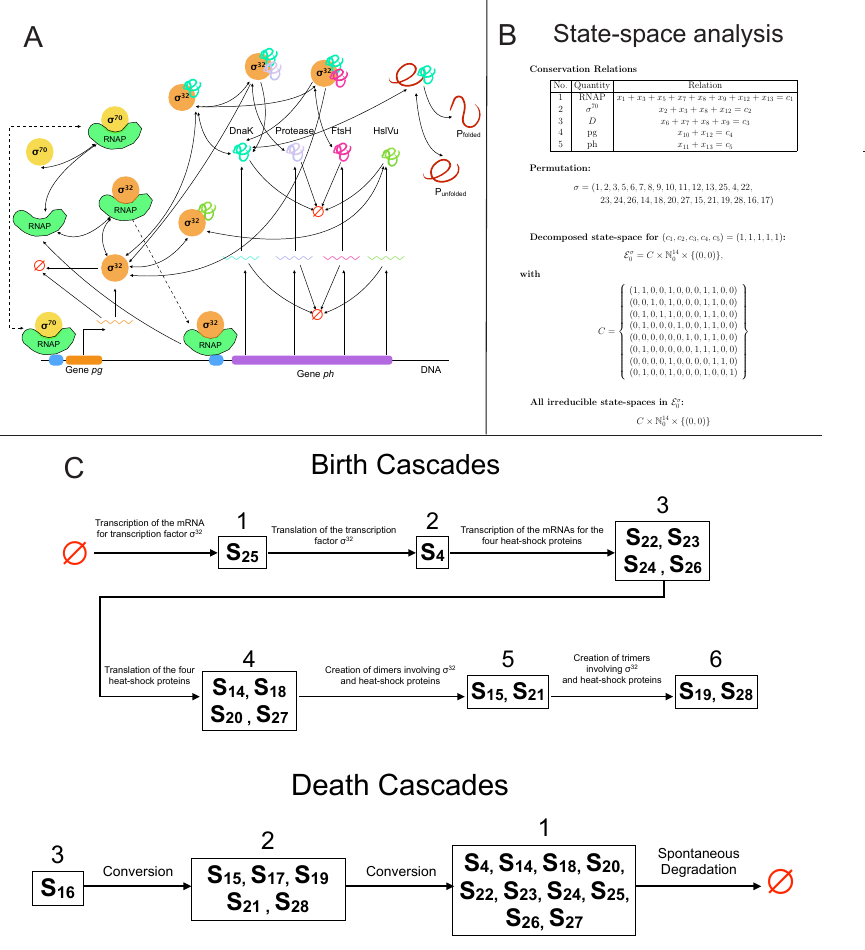}}
\caption{Results for the Bacterial Heat-Shock response model (Kurata et al. \cite{KurataHS}) depicted in panel {\bf A}. The mRNAs are not labeled to avoid clutter. The results of state-space analysis are presented in panel {\bf B} and the birth-death cascades are shown in panel {\bf C}. The birth-cascades signify various stages in the network. Note that in this network there are twelve \emph{bounded} species ($\sigma(1),\dots,\sigma(12)$) and fourteen \emph{free} species ($\sigma(13),\dots,\sigma(26)$). }
\label{figure:heatshock}
\end{figure}

\subsection{A simple Toxin-Antitoxin network} \label{ex:toxin}

Many intracellular networks have reactions taking place on many different time-scales \cite{Kang,weinan2}. It is well known that direct stochastic simulations of such networks, using Gillespie's SSA \cite{GP}, is computationally infeasible because most of the resources are spent on simulating the ``fast" reactions which are generally less important than the ``slow" ones \cite{ssSSA}. To remedy this problem, approximate simulation approaches have been developed that \emph{separate the time-scales} by applying the quasi-stationary assumption (QSA) on the fast subcomponents of the network \cite{Kang}. Under this assumption, the slow reactions are ``switched-off" and the dynamics of the slow species involved in these reactions is ``frozen in time", while the stable Markovian dynamics of the fast subnetwork relaxes to a unique stationary distribution (see \eqref{ergodicconvergenceoff}). Using this stationary distribution the propensities of the slow reactions can be estimated and the next slow reaction can be sampled. Applying QSA between every consecutive slow reaction events allows one to approximate the original dynamics without having the simulate the fast reactions, thereby reducing the overall computational effort drastically.

For a successful application of QSA it is imperative to ascertain the stability or ergodicity of the fast subnetwork in \emph{real-time} (i.e. during the simulation run), as the set of available fast reactions can depend on the current state of the slow species. Moreover for efficiently estimating the slow reaction propensities it is useful to determine the exact time-varying support-sets of the stationary distributions. These tasks can be \emph{automatically} accomplished by integrating our computational procedures for state-space analysis, with any QSA-based simulation algorithm like the \emph{slow-scale} SSA or ssSSA \cite{ssSSA} or the Nested SSA \cite{weinan1,weinan2}. We illustrate this next using a simple Toxin-Antitoxin network module which is found in many bacterial cells and is believed to lead to the formation of slow-growing \emph{persister} cells that exhibit multi-drug tolerance. This example shows how \emph{restricted} species arise naturally when we restrict our attention to a subnetwork within the larger network. Moreover this example also highlights that the flexibility in the classification of \emph{unbounded} species as \emph{free} or \emph{restricted} (see Remark \ref{rem:flixibilityofI}) can be really important for applications. In fact for this example, this classification will change randomly with time depending on the states visited by a stochastic trajectory.

The simple Toxin-Antitoxin network we consider is depicted in Figure \ref{figure:toxin} and it is based on the more detailed model given in \cite{ToxinAntitoxinGelens}. It consists of a single DNA strand containing genes for both Toxin $T$ and Antitoxin $A$ protein molecules. Both these proteins are translated by a common \emph{biscistronic} mRNA $M$ and they both annihilate each other in the sense that they bind to form an inactive complex $AT$ which does not participate in the dynamics. The Antitoxin molecules directly inhibit the transcription of mRNA and the Toxin molecules convert to a protein $P$ that interferes with bacterial metabolism and harms the cells in various ways \cite{ToxinAntitoxinSavageau}. Following the Species Chart in Table \ref{secchart:toxinantitoxin}, we can describe our simple Toxin-Antitoxin system as a network with $8$ reactions which are listed along with their propensity functions in Table \ref{rxnchart:toxinantitoxin}. The choice of rate constants in these propensity functions is arbitrary but reflective of the values found in the literature \cite{ToxinAntitoxinGelens}. From these rates it can be inferred that two reactions can be considered fast, namely, the translation of mRNA and the mutual annihilation of Toxin and Antitoxin proteins (see Figure \ref{figure:toxin}{\bf A}).

\begin{table}[h!]
\begin{align*}
\begin{array}{|l|l|} \hline
\multicolumn{2}{|c|}{\textnormal{Species Chart} } \\ \hline
{\bf S}_1 = M &  {\bf S}_{3} = A \\
{\bf S}_2 = T & {\bf S}_{4} = P  \\  \hline
\end{array}
\end{align*}
\caption{Species chart for the Toxin-Antitoxin model. $T$ and $A$ denote the Toxin and Antitoxin proteins respectively. They are both translated by the \emph{biscistronic} mRNA $M$ which is produced constitutively by a gene. The Toxin protein converts to another protein $P$ which inhibits the metabolism of the cell. 
}
\label{secchart:toxinantitoxin}
\end{table}

\begin{table}[h!]
\begin{align*}
\begin{array}{|c|l|c|}\hline 
\textnormal{No.} & \textnormal{Reaction} &  \textnormal{Propensity} \\ \hline
1 & \emptyset \longrightarrow  {\bf S}_1  &  \lambda_1(x) = \left( \frac{20}{1 + x_3} \right) \\ 
2 &{\bf S}_1  \longrightarrow  {\bf S}_1 + {\bf S}_2 + {\bf S}_3 &  \lambda_2 (x) = \theta_1 x_1 \\
 3 &{\bf S}_2 + {\bf S}_3  \longrightarrow  \emptyset &  \lambda_3 (x) = \theta_2 x_2 x_3 \\
 4 &{\bf S}_1  \longrightarrow  \emptyset  &  \lambda_4 (x) = 2 x_1 \\
 5 &{\bf S}_2  \longrightarrow  \emptyset &  \lambda_5 (x) = 5 x_2\\
 6 &{\bf S}_3  \longrightarrow \emptyset &  \lambda_6(x) =x_3 \\
  7 &{\bf S}_2  \longrightarrow  {\bf S}_4&  \lambda_7 (x) = 0.5 x_2\\
 8 &{\bf S}_4  \longrightarrow  \emptyset&  \lambda_8 (x) = 0.1x_4 \\
\hline
\end{array}
\end{align*}
\caption{Reactions for the Toxin-Antitoxin model according to the Species Chart in Table \ref{secchart:toxinantitoxin}. The associated propensity functions $\lambda_k$-s are also provided. Here $x_i$ denotes the copy-number of species ${\bf S}_i$. We choose $\theta_1 = 100$ and $\theta_2= 10$, and hence reactions $2$ and $3$ are much \emph{faster} in comparison to the other reactions. Reaction $3$ is the annihilation reaction between Toxin and Antitoxin molecules, which actually represents formation of an inactive complex $AT$. In reactions $2$ and $3$, we choose parameters $\theta_1 =100$ and $\theta =10$, and so the subnetwork formed by these two reactions can be considered \emph{fast} in the context of the larger network.
}
\label{rxnchart:toxinantitoxin}
\end{table}

\begin{figure}[h!]
\centering
\frame{\includegraphics[width=0.93\textwidth]{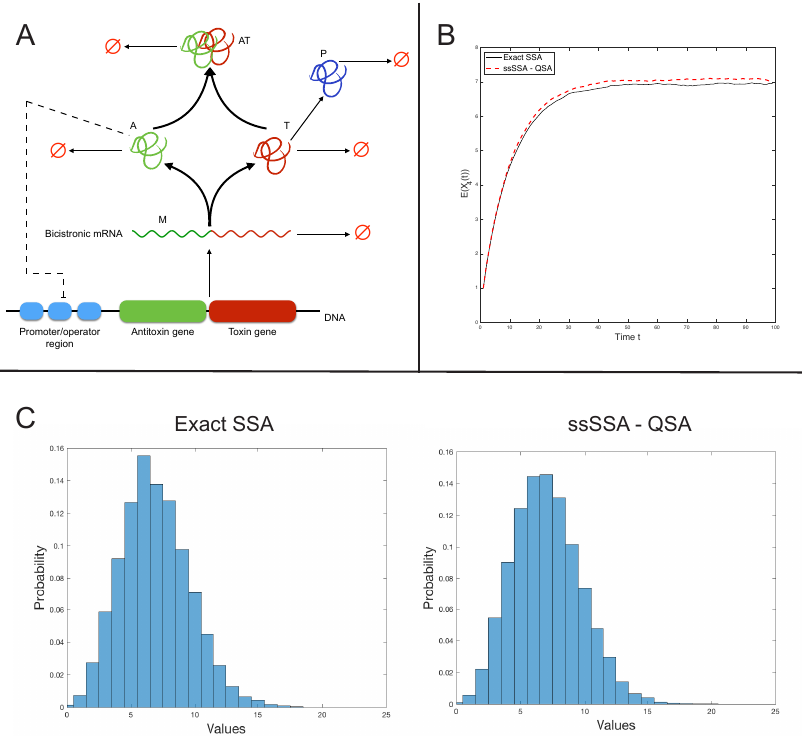}}
\caption{ Stochastic simulations for the Toxin-Antitoxin model depicted in panel {\bf A}. The \emph{fast} reactions are indicated with a thicker arrow. In panel {\bf B} the estimated mean-dynamics $\E(X_4(t))$ for copy-numbers of protein $P$ is plotted, and in panel {\bf C}, the estimated probability distributions (histograms) for these copy-numbers at the final time-point $T=100$ are shown. All these estimations were performed with two simulation schemes - the Exact SSA \cite{GP} and the approximate ssSSA \cite{ssSSA} that uses the quasi-stationary assumption (QSA). Observe that ssSSA is quite accurate and for this example, the simulations using ssSSA were about $6$ times faster than those using the Exact SSA.
}
\label{figure:toxin}
\end{figure}

Suppose one is interested in simulating the stochastic model for this network in the time period $[0,T]$ for $T =100$. Performing exact stochastic simulations is cumbersome due to the presence of fast reactions, but an approximate QSA-based algorithm can be used for carrying out these simulations with far lesser computational effort. To apply QSA, we consider the dynamics of the component of the network consisting of two fast reactions                                                                                  
\begin{align}
\label{defn:tox_antitox}
\mathbf{S}_1 \longrightarrow \mathbf{S}_1 + \mathbf{S}_2 + \mathbf{S}_3 \quad  \textnormal{and} \quad  \mathbf{S}_2 + \mathbf{S}_3 \longrightarrow \emptyset,
\end{align}
assuming that all the other \emph{slow} reactions are switched off and the state of the slow variables is fixed. Following \cite{Kang}, the slow variables in this network are $x_1$ (mRNA copy-number) and $y = x_2 - x_3$ (difference between Toxin and Antotoxin copy-numbers), because these variables are unaffected by the fast reactions in \eqref{defn:tox_antitox}. Fixing $x_1$ and $y$ we compute the decomposed state-space for this subnetwork using Algorithm \ref{finddecompss}. Due to the semi-positive conservation relation $\gamma_1 = (1,0,0)$, the species $\mathbf{S}_1$ is \emph{bounded} and its finite state-space is simply the singleton $\{x_1 \}$. The only other independent conservation relation is $\gamma_2 = (0,1,-1)$, which is mixed-sign, and so there exists a \emph{restricted} species. Algorithm \ref{finddecompss} will automatically classify one of the remaining species as \emph{free} and the other as \emph{restricted}, depending on the value of $y$, in such a way that the associated affine map $\phi$ is compatible with the network (recall Assumption \ref{affinefunction}). If $y \geq 0$, then Algorithm \ref{finddecompss} picks species $\mathbf{S}_3$ as \emph{free}, species $\mathbf{S}_2$ as \emph{restricted}, the affine map $\phi$ as $\phi(x_2) = x_3 + y$ and the permutation $\sigma$ as $\sigma = (1,3,2)$. However these choices violate network compatibility when $y< 0$ and so in this situation Algorithm \ref{finddecompss} picks species $\mathbf{S}_2$ is \emph{free}, species $\mathbf{S}_3$ is \emph{restricted}, the affine map $\phi$ as $\phi(x_3) = x_2 - y$ and the permutation $\sigma$ as $\sigma = (1,2,3)$. For convenience, let $\mathbf{S}_y$, $\phi_y$ and $\sigma_y$ denote the $y$-dependent choices of the \emph{free} species, the affine map and the permutation $\sigma$ respectively. Then according to Algorithm \ref{finddecompss}, the decomposed state-space for the fast subnetwork permuted with $\sigma_y$ is simply $\{x_1\} \times \Phi_y$, where $\Phi_y$ is the graph (see \eqref{defngraphphi}) of map $\phi_y$. 

As described in Section \ref{sec:netreductionbyelimination} we can reduce the fast subnetwork \eqref{defn:tox_antitox} by eliminating the \emph{restricted} species to obtain the following network 
\begin{align}
\label{defn:tox_antitox2}
\mathbf{S}_1 \longrightarrow \mathbf{S}_1 + \mathbf{S}_y \quad  \textnormal{and} \quad  \mathbf{S}_y \longrightarrow \emptyset,
\end{align}
where the propensity function for the first reaction is same as before, but for the second reaction it changes to
\begin{align}
\label{toxin:newprop}
\tilde{ \lambda }_y (x_1,z) = \theta_2 z \phi_y(z) = \theta_2 z (z + |y|).
\end{align}
Here $z$ is the copy-number of the \emph{free} species $\mathbf{S}_y$. Observe that the compatibility of map $\phi_y$ ensures that $\tilde{ \lambda }_y(x_1,z) = 0$ if and only if $z = 0$, and hence the reduced network satisfies Assumption \ref{massacttype}. This allows us to apply Algorithm \ref{findrednet} to find all the irreducible state-spaces for this reduced network in the state-space $\{x_1\} \times \N_0$ and this algorithm outputs that there is a single irreducible state-space which is identical to the full state-space $\{x_1\} \times \N_0$. Therefore the stationary distribution for this network must be exactly supported on $\{x_1\} \times \N_0$, and in fact using the results in \cite{ACKProd} (see also Section \ref{ex:productdistr}) we can compute it to be exactly:
\begin{align}
\label{sddistreducednetworkexampletoxi}
\tilde{\pi}_y(x_1,z) = \left\{   
\begin{array}{cl}
1 & \textnormal{ if }  x_1 = 0 \textnormal{ and } z = 0\\
\frac{1}{L} \left(\frac{\theta_1}{  \theta_2}  \right)^z   x^z_1  \left( \frac{ |y| ! }{ z! (z+ |y|)!  } \right)   & \textnormal{ if }  x_1 > 0 \textnormal{ and } z = 0,1,2,\dots\\
0 & \textnormal{ otherwise},
\end{array}
\right.
\end{align}
where $a!$ denotes the factorial of $a$ and $L$ is the normalization constant given by
\begin{align*}
L=  \sum_{z = 0}^\infty  \left(\frac{\theta_1 }{  \theta_2}  \right)^z  x^z_1 \left( \frac{ |y| ! }{ z! (z+ |y|)!  } \right).
\end{align*}
It is easy to check that $L < \infty$ and hence $\tilde{\pi}_y$ is a valid stationary distribution over $\{x_1\} \times \N_0$. The uniqueness of this stationary distribution is guaranteed because $\{x_1\} \times \N_0$ is the only irreducible state-space for the reduced network. Due to Proposition \ref{mainprop_reducednetwork} we can conclude that $\{x_1\} \times \Phi_y$ is the only irreducible state-space for the full subnetwork \eqref{defn:tox_antitox} under permutation $\sigma_y$ and its unique stationary distribution $\pi_y$ on $\{x_1\} \times \Phi_y$ is just
\begin{align*}
\pi_y(x_1,z, \phi_y(z)) =  \tilde{\pi}_y(x_1,z). 
\end{align*}
This also shows that the subnetwork \eqref{defn:tox_antitox} is ergodic which is necessary for the application of QSA. For the \emph{slow} reactions $k \in \{1,4,5,6,7,8\}$ the propensities needed for the application of QSA can be estimated as
\begin{align*}
\hat{ \lambda }_k(x_1,y,x_4) = \sum_{z \in \N_0} \lambda_k(x_1,z,\phi_y(z),x_2)   \tilde{\pi}_y(x_1,z).
\end{align*}

Using these propensities we can simulate the Toxin-Antitoxin network with ssSSA \cite{ssSSA} to estimate the probability distribution of copy-numbers of protein $P$ (i.e. $X_4(t)$) as well as the mean dynamics of these copy-numbers. The results are reported in Figure \ref{figure:toxin}, where results from the exact SSA simulations are also shown for comparison. One can see that ssSSA simulations are quite accurate thanks to the correct identification of the ergodic subnetwork and its stationary distribution. Moreover these ssSSA simulations are about $6$ times faster than simulations with Exact SSA. This simple example nicely illustrates how our algorithms for state-space analysis can be integrated with a simulation scheme like ssSSA, to aid the application of QSA. Note that in more complicated examples, the exact form of the stationary distribution may not be available, but our analysis can verify its uniqueness and provide a description of its exact support. In such cases this knowledge can be combined with other simulation-based schemes to sample from the stationary distributions and estimate the slow propensities for applying QSA (see \cite{weinan1,weinan2}).

\subsection{Networks with product-form stationary distributions} \label{ex:productdistr}

In this paper we present a method for finding irreducible state-spaces for networks where infinitely many states are accessible. On each of these irreducible state-spaces, the uniqueness of the stationary distribution is guaranteed but its existence needs to be checked by other means (like the analysis in \cite{GuptaPLOS}). This is complementary to certain other results in the literature which assume the knowledge of irreducible state-spaces and demonstrate the existence of product-form stationary distributions for a large class of networks \cite{ACKProd}. Exploiting this complementarity, we now explore how the combination of our results with the results on product-form stationary distributions, can provide us with a complete characterization of the simplex of stationary distributions (see \eqref{stat_simplex}) for several networks. 

Consider a reaction network $\mathcal{N} = ( \mathcal{V}, \mathcal{O},  \Lambda  )$ with $d$ species and $K$ reactions of the form \eqref{formofthereaction} (see Section \ref{sec:prelim}). For now we assume that each propensity function $\lambda_k$ has the mass-action form \eqref{massactionkinetics} with some rate constant $\theta_k >0$. In the deterministic setting, the state of the network at time $t$ is a vector of species concentrations $x(t) =(x_1(t),\dots,x_d(t)) \in \R^d_+$, which evolves according to the following ODE
\begin{align}
\label{defn_rre}
\frac{dx}{dt} = \sum_{k=1}^K \theta_k  \left( \prod_{i=1}^d x^{\nu_{ik}}_i \right) ( \rho_k - \nu_k ), 
\end{align}
where $\nu_k = (\nu_{1k} ,\dots, \nu_{dk} ) $ and $\rho_k= (\rho_{1k} ,\dots, \rho_{dk} )$ are vectors in $\N^d_0$ denoting the $k$-th column of matrices $ \mathcal{V}$ and $ \mathcal{O}$ respectively. Observe that the $k$-th reaction in the network can simply be represented as $\nu_k \longrightarrow \rho_k$. With this association, $\nu_k$ and $\rho_k$ are called network \emph{complexes}. 
Let $\mathbb{C} = \{ \nu_k,\rho_k : k=1,\dots,K \}$ be the set of all network complexes and suppose there exists a strictly positive vector $r = (r_1,\dots,r_d) \in \R^d_+ $ such that for each $z \in \mathbb{C}$ we have
\begin{align*}
 \sum_{k =1, \nu_k =z}^K  \theta_k  \left( \prod_{i=1}^d r^{\nu_{ik}}_i \right) =  \sum_{k =1, \rho_k =z}^K  \theta_k  \left( \prod_{i=1}^d r^{\nu_{ik}}_i \right). 
\end{align*}
This relation simply says that when the vector of species concentrations is $r$, the rate at which a complex $z$ is consumed (l.h.s.) is same as the rate at which this complex is produced (r.h.s.). It can be shown that $r$ is a fixed point for the deterministic state-dynamics \eqref{defn_rre}, i.e. r.h.s of \eqref{defn_rre} is ${\bf 0}$ when $x =r$. Hence the network is called \emph{complex-balanced} and $r$ is called a \emph{complex-balanced} fixed point. This property has several implications regarding the topology of the network as well as the existence, uniqueness and stability of its fixed points (see \cite{HornJackson,Feinberg1,Feinberg2}). The existence of a complex-balanced equilibrium can be verified computationally, but it can also be checked for many networks using the famous \emph{deficiency zero theorem} (see \cite{Feinberg2}) from Chemical Reaction Network Theory.

In \cite{ACKProd}, Anderson et al. prove that if a network $\mathcal{N}$ with mass-action propensities is complex-balanced in the sense described above, then there exists a \emph{product-form} stationary distribution $\pi$ for its stochastic model, on each irreducible state-space $\mathcal{E}$ within the state-space $\mathcal{E}_0 \subset \N^d_0$ of the network. This stationary distribution is given by                   
\begin{align}
\label{stdis:prodform}
\pi(x) =  \frac{1}{M}  \prod_{i=1}^d \frac{ r^{x_i}_i   }{x_i!}  \quad \textnormal{ for all } \quad x  \in \mathcal{E},
\end{align}
where $a!$ denotes the factorial of $a$, $r$ is the complex-balanced fixed point and $M$ is the normalizing constant given by
\begin{align}
\label{defn_norm_constant}
M  =  \sum_{x  \in \mathcal{E}} \prod_{i=1}^d \frac{ r^{x_i}_i }{x_i!} .
\end{align}
To use this elegant result one needs to know the irreducible state-spaces, which is precisely what is accomplished in this paper. In fact, the examples considered before suggest that for many networks we can provably find all the irreducible state-spaces $\mathcal{E}_1,\dots,\mathcal{E}_Q$ for the network using our main result Theorem \ref{maintheorem}. If this is true, then we can easily compute the corresponding product-from \eqref{stdis:prodform} stationary distributions supported on these classes and obtain the exact simplex $\Sigma$ (see \eqref{stat_simplex}) of stationary distributions for the network.

Note that if the whole nonnegative integer orthant $ \mathcal{E} = \N^d_0$ is an irreducible state-space, then this is the only irreducible state-space, and the simplex $\Sigma$ consists of just one stationary distribution $\pi$ given by \eqref{stdis:prodform} with $\mathcal{E} = \N^d_0$ and $M = \exp( \sum_{i=1}^d r_i )$. In other words, $\pi$ is just a product of Poisson distributions and in such a scenario, the species copy-numbers are independent at stationarity, with the copy-number distribution of species $i$ being Poisson with mean $r_i$.  As the species are constantly interacting through reactions, having this independence is quite remarkable, and it has been argued that this independence could play an important role in metabolic pathways \cite{LevineHwa}.

We now demonstrate how for certain networks, our results can help in accurately computing the stationary networks by replacing a possibly infinite sum in \eqref{defn_norm_constant} with a finite sum, thereby avoiding any truncation errors associated with the problem of estimating infinite sums. Assume that network $\mathcal{N}$ is complex-balanced and it does not have any \emph{restricted} species. Also suppose that using Theorem \ref{maintheorem}, we obtain the irreducible state-space $\mathcal{E} = \mathcal{E}_b \times \N^{ d_1}_0 \times \{0\}$, under some permutation $\sigma$ which we can assume to be identity (i.e. $\sigma(i) = i$ for each $i$), without any loss of generality. Here $\mathcal{E}_b$ is a finite set in $\N^{d_b}_0$, $d_b$ is the number of \emph{bounded} species and $d_1$ is some number less than the number of \emph{free} species $d_f = (d - d_b)$. The decomposed form of $\mathcal{E}$ shows that any element $x  \in \mathcal{E}$ can be expressed as $x = (x_b,x_f,{\bf 0})$, where $x_b = (x_{1} ,\dots, x_{d_b} ) \in \mathcal{E}_b$ and $x_f = (x_{d_b+1} ,\dots, x_{d_b+ d_1} ) \in \N^{ d_1}_0$. Therefore we can write \eqref{defn_norm_constant} as     
\begin{align*}
M  =  \sum_{ (x_b,x_f, {\bf 0})  \in \mathcal{E} }   \prod_{i=1}^d \frac{ r^{x_i}_i }{x_i!} &= \left( \sum_{ x_b  \in \mathcal{E}  }  \prod_{i=1}^{d_b}  \frac{ r_i^{x_i} }{ x_i! } \right) \left( \sum_{ x_f \in \N^{d_1}_0}  \prod_{i=d_b+1}^{d_b+d_1}  \frac{ r_i^{x_i} }{ x_i! }  \right) \\&=  \exp \left( \sum_{i=d_b+1}^{d_b+d_1}  r_{i} \right)   \left( \sum_{ x_b  \in \mathcal{E}  }  \prod_{i=1}^{d_b}  \frac{ r_i^{x_i} }{ x_i! } \right), 
\end{align*}
which is a finite sum, that can be easily computed without incurring any truncation errors. In deriving the last relation, we have used that $\sum_{n=0}^\infty a^n/n! = \exp(a)$ for any $a \in \R$.

Until now we were assuming that propensity functions of the complex-balanced network $\mathcal{N}$ satisfy mass-action kinetics. Theorem 6.1 in \cite{ACKProd} relaxes this assumption of mass-action kinetics and proves the existence of product-form stationary distributions for more general kinetics, where each species $i$ has a ``rate of association" function $\kappa_i : \N_0 \to \R_+$. In this setting, the mass-action formula \eqref{massactionkinetics} for the propensity function changes to
\begin{align*}
\lambda_k(x) = \theta_k \prod_{i=1}^d  \prod_{j=0}^{ \nu_{ik} -1 } \kappa_i(x_i - j)
\end{align*}
and the product-form stationary distribution becomes
\begin{align*}
\pi(x) =  \frac{1}{M} \prod_{i=1}^d    \frac{ r_i^{x_i} }{ \prod_{j=1}^{x_i} \kappa_i(j)  } \quad \textnormal{ for all } \quad x  \in \mathcal{E},
\end{align*}
with the normalizing constant $M$ chosen to ensure that $\sum_{x \in \mathcal{E}} \pi(x) = 1$. Observe that this result can be used to compute the stationary distribution \eqref{sddistreducednetworkexampletoxi} for network \eqref{defn:tox_antitox2} in Section \ref{ex:toxin}. Recently this result has been extended even further to consider more general species-specific association rate functions (see \cite{ACProd2}). Such results along with our computational framework for identifying irreducible state-spaces, provide a way for characterizing the stationary distributions for complex-balanced networks. 

\section{Conclusion} \label{sec:conclusion}

 The aim of this paper is to provide a new tool for analyzing continuous-time Markov chain (CTMC) models of biomolecular reaction networks. Specifically we are interested in situations where the state-space for the CTMC needs to be countably infinite due to the lack of a global conservation relationship among all the species. Such situations arise frequently in Systems Biology as stochastic models generally describe the activity of a small subnetwork embedded within a larger network. We develop a simple procedure to systematically explore the space of conservation relations among species and represent the state-space of the CTMC in a special decomposed-form based on the copy-number ranges of all the species. This form can help in assessing the reachability relations and the communication structures within the infinite state-space of the underlying CTMC. In this context, the main goal of this paper is to construct a computational method for finding all the \emph{closed communication classes} for the CTMC. Such classes are natural attracting sets for the dynamics and they can also be viewed as \emph{irreducible state-spaces} for the CTMC (see Section \ref{subsec:reachability}). Under the existence of a suitable Foster-Lyapunov function (see \cite{GuptaPLOS,Meyn} and Section \ref{sec:intro}), each irreducible state-space supports a unique stationary distribution and these distributions form the vertices of the full simplex (see \eqref{stat_simplex}) of stationary distributions of the CTMC.    
 
As we discuss in this paper, finding irreducible state-spaces for networks where infinitely many states are accessible, is a challenging but a very important problem for several reasons. These reasons include understanding the stability and ergodic properties of networks (see Section \ref{sec:intro}), analyzing network topologies using methods from \emph{Transition Path Theory} \cite{Metzner} and in obtaining the exact stationary distribution for a large class of networks (see Section \ref{ex:productdistr}). Furthermore our computational procedure for finding irreducible states-spaces can assist in an automated discovery of \emph{fast} ergodic subnetworks for the application of \emph{quasi-stationary assumption} (QSA) in the simulation of multiscale reaction networks (see Section \ref{ex:toxin} and \cite{ssSSA,weinan1,weinan2}). Using our procedure we demonstrate in Section \ref{sec:examples} that generally networks from Systems Biology admit a single irreducible state-space which corresponds to the situation where the CTMC describing the reaction dynamics is ergodic and has a unique stationary distribution. We discuss how this information along with the structure of irreducible state-space can sometimes provide valuable biological insights into the design and functionality of the underlying network (see Sections \ref{ex:cc} and \ref{ex:hsr}).

Our approach works by classifying each species as \emph{free}, \emph{bounded} and \emph{restricted}, based on their admissible copy-number ranges. The \emph{bounded} species have a finite copy-number range and their dynamics evolves in a finite set $\mathcal{E}_b$. On the other hand the copy-number range of \emph{free} species is the infinite set of all non-negative integers $\N_0$, and hence their dynamics evolves in an orthant $\N^{d_f}_0$. The \emph{restricted} species imitate the dynamics of \emph{free} species according to some affine function, and they can be removed from the network for the purpose of finding irreducible state-spaces (see Section \ref{sec:netreductionbyelimination}). We develop a computational procedure that can provably locate all the irreducible state-spaces for the underlying CTMC within the infinite state-space $\mathcal{E}_b \times \N^{d_f}_0$. This is accomplished by suitably combining the matrix methods used in the finite state-space case where only \emph{bounded} species are present (see Section \ref{sec:netbounded}), along with the construction of \emph{birth-death cascades} for the \emph{free} species (see Section \ref{sec:freebddspecies}). We demonstrate the versatility of our method through many examples from Systems Biology in Section \ref{sec:examples}. From these examples one can also observe that the birth-death cascades correspond naturally to various important stages in the network and hence this cascade construction process facilitates a better understanding of the network design.

Finally we would like to mention that since our computational procedure only involves basic linear-algebraic tasks (such as matrix computations, solving linear equations and Linear Programs etc.), it can be efficiently applied in very high dimensions. Hence our method can easily handle large reaction networks with several species and reactions. However computational issues may arise if the size $N_b$ of the finite state-space $\mathcal{E}_b$ for \emph{bounded} species becomes too large, as our method requires several computations with matrices of size $N_b \times N_b$. We hope to resolve these issues in a future work.

\section*{Appendix: Proof of Theorem \ref{maintheorem} \label{sec:APPENDIX}}

Let $\mathbb{L}_{ \textnormal{min} }$ be the set of al minimal leaf nodes of the Birth-Cascade Tree (BCT) constructed in Section \ref{sec:freebddspecies}. We pick any $(C,A) \in \mathbb{L}_{ \textnormal{min} }$ and consider the dynamics of network $\tilde{\mathcal{N} }^{\sigma}$ under the permutation $\sigma = \sigma_3(A)$ mentioned in the statement of Theorem \ref{maintheorem}. Consider the directed path \eqref{defn_bct} in BCT, starting with the node $(C_0, \emptyset)$ and terminating at the leaf node $(C,A)$. The next lemma is a simple consequence of the construction of BCT. 
\begin{lemma}
\label{lem:level3} 
Pick any $z_1 \in C_0$, $z_2 \in C$ and $r_0 \in \N^{|A|}_0$. Then there exists a vector $x \in \N^{|A | }_0$ such that $x \geq r_0$ and  
\begin{align*}
 (z_1, {\bf 0} ) \stackrel{ \tilde{\mathcal{N} }^{\sigma}  }{\longrightarrow} ( z_2, x, {\bf 0}).  
 \end{align*}
 Moreover this relation also holds for any $z_1 \in C$. 
 \end{lemma}
 \begin{proof}
Throughout this proof we denote the relation $\stackrel{ \tilde{\mathcal{N} }^{\sigma}  }{\longrightarrow} $ by $\longrightarrow$. Let $C_i$-s and $G_i$-s be as in the path \eqref{defn_bct} consisting of $(l+1)$ nodes. Let $d_i = |G_i|$ for each $i$ and since $G_i$-s in an increasing family of sets we must have
\begin{align*}
0 = d_0 < d_1 < d_2 <\dots < d_l = |A|.
\end{align*}
By permuting the network if necessary, we can assume that $G_i = \{1,\dots, d_i\}$ for each $i$.

Clearly the case $l=0$ is trivial because in this case $C_0 =  C$ and $A = \emptyset$. We now consider the case $l=1$ and $A = G_1$. Due to reactions in the set $\mathcal{K}_r( \emptyset )$, the dynamics of \emph{bounded} species can reach any state in $C_0$ from any other state in $C_0$. The construction of BCT ensures that for each \emph{free} species $\mathbf{S}_i$ in $\mathcal{D}^{ \sigma }_f(A)$, there exists a state $y \in C_0$ and a reaction $\mathcal{K}_r(y, \emptyset) \subset \mathcal{K}_r( \emptyset )$, which produces this \emph{free} species. By repeated firings of this reaction we can push the molecular-count of species ${\bf S}_i$ beyond any positive integer. Note that reactions in $\mathcal{K}_r(\emptyset)$ do not consume any \emph{free} species. Hence we can perform this procedure independently for all the \emph{free} species in $\mathcal{D}^{ \sigma }_f(A)$, and prove that for any $r \in \N^{d_1}_0$ there exists a vector $x' \in \N^{d_1}_0$ along with some state $y \in C_0$ such that $x' \geq r$ and
\begin{align*}
 (z_1, {\bf 0} )  \longrightarrow (y, x', {\bf 0}).
\end{align*}
However since $C \in \Psi_{ ( \emptyset, A_1 ) }( C_0 )$, after finitely many reactions in $\mathcal{K}_r(A)$ that only consume the \emph{free} species in $\mathcal{D}^{ \sigma }_f(A)$, the state of the \emph{bounded} species can go from $y$ to $z_2 \in C$, thereby ensuring that
\begin{align*}
(y, x', {\bf 0})  \longrightarrow (z_2,x,{\bf 0})
\end{align*}
for some $x \in \N^{d_1}_0$ with $x \leq x'$. By choosing $r$ with large enough entries we can ensure that $x \geq r_0$. The assertion of this lemma (i.e. $ (z_1, {\bf 0} ) \longrightarrow (z_2,x,{\bf 0}) $) then follows for $l=1$ from the transitivity of relation $\longrightarrow$ (see Section \ref{subsec:reachability}).

For a general $l >1$, one can repeat the above arguments at each stage $i= 1,2,\dots,(l-1)$ to prove the lemma's assertion. At each stage we rely on the fact that $C_i$ is a closed communication class under reactions in $\mathcal{K}_r(G_i)$ and these reactions do not consume any \emph{free} species outside the set $\mathcal{D}_f^{\sigma}(G_i)$.

Now suppose $z_1 \in C$. After finitely many reactions in the set $\mathcal{K}_r( \emptyset ) $, the state of \emph{bounded} species will reach a state $z'_1$ in some closed communication class $C'_0 \in \mathcal{C}( \emptyset )$. As $z_1 \in C$ and $\mathcal{K}_r( \emptyset ) \subset \mathcal{K}_r( A )$ we must have $C'_0 \subset C$. There exists a BCT-path of the form
\begin{align*}
(C'_0, \emptyset) = (C'_0,G'_0) \Rightarrow  (C'_1,G'_1) \Rightarrow \dots \Rightarrow  (C'_m,G'_m) = (C',A'),  
\end{align*}  
culminating in the leaf node $(C',A') \in \mathbb{L}$. Note that as $C'_0 \subset C$ we have $\mathcal{K}_r(C'_0, \emptyset ) \subset \mathcal{K}_r(C, A )$ and hence $C'_1 \subset C$ and $G'_1 \subset A$. Repeating this argument $(m-1)$ times we can conclude that $C'_m = C' \subset C$ and $G'_m = A' \subset A$. However as $(C,A)$ is a minimal leaf node we must have $C'=C$ and $A' =  A$. The result now follows from the assertion already proven above.
\end{proof}

For the minimal leaf node $(C,A) \in \mathbb{L}_{ \textnormal{min} }$, let $\tilde{\mathcal{N} }^{\sigma}(C,A)$ be the network formed by restricting the \emph{free} species to the set $\mathcal{D}^{\sigma}_f(A)$, the reactions to the set $\mathcal{K}_r( C, A  )$ and the \emph{bounded} species state-space to the set $C$. Let $\mathbb{L}_d(C,A)$ denote the set of all leaf nodes for the Death-Cascade Tree (DCT) for network $\tilde{\mathcal{N} }^{\sigma}(C,A)$. The following lemma pertains to the situation when the leaf node $(C,A)$ is death-exhaustive. 
\begin{lemma}
\label{lem:deathexhaustive}
Suppose that the minimal leaf node $(C,A) \in \mathbb{L}_{ \textnormal{min} }$ is death-exhaustive, i.e. there exists a node $(C',A') \in \mathbb{L}_d(C,A)$ such that $A =A'$. Then $C' = C$ and for any $z_1,z_2 \in C$ and $r_0 \in \N^{|A|}_0$ there exists a vector $x \in \N^{|A | }_0$ such that $x \geq r_0$ and  
\begin{align*}
( z_1, x, {\bf 0})  \stackrel{ \tilde{\mathcal{N} }^{\sigma}  }{\longrightarrow}   (z_2, {\bf 0} ).  
 \end{align*} 
\end{lemma}
\begin{proof}
Let $\tilde{\mathcal{N} }^{\sigma}_{ \textnormal{inv} }(C,A)$ be the inverse of network $\tilde{\mathcal{N} }^{\sigma}(C,A)$. Note that since $(C,A)$ is a leaf node for BCT, relation \eqref{leafnodecond} holds and hence the reactions in $\mathcal{K}_r( C, A  )$ do not produce any \emph{free} species outside the set $\mathcal{D}_f^{\sigma}(A)$. This implies that in the inverse network $\tilde{\mathcal{N} }^{\sigma}_{ \textnormal{inv} }(C,A)$ no reaction in $\mathcal{K}_r( C, A  )$ can consume any \emph{free} species outside the set $\mathcal{D}_f^{\sigma}(A)$. Hence all the reactions in $\mathcal{K}_r( C, A  )$ can fire in this inverse network as it satisfies Assumption \ref{massacttype}. Since communication is a symmetric relation and $C' \subset C$ is a closed communication class, $A' = A$ implies that $C'= C$. By definition, the DCT for network $\tilde{\mathcal{N} }^{\sigma}(C,A)$ is the BCT of network $\tilde{\mathcal{N} }^{\sigma}_{ \textnormal{inv} }(C,A)$. Consider the directed path \eqref{defn_bct} in this BCT, starting with the node $(C_0, \emptyset)$ and terminating at the leaf node $(C,A)$. Pick any $z_1,z_2 \in C$ and $r_0 \in \N^{|A|}_0$. Applying Lemma \ref{lem:level3} we can conclude that there exists a vector $x \in \N^{|A | }_0$ such that $x \geq r_0$ and  
\begin{align*}
  (z_2, {\bf 0} )   \stackrel{ \tilde{\mathcal{N} }^{\sigma}_{ \textnormal{inv} }(C,A) }{\longrightarrow}  ( z_1, x, {\bf 0}),  
 \end{align*}  
which also implies
\begin{align*}
( z_1, x, {\bf 0})  \stackrel{ \tilde{\mathcal{N} }^{\sigma}  }{\longrightarrow}   (z_2, {\bf 0} ) ,  
 \end{align*}
 due to relation \eqref{reln:netinverse} and the fact that network $\tilde{\mathcal{N} }^{\sigma}(C,A)$ is simply a restriction of the network $ \tilde{\mathcal{N} }^{\sigma} $. This completes the proof of this lemma.  
\end{proof}

\begin{lemma}
\label{lem:level3d}
Let $(C,A) \in \mathbb{L}_{\textnormal{min}}$ be a minimal leaf node such that all the free species in $\mathcal{D}^{\sigma}_f(A)$ are singularly-degradable w.r.t. $A$. Then there exists a $r_0 \in \N^{|A| }_0$ such that for any $z \in C$ and $x_1,x_2 \in \N^{|A|}_0$ satisfying $x_1 \geq x_2 \geq r_0$ we have
\begin{align*}
( z, x_1, {\bf 0}) \stackrel{ \tilde{\mathcal{N} }^{\sigma}  }{\longrightarrow}  (z, x_2, {\bf 0} ).
 \end{align*} 
 \end{lemma}
\begin{proof}
Throughout this proof we denote the relation $\stackrel{ \tilde{\mathcal{N} }^{\sigma}  }{\longrightarrow} $ by $\longrightarrow$. Without loss of generality we can assume that $A = \{1,2,\dots,|A|\}$. The proof of this lemma is inspired by the proof of Theorem 3.4 in \cite{Craciun}. Fixing a $z \in C$, we first show that for each $i \in A$ there exists a $r_i \in \N^{ |A| }_0$ such that 
\begin{align}
\label{key_lem:level3_1}
(z,r_i, {\bf 0} ) \longrightarrow (z,r_i -e_i,{\bf 0} ),
\end{align}
where $e_i$ is the $i$-th standard basis vector in $\R^{|A|}$.

Pick any $i \in A$ and let $\tilde{\mathcal{K}}_p(A) $ be the set of reactions given by \eqref{defn_noboundedfreecons}. Since each species in $\mathcal{D}^{\sigma}_f(A)$ is singularly-degradable, there exists a sequence of reactions $k_1,\dots,k_n \in \tilde{\mathcal{K}}_p(A)$ such that 
\begin{align}
\label{stoichiometrixreln1}
(-e_i, {\bf 0} )  = \sum_{j=1}^n ( \tilde{\rho}^{ \sigma }_{k_j} -  \tilde{\nu}^{ \sigma }_{k_j}   ),
\end{align}
where $ \tilde{\nu}^{ \sigma }_{k}$ and $ \tilde{\rho}^{ \sigma }_{k}$ are as in Section \ref{sec:freebddspecies}. Note that since for each reaction $k_j$ we have $\textnormal{supp}( \tilde{\rho}^{ \sigma }_{k_j} ) \subset A$, we must also have $\textnormal{supp}( \tilde{\nu}^{ \sigma }_{k_j} ) \subset A$, or otherwise the last $(d_f -|A|)$ components in the r.h.s. of \eqref{stoichiometrixreln1} cannot be ${\bf 0}$. Therefore each of these reactions can only consume the \emph{free} species in $\mathcal{D}^{\sigma}_f(A)$. For each $m=1,\dots, n$ let   
$$y_m = \sum_{j=1}^{m-1} ( \tilde{\rho}^{ \sigma }_{k_j} -  \tilde{\nu}^{ \sigma }_{k_j}   ).$$
By choosing a $r_i \in \N_0^{|A|}$ with large enough entries we can ensure that $(r_i,{\bf 0}) + y_m \geq \tilde{\nu}^{ \sigma }_{k_m}$ for each $m$. Since none of these reactions involve the \emph{bounded} species and Assumption \ref{massacttype} is satisfied, such a choice of $r_i$ also ensures that each reaction $k_m$ has a positive probability of firing when the state of the \emph{free} species is $(r_i,{\bf 0}) + y_m $, thereby implying that \eqref{key_lem:level3_1} is satisfied.

We find such a $r_i$ for each $i \in A$ and compute the maximum $r_0:= \max_{i \in A} \{r_i\}$ of these vectors in the componentwise sense. Using Proposition \ref{prop_reach_reln_add} we can conclude that for any $x \in \N^{|A|}_0$ satisfying $x \geq r_0$ we have  
\begin{align}
\label{key_lem:level3_2}
(z,x,{\bf 0}) \longrightarrow (z,x -e_i ,{\bf 0} )
\end{align}
for each $i \in A$. Now select any $x_1,x_2 \in \N_0^{|A| }$ satisfying $x_1 \geq x_2 \geq r_0$ and let $\alpha = (x_1 -x_2 ) \in \N^{|A|}_0$. We can express $\alpha$ as the sum 
$$\alpha = \sum_{i=1}^{ |A| } \alpha_i e_i.$$
Exploiting the transitivity of relation $\longrightarrow$ and using \eqref{key_lem:level3_2}, $\alpha_i$ times for each $i$ we obtain the accessibility chain $ (z,x_2+\alpha, {\bf 0}) \longrightarrow (z,x_2 + \alpha- \alpha_1e_1, {\bf 0}) \longrightarrow  (z,x_2 + \alpha- \alpha_1e_1 - \alpha_2 e_2, {\bf 0}) \longrightarrow  \dots \longrightarrow (z,x_2 + \alpha- \sum_{i=1}^{|A| }\alpha_ie_i , {\bf 0}).$ But $x_2+\alpha =x_1$ and $(x_2 + \alpha- \sum_{i=1}^{|A|  }\alpha_i e_i) =x_2$ and hence the proof of this lemma is complete.
\end{proof}

Next we present a simple modification of the above lemma in the case where all the \emph{free} species are singularly-degradable w.r.t. $A$.
\begin{lemma}
\label{lem:level3d1}
Let $(C,A) \in \mathbb{L}_{ \textnormal{min} }$ be a minimal leaf node such that all the free species in $\mathcal{D}_f$ are singularly-degradable w.r.t. $A$. Then there exists a $r_0 \in \N^{|A| }_0$ such that for any $z \in C$, $ y \in \N^{d_f - |A|}_0$ and any $x  \in \N^{|A|}_0$ satisfying $x \geq r_0$ we have
\begin{align*}
( z, x, y ) \stackrel{ \tilde{\mathcal{N} }^{\sigma}  }{\longrightarrow}  (z, x, {\bf 0} ).
 \end{align*} 
 \end{lemma}      
 \begin{proof}
 The proof is similar to the proof of Lemma \ref{lem:level3d}. As before we denote the relation $\stackrel{ \tilde{\mathcal{N} }^{\sigma}  }{\longrightarrow} $ by $\longrightarrow$ and assume that $A = \{1,2,\dots,|A|\}$. Fix a $z \in C$. We first show that for each $i \in \{ |A| + 1,\dots, d_f \}$ there exists a $r_i \in \N^{ |A| }_0$ such that 
\begin{align}
\label{key_lem:level3_11}
(z,r_i, e_i) \longrightarrow (z,r_i,{\bf 0} ),
\end{align}
where $e_i$ is the $i$-th standard basis vector in $\R^{d_f - |A|}$.

Fix any $i \in  \{ |A| + 1,\dots, d_f \}$ and since all the \emph{free} species in $\mathcal{D}_f$ are singularly-degradable w.r.t. $A$, there exists a sequence of reactions $k_1,\dots,k_n \in \tilde{\mathcal{K}}_p(A)$ such that 
\begin{align*}
({\bf 0} ,-e_i )  = \sum_{j=1}^n ( \tilde{\rho}^{ \sigma }_{k_j} -  \tilde{\nu}^{ \sigma }_{k_j}   ).
\end{align*}
Note that for each reaction $k_j$ we have $\textnormal{supp}( \tilde{\rho}^{ \sigma }_{k_j} ) \subset A$. Hence there must exist a $n_0 \in \{1,\dots,n\}$ such that $\textnormal{supp}( \tilde{\nu}^{ \sigma }_{k_j} ) \subset A$ for all $j \neq n_0$ and for reaction $k_{n_0}$ the following must be satisfied:
\begin{align*}
\textnormal{supp}( \tilde{\nu}^{ \sigma }_{k_{n_0}} ) \subset A \cup \{i\} \quad \textnormal{and} \quad  \tilde{\nu}^{ \sigma }_{i k_{n_0}} = 1.
\end{align*}
For each $m=1,\dots, n$ let   
$$y_m = \sum_{j=1}^{m-1} ( \tilde{\rho}^{ \sigma }_{k_j} -  \tilde{\nu}^{ \sigma }_{k_j}   ).$$
By choosing a $r_i \in \N_0^{|A|}$ with large enough entries we can ensure that $(r_i,e_i) + y_m \geq \tilde{\nu}^{ \sigma }_{k_m}$ for each $m$. Since none of these reactions involve the \emph{bounded} species and Assumption \ref{massacttype} is satisfied, such a choice of $r_i$ also ensures that each reaction $k_m$ has a positive probability of firing when the state of the \emph{free} species is $(r_i,e_i) + y_m$, thereby implying that \eqref{key_lem:level3_11} is satisfied.

We find such a $r_i$ for each $i \in  \{ |A| + 1,\dots, d_f \}$ and compute the maximum $r_0:= \max\{r_i\}$ in the componentwise sense. Using Proposition \ref{prop_reach_reln_add} we can conclude that for any $x \in \N^{|A|}_0$ satisfying $x \geq r_0$ and any $ y = (y_1,\dots, y_{d_f- |A|}) \in \N^{d_f - |A|}_0$ we have  
\begin{align}
\label{key_lem:level3_21}
(z,x,y) \longrightarrow (z,x , y -e_i )
\end{align}
as long as $y_i \geq 1$. Exploiting the transitivity of relation $\longrightarrow$ and using \eqref{key_lem:level3_21}, $y_i$ times for each $i$ we obtain the accessibility chain $ (z,x,y) \longrightarrow (z,x, y - y_1e_1) \longrightarrow  (z,x, y - y_1e_1 - y_2 e_2) \longrightarrow  \dots \longrightarrow (z,x, y - y_1e_1 - y_2 e_2 - \dots - y_{d_f - |A|} e_{d_f - |A|} )$. But $y - y_1e_1 - y_2 e_2 - \dots - y_{d_f - |A|} e_{d_f - |A|} = {\bf 0}$ and hence the proof of this lemma is complete.
\end{proof}                                

 We are now ready to prove Theorem \ref{maintheorem}. We shall use the fact that irreducible state-spaces for network $\tilde{\mathcal{N}}^{ \sigma }$ must be necessarily disjoint (see Section \ref{subsec:reachability}).
  
 \begin{proof}[Proof of Theorem \ref{maintheorem}]
Throughout the proof we denote the relation $\stackrel{ \tilde{\mathcal{N} }^{\sigma}  }{\longrightarrow} $ by $\longrightarrow$. We pick any minimal leaf node $(C,A) \in \mathbb{L}_{ \textnormal{min} }$ and we suppose that this node is death-exhaustive and all the free species in $\mathcal{D}^{ \sigma }_f(A)$ are singularly-degradable w.r.t. $A$. We now show that $C \times \N_0^{ |A| } \times \{ \mathbf{0} \}$ is an irreducible state-space for network $\tilde{\mathcal{N} }^{\sigma}$ permuted according to the permutation $\sigma = \sigma_3(A)$ defined in Section \ref{sec:freebddspecies}. To prove this assertion it suffices to show that for any $z_1,z_2 \in C$ and $x \in \N^{|A|}_0$ we have
 \begin{align}
\label{meainthmrel1}
 (z_1, {\bf 0}) \longrightarrow ( z_2, x, { \bf 0 })   \longrightarrow (z_1, {\bf 0}).
\end{align}
This is because if \eqref{meainthmrel1} holds then for any $(z_1,x_1,{ \bf 0 }),(z_2,x_2,{ \bf 0 }) \in C \times \N_0^{ |A| } \times \{ \mathbf{0} \}$ we have $(z_1,x_1,{ \bf 0 })  \longrightarrow  (z_2,x_2,{ \bf 0 })$ due to the following chain of accessibility relations $(z_1,x_1,{ \bf 0 })  \longrightarrow   (z_1, { \bf 0 })  \longrightarrow  ( z_2, x_2,{ \bf 0 }) $ and the transitivity of $\longrightarrow $.

We now prove \eqref{meainthmrel1} for a fixed $z_1,z_2 \in C$ and $x \in \N^{|A|}_0$. Since the leaf node $(C,A)$ is death-exhaustive, Lemma \ref{lem:deathexhaustive} implies that for any $r_0 \in \N^{|A|}_0$ there exists a vector $x' \in \N^{|A|}_0$ satisfying $x' \geq r_0$ and 
\begin{align}
\label{proofmainthema1}
( z_2, x',{\bf 0}) \longrightarrow (z_2, {\bf 0}).
\end{align}
We assume that $r_0$ is as in Lemma \ref{lem:level3d}. Due to Proposition \ref{prop_reach_reln_add}, relation \eqref{proofmainthema1} implies $(z_2,x+x',{\bf 0}) \longrightarrow ( z_2,  x, {\bf 0} )$. Using Lemma \ref{lem:level3} we can find a vector $y \in \N^{|A|}_0$ satisfying $y \geq (x+x')$ and $ (z_1,  {\bf 0}) \longrightarrow ( z_2, y, {\bf 0})$. Since $(x+x') \geq r_0$, Lemma \ref{lem:level3d} implies that $( z_2, y,{\bf 0}) \longrightarrow (z_2,x+x',{\bf 0})$. We now have an accessibility chain $(z_1, ,{\bf 0}) \longrightarrow ( z_2, y,{\bf 0}) \longrightarrow (z_2,x+x',{\bf 0})  \longrightarrow ( z_2,  x,{\bf 0}),$
and since $\longrightarrow $ is transitive we have the first relation in \eqref{meainthmrel1}.

We now show the converse (i.e. $( z_2, x, { \bf 0 })   \longrightarrow (z_1, {\bf 0})$). Using Lemma \ref{lem:level3} and \ref{lem:deathexhaustive}, we can find vectors $x_1,x_2 \in \N^{|A|}_0$ such that $x_1 \geq x_2  \geq {\bf 0}$, $ (z_2, x_2,{\bf 0}) \longrightarrow ( z_1, {\bf 0}  ) $ and $( z_2,{\bf 0} )  \longrightarrow ( z_2, x_1,{\bf 0})$. The last relation also implies $( z_2, x,{\bf 0} )  \longrightarrow ( z_2, x_1+x,{\bf 0})$ due to Proposition \ref{prop_reach_reln_add}. Since $(x_1+x) \geq x_2 \geq r_0$ we have $  ( z_2, x_1+x,{\bf 0}) \longrightarrow ( z_2, x_2 ,{\bf 0}) $ due to Lemma \ref{lem:level3d}. This gives us the following chain of accessibility relations $( z_2, x,{\bf 0}) \longrightarrow ( z_2, x_1+x,{\bf 0}) \longrightarrow ( z_2, x_2 ,{\bf 0}) \longrightarrow ( z_1,{\bf 0} )$ which shows the second relation of \eqref{meainthmrel1} and proves that $C \times \N_0^{ |A| } \times \{ \mathbf{0} \}$ is an irreducible state-space for network $\tilde{\mathcal{N} }^{\sigma}$.

Assuming that all the free species in the set $\mathcal{D}_f$ are singularly-degradable w.r.t. $A$, we now prove that this is the network's only irreducible state-space that can contain elements in the set $C \times \N^{d_f}_0$ under the permutation $\sigma = \sigma_3(A)$. We prove this claim by contradiction by showing that if $\mathcal{E} \subset \mathcal{E}^{ \sigma }_b  \times \N^{d_f}_0$ is any irreducible state-space for network $\tilde{ \mathcal{N} }^{\sigma}$ satisfying $\mathcal{E} \cap ( C \times \N^{d_f}_0) \neq \emptyset $ then it cannot be disjoint from $C \times \N_0^{ |A| } \times \{ \mathbf{0} \}$. Let $(z,x)$ be any state in $\mathcal{E}$, and as this set is irreducible, any state that can be accessed by $(z,x)$ with relation $\longrightarrow$, must also be in $\mathcal{E}$. We can write $x \in \N^{d_f}_0$ as $x = (x_1,x_2)$ where $x_1 \in \N^{|A|}_0$ and $x_2 \in \N^{d_f - |A|}_0$. Let $r_0 \in \N^{|A|}_0$ be as in Lemma \ref{lem:level3d1}. Using Lemma \ref{lem:level3} and Proposition \ref{prop_reach_reln_add} if necessary, we can assume that $(z,x_1,x_2) \longrightarrow (z,x'_1,x_2)$ for some $x'_1 \geq r_0$. Lemma \ref{lem:level3d1} then implies that $(z,x'_1,x_2) \longrightarrow (z,x'_1,{\bf 0})$, which also means that $(z,x'_1,{\bf 0}) \in \mathcal{E}$. But $(z,x'_1,{\bf 0})  \in C \times \N_0^{ |A| } \times \{ \mathbf{0} \}$ and hence $\mathcal{E}$ cannot be disjoint from $C \times \N_0^{ |A| } \times \{ \mathbf{0} \}$, giving a contradiction. This completes the proof of part (A) of Theorem \ref{maintheorem}.

We now prove part (B). Let $\mathcal{E} \subset \mathcal{E}^{ \sigma }_b  \times \N^{d_f}_0$ be a nonempty irreducible state-space for network $\tilde{ \mathcal{N} }^{\sigma}$. Let $(z,x)$ be any state in $\mathcal{E}$. After finitely many reactions in the set $\mathcal{K}_r( \emptyset ) $, the state of \emph{bounded} species will reach a state $z'$ in some closed communication class $C'_0 \in \mathcal{C}( \emptyset )$. Hence there exists a $x' \in \N^{d_f}_0$ such that $(z,x) \longrightarrow (z',x') \in \mathcal{E}$. Moreover there exists a BCT-path of the form
\begin{align*}
(C'_0, \emptyset) = (C'_0,G'_0) \Rightarrow  (C'_1,G'_1) \Rightarrow \dots \Rightarrow  (C'_m,G'_m) = (C',A'),  
\end{align*}  
culminating in the leaf node $(C',A') \in \mathbb{L}$. Since $\mathbb{L} = \mathbb{L}_{ \textnormal{min} }$ the leaf node $(C',A')$ is also minimal. We view the dynamics of network $\tilde{\mathcal{N}}^{\sigma}$ under the permutation $\sigma = \sigma_3(A')$ defined in Section \ref{sec:freebddspecies}. Pick any $z_1 \in C'$. Due to Lemma \ref{lem:level3} and Proposition \ref{prop_reach_reln_add} we know that for any $r_0 \in \N^{|A'|}_0$, there exists a $x_1 \in \N^{|A'|}_0$ and $x_2 \in \N_0^{d_f - |A'|}$ such that $(z',x')  \longrightarrow  (z_1,x_1,x_2) \in \mathcal{E}$. However $(z_1,x_1,x_2)$ is an element in $C' \times \N_0^{ |A'| } \times \{ \mathbf{0} \}$ and hence, as argued previously, we arrive at a contradiction unless $\mathcal{E} = C' \times \N_0^{ |A'| } \times \{ \mathbf{0} \}$. This completes the proof of part (B) of Theorem \ref{maintheorem}.
\end{proof}

\bibliographystyle{unsrt}

 \end{document}